\documentclass[11pt]{article}

\usepackage[margin=1.2in]{geometry}
\usepackage{graphicx}%
\usepackage{multirow}%
\usepackage{amsmath}
\usepackage{amssymb,amsmath}
\usepackage{amsthm}
\usepackage[title]{appendix}%
\usepackage{xcolor}%
\usepackage{cite}
\usepackage{mathrsfs}
\usepackage{textcomp}
\usepackage{stmaryrd}
\usepackage{listings}
\usepackage{manyfoot}%
\usepackage{booktabs}%
\usepackage{algorithm}%
\usepackage{algorithmicx}%
\usepackage{algpseudocode}%
\usepackage{fontenc}
\usepackage{verbatim}
\usepackage{authblk}
\usepackage{sectsty}
\usepackage[colorlinks,linkcolor=blue,citecolor=blue]{hyperref}
\theoremstyle{thmstyleone}%

\makeatletter

\newdimen\bibspace
\renewenvironment{thebibliography}[1]{%
 \section*{\refname %or \bibname if you use ``book'' as the documentclass
\@mkboth{\MakeUppercase\refname}{\MakeUppercase\refname}}%    
\list{\@biblabel{\@arabic\c@enumiv}}%
{\settowidth\labelwidth{\@biblabel{#1}}%
     \leftmargin\labelwidth
   \advance\leftmargin\labelsep
        \itemsep\bibspace
           \parsep\z@skip     %
           \@openbib@code
           \usecounter{enumiv}%
           \let\p@enumiv\@empty 
\renewcommand\theenumiv{\@arabic\c@enumiv}}%
\sloppy\clubpenalty4000\widowpenalty4000%
     \sfcode`\.\@m}
    {\def\@noitemerr
      {\@latex@warning{Empty `thebibliography' environment}}%
     \endlist}

\makeatother

\makeatletter

\newtheorem{thm}{Theorem}[section]
\newtheorem{lem}[thm]{Lemma}
\newtheorem{prop}[thm]{Proposition}
\newtheorem{defn}[thm]{Definition}

\newtheorem{cor}[thm]{Corollary}
\newtheorem{rem}[thm]{Remark}

\def\XXint#1#2#3{{\setbox0=\hbox{$#1{#2#3}{\int}$}
  \vcenter{\hbox{$#2#3$}}\kern-.5\wd0}}

           \newcommand{\ud}{\mathrm{d}}
\newcommand{\be}{\begin{equation}}      \newcommand{\ee}{\end{equation}}

\begin{document}

\title {\textbf{Heat Flows with Prescribed Singularities from 3-dimensional Manifold}\bigskip}
\author{\medskip  Jie Ji\footnote{Corresponding author: jie.ji@mail.bnu.edu.cn}\quad and \quad
Jingru Niu\footnote{Contributing author: 202331130026@mail.bnu.edu.cn}}
%\author{First Author\thanks{Contributing author:iauthor@gmail.com}} and Second Author \thanks{Contributing author:iiauthor@gmail.com}}
%\equalcont{These authors contributed equally to this work.}\\

\affil{School of Mathematical Sciences, Beijing Normal University,  Beijing 100875, The People’s Republic of China}
\maketitle

\abstract{In this paper, we study singular heat flows from a 3-dimensional complete bounded Riemannian manifold without boundary into the hyperbolic space with prescribe singularity along a closed curve. We prove the existence and regularity of the singular heat flows. Furthermore, we prove that the singular heat flows converge to a singular harmonic map at an exponential rate.}

%\keywords{keyword1, Keyword2, Keyword3, Keyword4}
%\keywords{harmonic map, heat flow, Prescribed Singularities}

\section{Introduction}
\numberwithin{equation}{section}
There has been many works on the harmonic map and heat flow to non-positive sectional curvature manifolds. Eells and Sampson \cite {ES} established the existence of heat flow and proved its convergence to a harmonic map when both the domain manifold and the target manifold are compact and without boundary. Later, Hamilton \cite{H} presented a more generalized theory for harmonic maps from manifolds with boundary, utilizing the heat flow method. 
%In his works, the facts that the target manifold has non-positive sectional curvature and the total energy is finite and decreasing under the flow are both important. 
Other studies, see \cite{LT1}, \cite{LT2} and \cite{W}. These literatures studied the existence and convergence of solutions for heat flow. Moreover, the relevant theories of heat flow are useful for proving the Schoen-Li-Wang conjecture proposed in \cite{S} and \cite{LW2}. The related results can be found in \cite{LT3, LT4, LW2, M, LM}.

%Liao and Tam \cite{LT2} generalize Eells and Sampson's result to the case that the domain manifold is complete and non-compact under the assumption that initial function has finite total energy. Li and Tam \cite{LT1} proved convergence of the heat flow if the tension field of the initial map is bounded and in $L^p$ for some $p>2$. Wang \cite{W} extended Li and Tam's results that it is enough to assume that the tension field of the initial map tends to zero uniformly near the boundary of domain.

Regarding harmonic maps with prescribed singularities, they were derived from the physical background.
The Einstein vacuum equations in the stationary axially symmetric case reduce to a harmonic map from $\mathbb{R}^3$ into $\mathbb{H}^2$, the hyperbolic plane, with prescribed singularities along the axis of symmetry. These represent rotating black holes in equilibrium. Weinstein \cite{G3} early proved the existence and uniqueness of solutions in the case of zero total angular momentum.  He extensively studied the related problem in \cite{G}, \cite{G1}, and \cite{G2},  established a comprehensive existence and uniqueness theory for harmonic maps from a domain to a target manifold with prescribed singularities. 
 Motivated from one of Hawking's conjectures on the uniqueness of 
Kerr solutions, Li-Tian \cite{LT} examined the boundedness and the regularity of harmonic maps into hyperbolic spaces with prescribed singularities along codimension two closed submanifolds. Recently, Han-Marcus-Weinstein-Xiong \cite{HX} concluded their study of the regularity and asymptotic analysis of harmonic maps arising from stationary vacuum spacetimes, where they derived the geodesic equation in the hyperbolic plane and established the rates of convergence to the tangent map. In a more general case, Fioravanti \cite{GF} deriveed the Schauder estimates for the Dirichlet problem on lower dimensional manifolds, specifically for a class of weighted elliptic equations where the coefficients are singular on such sets. These findings provide a comprehensive regularity theory for harmonic maps into hyperbolic spaces with prescribed singularities. \cite{A} and \cite{LS1} have studied the asymptotic convergence of nonlinear equations. We primarily employ the method outlined in \cite{HX}, wherein the Łojasiewicz-Simon gradient inequality serves as a pivotal tool. This inequality is adept at handling more general asymptotic analyses of gradient-type differential equation solutions. Given our specific energy and negative curvature conditions for the target manifolds, we can adopt simpler methods adapted to this particular scenario.

%Under a crucial analyticity assumption, we will prove that bounded smooth solutions indeed have unique limit points.
In this paper, we consider the prescribed singularity problem which extend the results presented in Li-Tian \cite{LT} to the corresponding parabolic equation. Let $(M,g)$ be a 3-dimensional complete and bounded Riemannian manifold without boundary and $\Gamma \subset M$ be a 1-dimensional closed submanifold. Let $h:M \rightarrow \mathbf{\overline{H}}^2$ be a smooth map, where $\mathbf{\overline{H}}^2$ denotes the upper half-plane model of 2-dimensional hyperbolic space, and assume that $h(M\backslash \Gamma)>0$. Furthermore, let $\rho(x)$ be a function defined on $M$ such that $\rho(x)=dist(x,\Gamma)$, where $dist(x,\Gamma)$ represents the distance from the point $x$ to the submanifold $\Gamma$. 

From Li-Tian \cite{LT}, we know that the solution $u(x)$ of 
\begin{equation*}
	\begin{aligned}
		-\Delta u(x)=\Delta(\log \rho(x)), \ x\in M
	\end{aligned}
\end{equation*}
is smooth away from $\Gamma$ and H\"older continuous across $\Gamma$, such that for any $\varepsilon\in(0,1)$, there exists a positive constant $C(\varepsilon)>0$ satisfying
\[
|\nabla u(x)|\leq C(\varepsilon)\rho(x)^{\varepsilon-1}\quad \forall x\in M\backslash \Gamma.
\]
Let $h(x)=\rho(x)e^{u(x)}, x\in M$. Then we have
\begin{equation*}
	\begin{aligned}
		\Delta (\log h)=0 \ \mbox{ on } M\backslash \Gamma,
	\end{aligned}\label {h1}
\end{equation*}
and
\begin{equation*}
	\begin{aligned}
		\lim\limits_{\rho(x)\rightarrow 0}\frac{\log h(x)}{\log \rho(x)}=1.
	\end{aligned}\label {h2}
\end{equation*}
For any $\alpha>1$, Li-Tian \cite{LT} want to find a harmonic map into $\mathbf{\overline{H}}^2$ of form $(\varphi_1, h^\alpha e^{\varphi_2})$ satisfying $\varphi_1=0$ on $\Gamma$, equivalently, they want to find the critical point of the functional defined as follows:
\begin{equation*}\label{energy}
	\begin{aligned}
		F(\varphi_1, \varphi_2)=\int_M  h^{-2\alpha}e^{-2\varphi_2}|\nabla\varphi_1|^2 +|\nabla \varphi_2|^2 + \frac{2\alpha\nabla h\nabla \varphi_2}{h}\ud V,
	\end{aligned}
\end{equation*}
where $\ud V$ is the volume form on $M$.
In the case where $\Delta (\log h)=0 \mbox{ on } M\backslash \Gamma$, the last term in the above integration may be taken from $E$. There we have the reduced energy
\begin{equation}\label{reducedenergy}
	\begin{aligned}
		H(\varphi_1, \varphi_2)=\int_M  h^{-2\alpha}e^{-2\varphi_2}|\nabla\varphi_1|^2 +|\nabla \varphi_2|^2 \ud V.
	\end{aligned}
\end{equation}
The Euler-Lagrange equations of \eqref{reducedenergy} are
\begin{equation*}
	\left\{
	\begin{aligned}
		\mbox{div}(h^{-2\alpha}e^{-2\varphi_2}\nabla \varphi_1)&=0 \quad  \mbox{ in }\ M\setminus \Gamma\\
		\Delta\varphi_2+h^{-2\alpha}e^{-2\varphi_2}|\nabla \varphi_1|^2&=0 \quad \mbox{ in } \  M \setminus \Gamma.\\
	\end{aligned}
	\right. 
\end{equation*}
We will study the corresponding parabolic equation. For fixed $T>0$, we set $Q_T= M \backslash \Gamma \times(0,T]$ and consider the following problem, for $\alpha>1$
\begin{equation}
	\left\{
	\begin{aligned}
		\partial_t\varphi_1-h^{2\alpha}e^{2\varphi_2}\mbox{div}(h^{-2\alpha}e^{-2\varphi_2}\nabla \varphi_1)&=0 \ &\mbox{ in } Q_T\\
		\partial_t\varphi_2-\Delta\varphi_2-h^{-2\alpha}e^{-2\varphi_2}|\nabla \varphi_1|^2&=0 \ &\mbox{ in } Q_T\\
		\varphi_1=\varphi_1^0, \varphi_2&=\varphi_2^0 \ &\mbox{ on } M\times\{0\}\\
		\varphi_1&=0 \ &\mbox{ on }\Gamma \times [0,T]
	\end{aligned}
	\right. \label{eq:equation}
\end{equation}
where $\varphi_1^0$ and $\varphi_2^0$ are functions on $M$, serving as initial values for this parabolic problem. 

We define the following spaces and norm: 
\begin{equation*}
	\begin{aligned}
		\mathcal{A}_0&:=\{u(x) | u\in  C^{[2\alpha+2],2\alpha+2- [2\alpha+2]}(M), u=0 \mbox{ on } \Gamma\},\\
		\mathcal{B}_0&:=C^{[2\alpha+2],2\alpha+2- [2\alpha+2]}(M),\\
		\| w \|_{C_{*}^{2}(M)}(t)&:= \sum_{k=0}^{2} \max_M
		(|\nabla^k w_1| \rho^{k+\frac{3}{2}-\alpha}+|\nabla^kw_2| \rho^{k+\frac{3}{2}})(t).\\
	\end{aligned}
\end{equation*}

	Under the condition that $(\varphi _0^1, \varphi _0^2) \in \mathcal{A}_0 \times \mathcal{B}_0$, 
	our main concerns are the existence, regularity and convergence of the solution in weighted spaces. Namely, we will show
	\begin{thm}\label{thm:existence}
		Let $\varphi^0=(\varphi^0_1,\varphi^0_2)\in \mathcal{A}_0\times \mathcal{B}_0$. For any $0<T\leq\infty $, there exists a vector-valued function $\varphi=(\varphi_1,\varphi_2) $ solving the equation (\ref{eq:equation}) in $M\times [0,T]$.
		For any $0<\varepsilon<\min\{\frac{1}{2}, \alpha-1\}$, there is a uniform
		constant $C_{\varepsilon}>0$ such that
		\begin{align*}
			\|\varphi_j\|_{C^{k+\lambda,\frac{k+\lambda}{2}}(M\times [0,T])}\leq C_{\varepsilon}, j=1,2,
		\end{align*}
		where $k=[2\alpha-2\varepsilon]$, $\lambda=2\alpha-k-2\varepsilon$ and $C_{\varepsilon}$ depending only on $M,\varphi^0$. %In particular, for any $k\leq 5$, $|\nabla^k\varphi_2(x,t)|\leq C_{\varepsilon}\rho(x)^{6-k-2\varepsilon}$, for all $x \in M, t>0$.
		
		Let $\varphi(t)=\varphi(\cdot,t)$ for brevity, and consider $\varphi(t)$ as a map from $M$ into $\mathbf{\overline{H}}^2$. Then, as $t\rightarrow +\infty$,  $\varphi(t)\rightarrow\overline {\varphi}$ in $C_{*}^{2}(M)$,
		where $\overline{\varphi}$ satisfies the system
		\begin{equation*}
			\left\{\begin{aligned}
				\Delta\overline{\varphi}_1-2(\nabla\overline{\varphi}_2+\frac{\alpha\nabla h}{h})\nabla\overline{\varphi}_1&=0 \ \mbox{ in } M\\
				\Delta\overline{\varphi}_2+h^{-2\alpha}e^{-2\overline{\varphi}_2}|\nabla\overline{\varphi}_1|^2&=0 \ \mbox{ in } M.\\
			\end{aligned}\right.
		\end{equation*}
		Moreover, for $t\geqslant1$, there exist positive constants $C$ and $C_0$ depending only on $M,\varphi_0$ such that
		\begin{align*}
			\|\varphi(t)-\overline{\varphi}\|_{C_{*}^{2}(M)}
			(s)\leq C e^{-\frac{C_0}{4}t}.
		\end{align*}
		For any $0<\varepsilon<2\alpha$, $
		(\overline{\varphi}_1,\overline{\varphi}_2)\in C^{k,\lambda}(M)$ with $k=[2\alpha-2\varepsilon], \lambda=2\alpha-2\varepsilon-k$.
		\end{thm}

The primary challenge lies in proving the weighted estimates. To address this, we integrate the Schauder theory of linear parabolic equations with the method of weighted H\"older estimates for elliptic equations, as presented in Chapter 4 of \cite{GT}. This approach enables us to derive the weighted Schauder estimate for the linear parabolic equation.
Specifically, we decompose the original equation into two components: a homogeneous equation without boundary conditions and a non-homogeneous equation with zero Dirichlet boundary conditions. Subsequently, instead of directly estimating the explicit expression of solutions in Chapter 4 of \cite{GT}, we utilize the properties of classical parabolic equation solutions to obtain weighted estimates of these solutions.% However, a key challenges we face is that the integrability of the non-homogeneous term is inferior to that encountered in classical case.

When proving the optimal regularity and its convergence of the solution in Sect. \ref{section3}, to ensure that the integrations  involved in Remark \ref{rem} are meaningful, we need to choose an appropriate weighted space  in Sect. \ref{section2}. The singularity of the space $M\setminus \Gamma$ restricts the regularity of the solution only in the direction passing through $\Gamma$, and the regularity of the solution in other directions is constrained by the initial values. Therefore, to obtain an appropriate weighted space in Sect. \ref{section2}, the regularity of the initial values we need to assume must be higher than that of the solution itself. This leads to the need for additional proofs when demonstrating the long-term existence of the solution in Theorem \ref{longtime}.

%在sect. 3中证明解的最佳正则性的一致估计和收敛性的时候，我们对解的存在空间有要求(cf.remark3.1). 所以我们在sect.2 中证明解的存在性时要选择合适的加权空间。空间的奇异性只能限制解在穿过F的方向的正则性，解的其他方向性上的正则性受初值限制。为此我们需要设置的初值的正则性比解本身的正则性更高，这就在证明解的长时间存在性的时候需要额外的论证().

The organization of this paper is as follows. In Sect. \ref{section2}, we will establish the regularity results for the linearized equations of problem (\ref{eq:equation}). During this process, we adopt the Galerkin method to prove the existence of weak solutions for a system of linear parabolic equations in a weighted space and use interpolation inequality to make weighted Schauder estimates. In Sect. \ref{section3}, we employ the inverse function theorem of Banach spaces to demonstrate the existence of short time solutions. In Sect. \ref{section4}, we present consistent estimate of optimal regularity and the existence of long time solutions. In this section, we extend the method of Li-Tian\cite{LT}  to the parabolic equations. We hope to generalize the results to the higher dimensional case. However, a current limitation is that we have utilized the Sobolev embedding theorem and are aiming to derive a monotonicity formula in Sect. \ref{section4}.  Finally, we establish the  convergence rate of heat flow as $t\rightarrow  \infty $ and find its limit as a singular harmonic map in Sect. \ref{section5}. 

\section{Linearized equations of the reduced heat flow}\label{section2}

This section is devoted to obtain the existence and regularity for the linearized equations of (\ref{eq:equation}).
For convenience, we define the following Banach spaces and norms in order to discuss our problems.  For any $X=(x,s)$, $Y=(y,t)$, define $\delta(X,Y)=\max\{|x-y|,|s-t|^{\frac{1}{2}}\}$ as the parabolic distance. Regarding the additional notation $\rho_{X} =\rho(x)$ and $\rho_{X,Y}=\max\{\rho_{X},\rho_{Y}\}$, these definitions are somewhat vague without more context. Fix $0<\beta<\min\{2\alpha-2,1\}$ and take $2+\beta<\gamma<2\alpha$.

\begin{defn}\label{def1}
(i) The weighted Sobolev spaces  $ W^{4,2}_{2}(Q_T;\rho^{-\alpha})$ and $W^{2,1}_{2}(Q_T;\rho^{-\alpha})$	consist of all functions $u: Q_T \rightarrow \mathbb{R}$ with
	\begin{equation*}
		\begin{aligned}
			\|u\|_{W^{4,2}_{2}(Q_T;\rho^{-\alpha})}&:=(\sum_{0\leq i+2j\leq4}\|\rho ^{-\alpha-1+i+(2j-2)^{+}}\partial_t^{j}\nabla_x^i u\|^2_{L^2(Q_T)})^{\frac{1}{2}}
		\end{aligned}
	\end{equation*}
		and 
	\begin{equation*}
		\begin{aligned}
			\|u\|_{W^{2,1}_{2}(Q_T;\rho^{-\alpha})}&:=(\sum_{0\leq i+2j\leq2}\|\rho ^{-\alpha+1+i}\partial_t^{j}\nabla_x^i u\|^2_{L^2(Q_T)})^{\frac{1}{2}}
		\end{aligned}
	\end{equation*}
	are finite.\\
 (ii) The weighted H\"older spaces
	$C^{2+\beta,1+\frac{\beta}{2}}(\overline{Q_T};\rho^{-\gamma}) $ and $C^{\beta,\frac{\beta}{2}}(\overline{Q_T};\rho^{-\gamma})$
	consist of all functions $u: Q_T \rightarrow \mathbb{R}$ which the norm
	\begin{equation*}
		\begin{aligned}
			\|u\|_{C^{2+\beta,1+\frac{\beta}{2}}(\overline{Q_T};\rho^{-\gamma}) }:= &\sum\limits_{0 \leqslant i+2j\leqslant 2}(\sup\limits_{X\in Q_T}\rho_{X}^{i+2j-\gamma}|\partial_t^{j}\nabla^i u(X)|\\
			&+\sup\limits_{X,Y\in Q_T}\rho_{X,Y}^{i+2j+\beta-\gamma}\frac{|\partial_t^{j}\nabla^iu(X)-\partial_t^{j}\nabla^iu(Y)|}{\delta(X,Y)^{\beta}})
		\end{aligned}
	\end{equation*}
		and
	\begin{equation*}
		\begin{aligned}
			\|u\|_{C^{\beta,\frac{\beta}{2}}(\overline{Q_T};\rho^{-\gamma})}:=& \sup\limits_{X\in Q_T}\rho_{X}^{2-\gamma}| u(X)|+\sup\limits_{X,Y\in Q_T}\rho_{X,Y}^{2+\beta-\gamma}\frac{|u(X)-u(Y)|}{\delta(X,Y)^{\beta}}
		\end{aligned}
	\end{equation*}
	are finite.\\
\end{defn}
Analogous to the above definition, we define the space $ W^{1,1}_{2}(Q_T;\rho^{-\alpha})$
 with the finite norm
	\begin{equation*}
		\begin{aligned}
		\|u\|_{W^{1,1}_{2}(Q_T;\rho^{-\alpha})}:=(\|\rho^{-\alpha-1}u\|^2_{L^2(Q_T)}+\|\rho^{-\alpha}\nabla u\|^2_{L^2(Q_T)}+\|\rho^{-\alpha}\partial_t u\|^2_{L^2(Q_T)})^{\frac{1}{2}}
		\end{aligned}
		\end{equation*}
and 		
the space $ C({M};\rho^{-\gamma})$ with the finite norm
\begin{equation*}
	\begin{aligned}
		\|u\|_{C({M};\rho^{-\gamma})}:=\|\rho^{-\gamma}u\|_{C^0(M)}.
	\end{aligned}
\end{equation*}

 Given $0<\beta<1 $, $l\geqslant 1$, the Sobolev space $W^{2l,l}_2(Q_T)$, the H\"older space $C^{2l+\beta,l+\frac{\beta}{2}}(\overline{Q_T})$ and the $\beta$-H\"older seminorm of $u: Q_T \rightarrow \mathbb{R}$ are as usual.

In addition, we define the following spaces: 
\begin{equation*}
	\begin{aligned}
		\mathcal{A}:=&\{u(x,t)| u , \partial_t u  \in  C^{2+\beta,1+\frac{\beta}{2}}(\overline{Q_T};\rho^{-\gamma}),u\in W^{4,2}_{2}(Q_T;\rho^{-\alpha}),\rho^{-\alpha+3}\partial_t u  \in  W^{4,2}_{2}(Q_T),\\
		&u(x,0)=0 \},\\
		\mathcal{B}:=&\{u(x,t)| u , \partial_t u \in C^{2+\beta,1+\frac{\beta}{2}}(\overline{Q_T})\cap W^{4,2}_{2}(Q_T) ,u(x,0)=0\},\\
		\mathcal{C}:=& \{u(x,t) | u , \partial_t u \in C^{\beta,\frac{\beta}{2}}(\overline{Q_T};\rho^{-\gamma}),u\in  W^{2,1}_{2}(Q_T;\rho^{-\alpha}),  \rho^{-\alpha+3}\partial_tu\in W^{2,1}_2(Q_T) , \\
	&u(x,0)\in C({M};\rho^{-\gamma}),\rho^{-\alpha+3}u(0)\in H^3(M)\},	\\  
	\mathcal{D}:=& \{u(x,t) | u , \partial_t u\in C^{\beta,\frac{\beta}{2}}(\overline{Q_T})\cap W^{2,1}_{2}(Q_T),u(x,0)\in H^3(M)\}.
	\end{aligned}
\end{equation*}	
%Choose $e:M\times[0,T]\longrightarrow \mathbb{R}^2$ to be a smooth map with  $e\in C^{2+\frac{1}{2},1+\frac{1}{4}}(\overline{Q_T})\bigcap W^{4,2}_2(Q_T)\times C^{2+\frac{1}{2},1+\frac{1}{4}}_{\rho^3}(\overline{Q_T})\bigcap W^{4,2}_{2,\rho^3}(Q_T)$ and $e(x,0)=\varphi^0$. 

Denote $\varphi^{b} :=\varphi-\varphi^0$, then $\varphi^{b}=0$ at $\{t=0\}$ and $\varphi^{b}\in \mathcal{A}\times \mathcal{B} $. Also	define the  operator
\begin{equation}
	\mathcal{P}(\varphi)=\left\{
	\begin{aligned}
		& \partial_t\varphi_1-h^{2\alpha}e^{2\varphi_2}\mbox{div} (h^{-2\alpha}e^{-2\varphi_2}\nabla \varphi_1)\\
		&\partial_t\varphi_2-\Delta\varphi_2-h^{-2\alpha}e^{-2\varphi_2}|\nabla \varphi_1|^2 .\\
	\end{aligned}\right.
\end{equation}	

Consider $\varphi^b$ as the variable function. Then $\varphi^b\rightarrow \mathcal{P}(\varphi^b+\varphi^0)$ defines a differentiable map of $\mathcal{A}\times \mathcal{B}\rightarrow\mathcal{C}\times \mathcal{D}$. Its derivative in the direction $k\in \mathcal{A}\times\mathcal{B}$ at $\varphi^b=0$ is 
\begin{equation*}
	D\mathcal{P}(\varphi^0,k)=\left\{
	\begin{aligned}
		& \partial_tk_1-h^{2\alpha}e^{2\varphi^0_2}\mbox{div}(h^{-2\alpha}e^{-2\varphi^0_2}\nabla k_1)+2\nabla\varphi^0_1\nabla k_2\\
		&\partial_tk_2-\Delta k_2+2h^{-2\alpha}e^{-2\varphi_2^0}|\nabla \varphi^0_1|^2k_2-2h^{-2\alpha}e^{-2\varphi^0_2}\nabla \varphi^0_1\nabla k_1.\\
	\end{aligned}
	\right.\label {5}
\end{equation*}

Consider the problem
\begin{equation}\label {8'}
	\left\{
	\begin{aligned}
		 	D\mathcal{P}(\varphi^0,k)&=f \quad \mbox{ in } Q_T\\
		k_1=k_2&=0 \quad \mbox{ in } M\times \{ 0 \}.\\
	\end{aligned}\right.
\end{equation}	
where $f=(f_1,f_2) \in \mathcal{C}\times \mathcal{D} $ and $(\varphi_1^0,\varphi_2^0) \in \mathcal{A}_0\times \mathcal{B}_0$.

Combined with  $\varphi_1^0=0$ on $\Gamma$, we know that $|\nabla^k\varphi_1^0(x)|\leqslant C\rho(x)^{2\alpha+2-k}$ and $[\varphi_1^0(x)]_{k+\beta}\leqslant C\rho(x)^{2\alpha+2-k-\beta}$, for any $\beta\in (0,1),  k<2\alpha$. We will use this estimate in the proof later in this section. The main result of this section is the following theorem.
\begin{thm}\label{thm:theorem2}
	For every $f=(f_1,f_2) \in \mathcal{C}\times \mathcal{D} $, there exists a unique weak solution $k=(k_1,k_2)\in  \mathcal{A}\times\mathcal{B}$ satisfying 	$D\mathcal{P}(\varphi^0,k)=f$.
\end{thm}
\subsection{Existence and uniqueness of weak solutions}
In this section, we employ the Galerkin method to construct a weak solution of (\ref{8'}) and subsequently enhance its regularity.
%\begin{theorem}\label{theorem2}
%For every $f=(f_1,f_2) \in L^2(Q_T)\times L^{2,h^2}(Q_T) $, there exists a unique weak solution $v=(v_1,v_2)$ satisfies $DL(\varphi^0,v)=f$, where $v_1\in L^\infty (0,T;L^2{(M)})\bigcap L^2(0,T;H^1{(M)})$ and $v_2\in  L^\infty (0,T;L^{2,h^2}(M))\bigcap L^2(0,T;H^1_{h^2,0}(M))$	.
%\end{theorem}

Assume $\{\psi^{2}_{m} \}_{m=1}^{\infty}$ is a standard orthogonal basis of $L^{2}(M )$ and an orthogonal basis of $H^{1}(M)$.
We can choose $\{\psi^{2}_{m} \}_{m=1}^{\infty}$ to be the complete set of appropriately normalized eigenfunctions for $L=-\Delta$ in $H^{1}(M)$ , with $\lambda_k$ being the corresponding eigenvalues. That is
\begin{equation}
	\begin{aligned}
		-\Delta \psi^{2}_{m}=\lambda_m \psi^{2}_{m},\psi^{2}_{m} \in  H^{1}(M).
	\end{aligned}\label{8+}
\end{equation}
Set $\psi^{1}_{m} =h^{\alpha} e^{\varphi^0_2}\psi_{m}^2 $ . Then, the set $\{h^{-\alpha}e^{-\varphi^0_2}\psi^{1}_{m} \}_{m=1}^{\infty} $ is the subset of $H^{1}(M)$, and hence \begin{align*}
	\int_M \frac{|\psi^{1}_{m}|^2}{h^{2\alpha+2}}\ud V+\int_M \frac{|\nabla \psi^{1}_{m}|^2}{h^{2\alpha}}\ud V < \infty.
\end{align*}
We denote by $H^{1}_{0}(M; h^{-\alpha} )$ the completion of $C^{\infty}_0(M \backslash \Gamma)$ under the norm

\[
\|u\|_{H^{1}(M; h^{-\alpha})}^2:=\int_M h^{-2\alpha-2}u^2+h^{-2\alpha}|\nabla u|^2 \ud V .
\]
Then,  $\{\psi^{1}_{m} \}_{m=1}^{\infty}$ forms an  orthogonal basis of $H^{1}_{0}(M; h^{-\alpha})$.
Additionally, we define
\begin{align*}
	\|u\|^2_{L^{2}(M; h^{-\alpha})}:=\int_M h^{-2\alpha} e^{-2\varphi^0_1}|u|^2\ud V.
\end{align*}
We can see that $\|\psi _m^1\|^2_{L^{2}(M; h^{-\alpha})}=1$  since $\{\psi^{2}_{m} \}_{m=1}^{\infty}$ is a standard orthogonal basis of $L^{2}(M )$.
%Remark:
%\begin{equation}
%\begin{aligned}
%\sup\limits_{0<\overline{t}<T}\int_M  h^{-6}e^{6\varphi^0_1} | v^N_2|^2dV(\overline{t})
%\leq %\int_M (|\nabla v_1|^2+  h^{-4}e^{4\varphi^0_1} |\nabla v_2|^2)dV(0)%+\int_M 4h^{-4}e^{4\varphi^0_1}|\nabla \varphi^0_2|^2v_1^2 dV(0)
%C\int_M (|\varphi^0_1|^2+|\nabla\varphi^0_1|^2)+h^{-4}e^{4\varphi^0_1}(| \varphi^0_2|^2 +| \nabla \varphi^0_2|^2 )dV+C\int_0^T\int_M |f_1|^2 + h^{-4}e^{4\varphi^0_1}|f_2|^2 dVds\\
%\end{aligned}
%  \end{equation}

%For simplicity, we write the spaces involving time emit $M$ in the next content.
	\begin{defn}
		We say that $k=(k_1,k_2)\in (L^\infty (0,T;L^{2}(M; h^{-\alpha}))\cap L^2(0,T;H^1_{0}(M; h^{-\alpha})))\times (L^\infty (0,T;L^2(M))\cap L^2(0,T;H^1(M)))$ is a weak solution of problem (\ref{8'}) if
		\begin{equation}\label{eq:weaksolution}
			\left\{
			\begin{aligned}
				& (k_1(\cdot,t),h^{-2\alpha}e^{-2\varphi^0_2}\zeta_1(\cdot,t)) - \int_0^t (k_1,h^{-2\alpha}e^{-2\varphi^0_2}\partial_t\zeta_1(\cdot,t))\\&+(h^{-2\alpha}e^{-2\varphi^0_2}\nabla k_1, \nabla \zeta_1)\ud t+\int_0^t(2\nabla\varphi^0_1\nabla k_2,h^{-2\alpha}e^{-2\varphi^0_2}\zeta_1)\ud s\\
				&=\int_0^t(f_1,h^{-2\alpha}e^{-2\varphi^0_2}\zeta_1)\ud s\\
					&(k_2(\cdot,t),\zeta_2(\cdot,t))-\int_0^t(k_2,\partial_t\zeta_2)\ud s+\int_0^t(\nabla k_2,\nabla\zeta_2)\\
				&+(2h^{-2\alpha}e^{-2\varphi^0_2}|\nabla \varphi^0_1|^2k_2-2h^{-2\alpha}e^{-2\varphi^0_2}\nabla \varphi^0_1\nabla k_1,\zeta_2)\ud s\\&=\int_0^t(f_2,\zeta_2)\ud s\\
			\end{aligned}\right.
		\end{equation}
		for each $\zeta_1\in W^{1,1}_{2}(Q_T;\rho^{-\alpha})$, $\zeta_2\in W^{1,1}_{2}(Q_T)$ and any $t\in(0,T)$, where $(\cdot ,\cdot )$ denotes the inner product in $L^2(M)$.
	\end{defn}
	\begin{prop}\label{prop:prop1}
		For every $f=(f_1,f_2) \in L^{2}(0,T;L^2(M; h^{-\alpha+1}))\times L^{2}(Q_T)$, there exists a unique weak solution $k=(k_1,k_2)$ satisfying (\ref{8'}). 
	\end{prop}
	
	We will look for functions $k_1^N,k_2^N$ of the form
	\begin{equation}
		k_1^N=\sum ^N_{m=1}C^{N,1}_m(t)\psi^{1}_{m},\quad k_2^N=\sum ^N_{m=1}C^{N,2}_m(t)\psi^{2}_{m},\label{9}
	\end{equation}%12
	\\where we hope to select the coefficients $C^{N,1}_m(t),C^{N,2}_m(t)$ such that
	\begin{equation}
		\left \{
		\begin {aligned}
		&(\partial_t k_1^N, h^{-2\alpha} e^{-2\varphi^0_2}\psi^1_m)- (\mbox{div}(h^{-2\alpha}e^{-2\varphi^0_2}\nabla k^N_1),\psi^1_m)\\&+2(\nabla\varphi^0_1\nabla k^N_2, h^{-2\alpha} e^{-2\varphi^0_2}\psi^1_m)=(f_1,h^{-2\alpha}e^{-2\varphi^0_2}\psi^1_m),\\
		&(\partial_t k_2^N,\psi^2_m)-(\Delta k_2^N,\psi^2_m) +2(h^{-2\alpha}e^{-2\varphi ^0_2}|\nabla \varphi^{0}_{1}|^2k_2^N,\psi^2_m)\\
		&-2(h^{-2\alpha}e^{-2\varphi^0_2}\nabla \varphi^{0}_{1}\nabla k_1^N,\psi^2_m)=(f_2,\psi^2_m)\\
		&C_m^{N,1}(0)=0, C_m^{N,2}(0)=0
		\end {aligned}
		\label {10}\right.
	\end{equation}
	for $0 \leqslant t\leqslant T$, $m=1,2,\ldots,N$.
	%The second equation of (\ref{13}) is equivalent to
	%\begin{equation}
	%\begin {aligned}
	%	(\partial_tv_2^N,\psi^2_k)-(h^4e^{-4\varphi^0_1}div (h^{-4}e^{4\varphi^0_1}\nabla v^N_2),\psi^2_k)-(4\nabla\varphi^0_2\nabla v^N_1,\psi^2_k)=(f_2,\psi^2_k)\label {B13}
	%\end {aligned}	
	%\end{equation}%14
	Then (\ref{9}) and (\ref{10}) imply
	\begin{equation}\left \{
		\begin {aligned}
		\partial_tC_m^{N,1}(t) +\sum _{l=1}^{N}C_l^{N,1}(t)A_{ml} +\sum _{l=1}^{N}C_l^{N,2}(t)B_{ml}&=F_m^1	\\
		\partial_tC_m^{N,2}(t)+\sum _{l=1}^{N}C_l^{N,2}(t)C_{ml}+\sum _{l=1}^{N}C_l^{N,1}(t)D_{ml}&=F_m^2 \\
		C_m^{N,1}(0)&=0\\
		C_m^{N,2}(0)&=0 ,
		\end {aligned}\right.
		\label {11}
	\end{equation}
	where the coefficients are as follows
	\begin{equation*}
		\begin{aligned}
		A_{ml}&=(h^{-2\alpha}e^{-2\varphi_2^0}\nabla \psi_l^1,\nabla\psi_m^1),\\
		B_{ml}&=(2\nabla \varphi ^0_1\nabla \psi ^2_l, h^{-2\alpha}e^{-2\varphi^0_2}\psi _m^1),\\
		C_{ml}&=2(h^{-2\alpha}e^{-2\varphi ^0_2}|\nabla \varphi ^0_1|^2\psi ^2_l,\psi ^2_m)-(\Delta \psi _l^2,\psi _m^2),\\
		D_{ml}&=(-2h^{-2\alpha}e^{-2\varphi ^0_2}\nabla \varphi^0_1\nabla \psi _l^1,\psi _m^2),\\	
		F_m^1&=(f_1,h^{-2\alpha}e^{-2\varphi^0_2}\psi _m^1),\quad  F_m^2=(f_2,\psi _m^2).	
		\end{aligned}\label{12}
	\end{equation*}
	This forms a $2N$ dimensional linear system of ODEs subject to initial conditions $C_m^{N,1}(0)=0, C_m^{N,2}(0)=0$. According to standard existence theory for ODEs, there exists a unique solution $\{C_1^{N,1}(t),\ldots,C_N^{N,1}(t),C_1^{N,2}(t),\ldots,C_N^{N,2}(t)\}$ on $[0,T]$ such that (\ref{11}) holds. And then $k_1^N, k_2^N$ defined by (\ref{9}) solve (\ref{10}).
						
	To obtain the solution of (\ref{8'}), we need to derive the energy estimate for (\ref{10}). This will allow us to establish the corresponding weak convergence. To prove this, we will utilize the following remark.
  \begin{rem}\label{rem1} Using Corollary 4.1 of \cite{LT}, we have
 	\begin{equation}
 	    \begin{aligned}
 	\int_M \frac{|k_1^N|^2}{h^{2\alpha+2}}\mathrm{d} V(t)\leq C\int_M  \frac{|\nabla k_1^N|^2}{h^{2\alpha}}\mathrm{d} V(t) ,\quad \text{for}\ N=1,2,\dots, \text{and} \ t\in [0,T].
	%\leq  C\int_0^T\int_M |f_1|^2 + h^{-6}e^{-2\varphi^0_1}|f_2|^2 \ud V\ud s
	    \end{aligned}\label{40}
	\end{equation}
 \end{rem}
						
  \begin{lem}\label{lem:2.4} (Energy estimate).
There exists a constant $C$, depending on $M$ , $T$ and $\varphi^0$ such that
	\begin{equation*}
	     \begin{aligned}
 		&\max\limits_{0\leqslant t \leqslant T}(\|\rho^{-\alpha}k^N_1\|_{L^2(M)}+\|k^N_2\|^2_{L^2(M)})
		+\|\rho^{-\alpha}k^N_1\|_{L^2(0,T;H^1(M))}+\| k^N_2\|_{L^2(0,T;H^1(M))}\\
		\leq& C(\|\rho^{-\alpha+1}f_1\|_{{L^2(0,T;L^2(M))}}+\| f_2\|_{{L^2(0,T;L^2(M))}})
		 \end{aligned}\label {13}
	\end{equation*}
for $N=1,2,\cdots$.
  \end{lem}
						
	\begin{proof}
Multiply the first equation of (\ref{10}) by $C^{N,1}_m(t)$ and the second equation of (\ref{10}) by $C^{N,2}_m(t)$, then sum for $m=1,\cdots, N$. Recall (\ref{9}) to obtain the following system of equations for $0 \leqslant t\leqslant T$:
							
  \begin{equation}
	\left \{
	 \begin{aligned}
	 	&(\partial_tk_1^N,h^{-2\alpha}e^{-2\varphi^0_2}k_1^N)-(\mbox{div} (h^{-2\alpha}e^{-2\varphi^0_2}\nabla k^N_1),k_1^N)\\&+2(\nabla\varphi^0_1\nabla k^N_2,h^{-2\alpha}e^{-2\varphi^0_2}k_1^N)=(f_1,h^{-2\alpha}e^{-2\varphi^0_2}k_1^N),\\
		&	(\partial_t k_2^N,k_2^N)-(\Delta k_2^N,k_2^N) +2(h^{-2\alpha}e^{-2\varphi ^0_2}|\nabla \varphi^{0}_{1}|^2k_2^N ,k_2^N)\\&-2(h^{-2\alpha}e^{-2\varphi ^0_2}\nabla \varphi^{0}_{1}\nabla k_1^N,k_2^N) =(f_2,k_2^N).
    	\end{aligned}
	\right.\label{14}
 \end{equation}%13
  
 	Notice that 
\[
	 (\partial_tk_1^N,h^{-2\alpha}e^{-2\varphi^0_2}k_1^N)=\partial_t \frac{1}{2}\int_M h^{-2\alpha}e^{-2\varphi^0_2}|k_1|^2   \ud V  
	 \mbox{  and  }	(\partial_t k_2^N,k_2^N)=\partial_t \frac{1}{2}\int_M| k^N_2|^2\ud V . 
\]
Then (\ref{14}) and Cauchy's inequality  yield
	\begin{equation}
		\begin{aligned}
		   &\partial_t \frac{1}{2}\int_M h^{-2\alpha}e^{-2\varphi^0_2}|k^N_1|^2   \ud V +\int_M h^{-2\alpha}e^{-2\varphi^0_2}|\nabla k^N_1|^2\ud V\\
		   \leq &C\int_M h^{-2\alpha}e^{-2\varphi^0_2}| k^N_1|^2 + h^{-2\alpha+2}e^{-2\varphi^0_2} |f_1|^2 \ud V \\&+\delta\int_M |\nabla k^N_2|^2+ h^{-2\alpha-2}| k^N_1|^2\ud V \\
		\end{aligned}\label{15-}
	\end{equation}
	and\begin{equation}
		\begin{aligned}	   &\partial_t \frac{1}{2}\int_M| k^N_2|^2\ud V+\int_M|\nabla k^N_2|^2 \ud V\\ \leq & C\int_M|k^N_2|^2+|f_2|^2 \ud V+\delta\int_M h^{-2\alpha}e^{-2\varphi^0_2}|\nabla k^N_1|^2 \ud V.\\
		\end{aligned}\label{15}
	\end{equation}
In fact, we used integration by parts in $M\setminus  \Gamma$, which holds true and is not difficult to be deduced from Lemma 3.1 in \cite{LT}. Adding (\ref{15-}) and (\ref{15}),  we take $\delta>0$ small to get
	\begin{equation}
 		\begin{aligned}
			&\partial_t \frac{1}{2}\int_Mh^{-2\alpha}e^{-2\varphi^0_2}| k^N_1|^2+|k^N_2|^2 \ud V+\int_Mh^{-2\alpha}e^{-2\varphi^0_2}|\nabla k^N_1|^2+|\nabla k^N_2|^2 \ud V\\
			\leq & C(\int_Mh^{-2\alpha}e^{-2\varphi^0_2}| k^N_1|^2+|k^N_2|^2 \ud V+\int_M  h^{-2\alpha+2}e^{-2\varphi^0_2}|f_1|^2+|f_2|^2  \ud V).
		\end{aligned}\label{17}
	\end{equation}
	where we used (\ref{40}) in the first inequality.
Thus the Gronwall's inequality yields the estimate
	\begin{equation*}
		\begin{aligned}
			\max\limits_{0\leq{t}\leq T}\int_Mh^{-2\alpha}e^{-2\varphi^0_2}|k^N_1|^2+| k^N_2|^2 \ud V({t}) \leq C\int_0^T\int_M h^{-2\alpha+2}e^{-2\varphi^0_2}|f_1|^2 + |f_2|^2 \ud V\ud s.
		\end{aligned}\label{18}
	\end{equation*}

Returning to inequality (\ref{17}), we integrate from $0$ to $T$ with respect to $t$ and employ the inequality above to find
	\begin{equation*}
		\begin{aligned}
			\int_0^T\int_M h^{-2\alpha}e^{-2\varphi^0_2}|\nabla k^N_1|^2 +|\nabla k^N_2|^2\ud V\ud s
			\leq C\int_0^T\int_M h^{-2\alpha+2}e^{-2\varphi^0_2}|f_1|^2 + |f_2|^2 \ud V\ud s.
		\end{aligned}\label{19}
	\end{equation*}
		\end{proof}
 \begin{rem}\label{rem1-1} Take $\varepsilon$ such that $0<\varepsilon<\frac{1}{2}(\alpha-1)$. Then, Lemma \ref{lem:2.4} still holds provided that the initial function $(\varphi_1^0,\varphi_2^0)\in C^{[2\alpha-2\varepsilon],2\alpha-2\varepsilon- [2\alpha-2\varepsilon]}(M)\times \{u\in  C^{[2\alpha-2\varepsilon],2\alpha-2\varepsilon- [2\alpha-2\varepsilon]}(M)| u=0 \mbox{ on } \Gamma\}$.
 \end{rem}
%  \begin{rem}\label{rem1-2}
% Differentiating equations $\eqref{10}$ with respect $t$ and we can get for  $\partial_tk=(\partial_t k_1, \partial_tk_2)$ have the same equations as $\eqref{14}$ after replacing $(f_1, f_2)$ with $(\partial_tf_1, \partial_tf_2)$. Notice that $\partial_t k_1(0)=f_1(0)$ and $\partial_t k_2(0)=f_2(0)$, we can get for  $\partial_tk=(\partial_t k_1, \partial_tk_2)$ the following inequality:
%	\begin{equation*}
%		\begin{aligned}
%		&\max\limits_{0\leq{t}\leq T}\int_Mh^{-2\alpha}e^{-2\varphi^0_2}|\partial_tk^N_1|^2+| \partial_tk^N_2|^2 \ud V({t}) +\int_0^T\int_M h^{-2\alpha}e^{-2\varphi^0_2}|\nabla \partial_t k^N_1|^2 +|\nabla \partial_tk^N_2|^2\ud V\ud s\\
%		\leq& C(\int_Mh^{-2\alpha}e^{-2\varphi^0_2}|f_1|^2+| f_2|^2 \ud V({0})
%		+\int_0^T\int_M h^{-2\alpha+2}e^{-2\varphi^0_2}|\partial_tf_1|^2 + |\partial_tf_2|^2 \ud V\ud s)\\
%		\leq & C\int_0^T\int_M h^{-2\alpha}e^{-2\varphi^0_2}(|f_1|^2+|\partial_tf_1|^2 )+ |f_2|^2+|\partial_tf_2|^2 \ud V\ud s.\\
%		\end{aligned}
%	\end{equation*}	
%	The last inequality follows from Theorem 2 in Chapter 5.9.2 in \cite{LC}.
%\end{rem}		
Next, we give the proof of Proposition \ref{prop:prop1}.
		\begin{proof}
For any $N,m,t$,
	\begin{equation*}
		\begin{aligned}
			C_m^{N,1}(t)&=\int_M k_1^N h^{-2\alpha}e^{-2\varphi^0_2}\psi^1_m \ud V\leq \|k_1^N\|_{L^{2}(M;h^{-\alpha})},\\
			  C_m^{N,2}(t)&=\int_M k_2^N \psi_m^2\ud V\leq \| k_2^N\|_{L^2(M)}.\\
		\end{aligned}\label{20}
	\end{equation*}
Lemma \ref{lem:2.4} yields that $C_m^{N,1}(t)$ and $C_m^{N,2}(t)$ are uniformly bounded about $N,m,t$. 
Next, we will prove that for fixed $m$, $C_m^{N,1}(t)$ and $C_m^{N,2}(t)$ are equi-continuous about $N$. With the uniform estimates above, we integrate from $t$ to  $t+\Delta t$ in (\ref{11}) to find
	\begin{equation*}
		\begin{aligned}
			&|C_m^{N,1}(t+\Delta t)-C_m^{N,1}(t)|\\
			\leq &C\left(\|f_1\|_{L^{2}(0,T;L^2(M;h^{-\alpha+1}))}+\|\nabla k_2^N\|_{L^2(Q_T)}+\|\nabla k_1^N\|_{L^{2}(0,T;L^2(M;h^{-\alpha}))}\right)\\
			&\cdot\|\psi_m^1\|_{H^{1}(M;h^{-\alpha})}|\Delta t|^{\frac{1}{2}}
			\leq C |\Delta t|^{\frac{1}{2}}
		\end{aligned}\label {22}
	\end{equation*}
	and
		\begin{equation*}
		\begin{aligned}
			&|C_m^{N,2}(t+\Delta t)-C_m^{N,2}(t)|\\
			\leq	& C \left( \|f_2\|_{L^{2}(Q_T)}+\| k_2^N\|_{L^2(Q_T)}+\| \nabla k_2^N\|_{L^2(Q_T)}+\|\nabla k_1^N\|_{L^{2}(Q_T;\rho^{-\alpha})}\right)\cdot\|\psi _m^2\|_{H^1(M)}|\Delta t|^{\frac{1}{2}}\\
			\leq	& C |\Delta t|^{\frac{1}{2}},
		\end{aligned}\label {21}
	\end{equation*}
where we have used the H\"{o}lder inequality and $C$ is independent of $N,m$ and $\Delta t$. 

Hence, for every fixed $m$, $\{(C_m^{N,1}(t),C_m^{N,2}(t))\}_N$ is equi-continuous about $N$. By the Generalized Arzela-Ascoli's theorem, there exist a subsequence $\{(C_m^{N_i,1}(t), C_m^{N_i,2}(t))\}_{N_i}$ and functions $C^{1}_m(t),C^{2}_m(t)$ such that  
\[
	(C_m^{N_i,1}(t), C_m^{N_i,2}(t))\rightarrow  (C^{1}_m(t),C^{2}_m(t)) \ \mbox{ uniformly as }  \enspace N_i \rightarrow \infty \enspace in \enspace[0,T].
\]
Using the orthogonality of the basis functions and the energy estimate provided in Lemma \ref{lem:2.4} , for any $r \in \mathbb {N}$ with $r\leq N$, we can derive the following result:
	\begin{equation*}
		\begin{aligned}
			\sum ^r_{m=1}[C_m^{N,1}(t)]^2\leq \|k_1^N(\cdot ,t)\|^2_{L^{2}(M; h^{-\alpha})}\leq C, \\
			\sum ^r_{m=1}[C_m^{N,2}(t)]^2\leq \|k_2^N(\cdot ,t)\|^2_{L^{2}(M)}\leq C, 
		\end{aligned}\label {23}
	\end{equation*}
where $t\in [0,T]$ and $C$ is independent of $N$ and $t$. Let $N$ tend to $+\infty$ as $N_i$, then for any $ r\in \mathbb {N}$,
	\begin{equation}
		\sum ^{r}_{m=1}[C_m^{1}(t)]^2+[C_m^{2}(t)]^2\leq C,\quad \forall t\in [0,T].	\label {24}
	\end{equation}
							
Set $k_1=\sum \limits^\infty _{m=1}C^{1}_m(t)\psi^{1}_{m}$, $k_2=\sum \limits^\infty _{m=1}C^{2}_m(t)\psi^{2}_{m}$.
According to \eqref{24}, for any $ t\in [0,T]$, $k_1(\cdot,t) \in L^2(M; h^{-\alpha}), k_2(\cdot,t) \in L^{2}(M)$ and 
$$\sup\limits_{0\leqslant t\leqslant T} \|k_1(\cdot ,t)\|_{L^2(M; h^{-\alpha})}+\sup\limits_{0\leqslant  t\leqslant T}\|k_2(\cdot ,t)\|_{L^{2}(M)}\leq C.$$ Therefore, for $m\leqslant N_i$,
	\begin{equation*}
		\begin{aligned}
			&(k_1^{N_i}(\cdot,t)-k_1(\cdot,t),h^{-2\alpha}e^{-2\varphi_2^0}\psi _m^1)+(k_2^{N_i}(\cdot,t)-k_2(\cdot,t),\psi _m^2)\\
			&=C_m^{N_i,1}(t)-C_m^{1}(t)+C_m^{N_i,2}(t)-C_m^{2}(t)\\
			&\rightarrow 0,  \mbox{uniformly in } [0,T].
		\end{aligned}\label {26}
	\end{equation*}
Thus for any $ \phi _1\in L^2(M ;h^{-\alpha})$, $ \phi _2\in L^{2}(M)$, we have
	\begin{equation*}
		\begin{aligned}
			(k_1^{N_i}(\cdot,t)-k_1(\cdot,t), h^{-2\alpha}e^{-2\varphi_2^0}\phi _1)+(k_2^{N_i}(\cdot,t)-k_2(\cdot,t), \phi_2)\rightarrow 0, \mbox{uniformly in } [0,T].
		\end{aligned}\label {27}
	\end{equation*}
By Lemma \ref{lem:2.4}, there also exist subsequences of $\{h^{-\alpha} e^{-\varphi_2^0} k_1^{N_i}\}_{i=1}^{\infty}$ and $\{k_2^{N_i}\}_{i=1}^{\infty}$ such that the following weak convergences hold:
							
	\begin{equation*}
		\left \{
		\begin{aligned}
		h^{-\alpha} e^{-\varphi_2^0}k_1^{N_i}\rightharpoonup h^{-\alpha} e^{-\varphi_2^0}k_1, \	k_2^{N_i}\rightharpoonup k_2 \quad &in \ L^2(M),\\
		h^{-\alpha} e^{-\varphi_2^0}\nabla k_1^{N_i}\rightharpoonup h^{-\alpha} e^{-\varphi_2^0}\nabla k_1, \ \nabla k_2^{N_i}\rightharpoonup  \nabla k_2  \quad &in \ L^2(Q_T).
		%	h^{-3} e^{-\varphi _1^0}\nabla k_2^{N_i}\rightharpoonup  h^{-3} e^{-\varphi _1^0}\nabla k_2 \quad in \quad L^{2,h^3}(Q_T)\\
		\end{aligned}
		\right.\label {28}
	\end{equation*}
						
Next we will check that $k_1,k_2$ satisfy \eqref{eq:weaksolution}.
Set
	\begin{equation}\label{a28}
		\phi_1^r(x,t)=\sum \limits^r_{m=1}\theta _m^1(t)\psi _m^1(x), \
		\phi_2^r(x,t)=\sum \limits^r_{m=1}\theta _m^2(t)\psi _m^2(x),
	\end{equation}
where $\theta _m^1,\theta _m^2(m=1,2,\cdots, r.)$ are smooth functions in [0,T]. 
It's easy to check that $k_1^{N_i}$ and $k_2^{N_i}$ satisfy	\begin{equation*}
		   \begin{aligned}
				&(k_1^{N_i}(\cdot ,t),h^{-2\alpha}e^{-2\varphi_2^0}\phi_1^r(\cdot ,t))-\int_0^t (k_1^{N_i},h^{-2\alpha}e^{-2\varphi_2^0}\partial_t\phi_1^r) \ud s+\int_0^t (h^{-2\alpha}e^{-2\varphi^0_2}\nabla k^{N_i}_1,\nabla\phi_1^r) \ud s\\
				&=\int _0^t-(2\nabla\varphi^0_1\nabla k^{N_i}_2,h^{-2\alpha}e^{-2\varphi_2^0}\phi_1^r)+(f_1,h^{-2\alpha}e^{-2\varphi_2^0}\phi_1^r)\ud s
				\end{aligned}\end{equation*}
	and			\begin{equation*}
		   \begin{aligned}
					&(k_2^{N_i}(\cdot ,t),\phi_2^r(\cdot ,t))-\int _0^t( k_2^{N_i},\partial_t \phi_2^r)\ud s+\int _0^t(\nabla k_2^{N_i},\nabla\phi_2^r)\ud s\\
				&=\int _0^t(f_2,\phi_2^r)-2(h^{-2\alpha}e^{-2\varphi_2^0}|\nabla \varphi _1^0|^2k_2^{N_i},\phi_2^r)+(2h^{-2\alpha}e^{-2\varphi _2^0}\nabla \varphi _1^0\nabla k_1^{N_i},\phi_2^r)\ud s.\\
			\end{aligned}
\label {30}
	\end{equation*}
As $N_i$ tends to infinity, $k_1$ and $k_2$ satisfy the following equatities:	\begin{equation*}
			\begin{aligned}
				&(k_1(\cdot ,t),h^{-2\alpha}e^{-2\varphi_2^0}\phi_1^r(\cdot ,t))-\int_0^t (k_1,h^{-2\alpha}e^{-2\varphi_2^0}\partial_t\phi_1^r) \ud s+\int_0^t (h^{-2\alpha}e^{-2\varphi^0_2}\nabla k_1,\nabla\phi_1^r) \ud s\\
				&=\int _0^t-(2\nabla\varphi^0_1\nabla k_2,h^{-2\alpha}e^{-2\varphi_2^0}\phi_1^r)+(f_1,h^{-2\alpha}e^{-2\varphi_2^0}\phi_1^r)\ud s\end{aligned}\end{equation*}
				and 
			\begin{equation*}
			\begin{aligned}
					&(k_2(\cdot ,t),\phi_2^r(\cdot ,t))-\int _0^t( k_2,\partial_t \phi_2^r)\ud s+\int _0^t(\nabla k_2,\nabla\phi_2^r)\ud s\\
				&=\int _0^t(f_2,\phi_2^r)-2(h^{-2\alpha}e^{-2\varphi _2^0}|\nabla \varphi_1^0|^2k_2,\phi_2^r)+(2h^{-2\alpha}e^{-2\varphi _2^0}\nabla \varphi _1^0\nabla k_1,\phi_2^r)\ud s.\\
			\end{aligned}
		\label {31}
	\end{equation*}
The equalities hold for all functions $\zeta_1\in W^{1,1}_{2}(Q_T;\rho^{-\alpha}) $ and $\zeta_2\in W^{1,1}_{2}(Q_T) $ as functions of the form given in \eqref{a28} are dense in these two spaces, respectively. 

Regarding the uniqueness of the solution, it suffices to check that the only weak solution to (\ref{8'}) when $f\equiv 0$ is $k\equiv 0$.
By Lemma \ref{lem:2.4}, we can conclude that
	\begin{equation*}
		\begin{aligned}
		  &\max\limits_{0\leq {t}\leq T}\int_M h^{-2\alpha}e^{-2\varphi^0_2}|k_1|^2+| k_2|^2 \ud V({t})
		  +\int_0^T\int_M h^{-2\alpha}e^{-2\varphi^0_2}|\nabla k_1|^2 +|\nabla k_2|^2\ud V\ud s =0.
		\end{aligned}\label {31'}
   \end{equation*}
It's easy to see that  $k\equiv 0$.
	\end{proof}

\subsection{Weighted regularity estimates}
%In this section, we do not distinguish between $h^{\alpha}e^{}$ and $\rho$ on the right-hand side of the inequalities.
In this section, we will provide some weighted estimates for the solutions of (\ref{8'}) in order to improve their regularity. For brevity, we denote the time derivatives $(\partial_t{k_1},\partial_t{k_2})$ of a given vector function $k=(k_1, k_2)$ by $\widetilde{k}=(\widetilde{k}_1, \widetilde{k}_2)$.
\begin{lem}\label{lem:W112}
If $(k_1, k_2)$ is the solution of (\ref{8'}), then there exists a constant $C$ depending on $M$ , $T$ and $\varphi^0$ such that
  	\begin{equation*}
  	\begin{aligned}
  		&\sup\limits_{0\leqslant t \leqslant T}(\|\rho^{-\alpha}\partial_t k_1(t)\|_{L^2(M)}+\|\partial_t k_2(t)\|^2_{L^2(M)}+\|\rho^{-\alpha} k_1(t)\|_{H^1(M)}+\| k_2(t)\|^2_{H^1(M)})\\
  		&+\|\rho^{-\alpha}\partial_tk_1\|_{L^2(0,T;H^1(M))}+\| \partial_tk_2\|_{L^2(0,T;H^1(M))}\\
  		\leq &C\left(\|\rho^{-\alpha+1}f_1\|_{{L^2(0,T;L^2(M))}}+\|\rho^{-\alpha+1}\partial_t f_1\|_{{L^2(0,T;L^2(M))}}+\| f_2\|_{{H^1(0,T;L^2(M))}}\right.
\\&  \left.+\|\rho^{-\alpha}f_1(0)\|_{L^2(M)}^2+\|f_2(0)\|_{L^2(M)}^2\right).
  	\end{aligned}\label {32}
  \end{equation*}
\end{lem}

\begin{proof}
Fixing $N\geqslant 1$ and differentiating equations \eqref {10} with respect to $t$, we find
\begin{equation}
	\left \{
	\begin {aligned}
	&(\partial_t\widetilde{k^N_1},h^{-2\alpha} e^{-2\varphi^0_2}\psi^1_m)- (\mbox{div}(h^{-2\alpha}e^{-2\varphi^0_2}\nabla \widetilde{k^N_1}),\psi^1_m)\\&+2(\nabla\varphi^0_1\nabla \widetilde{k^N_2},h^{-2\alpha} e^{-2\varphi^0_2}\psi^1_m)=(\partial_t f_1,h^{-2\alpha} e^{-2\varphi^0_2}\psi^1_m),\\
	&(\partial_t \widetilde{k_2^N},\psi^2_m)-(\Delta \widetilde{k_2^N},\psi^2_m) +2(h^{-2\alpha}e^{-2\varphi ^0_2}|\nabla \varphi^{0}_{1}|^2\widetilde{k_2^N} ,\psi^2_m)\\&-2(h^{-2\alpha}e^{-2\varphi ^0_2}\nabla \varphi^{0}_{1}\nabla \widetilde{k_1^N},\psi^2_m)=(\partial_t f_2,\psi^2_m), m=1,\dots N .\\	
	\end {aligned}\label{a}
	\right.
\end{equation}%13 
Multiply the first equation of (\ref{a}) by $\frac{dC_m^{N,1}}{dt}$ and multiply the second equation of (\ref{a}) by $\frac{dC_m^{N,2}}{dt}$, then sum for $m=1,\cdots , N$ to find 
\begin{equation}
	\left \{
	\begin {aligned}
	&(\partial_t\widetilde{k_1^N},h^{-2\alpha} e^{-2\varphi^0_2}\widetilde{k_1^N})- (\mbox{div}(h^{-2\alpha}e^{-2\varphi^0_2}\nabla \widetilde{k}^N_1),\widetilde{k_1^N})\\
	&+2(\nabla\varphi^0_1\nabla \widetilde{k^N_2},h^{-2\alpha} e^{-2\varphi^0_2}\widetilde{k_1^N})=(\partial_t f_1,h^{-2\alpha} e^{-2\varphi^0_2}\widetilde{k}_1^N),\\
		&(\partial_t \widetilde{k_2^N},\widetilde{k_2^N})-(\Delta \widetilde{k_2^N},\widetilde{k_2^N}) +2(h^{-2\alpha}e^{-2\varphi ^0_2}|\nabla \varphi^{0}_{1}|^2\widetilde{k_2^N} ,\widetilde{k_2^N})\\
	&-2(h^{-2\alpha}e^{-2\varphi ^0_2}\nabla \varphi^{0}_{1}\nabla \widetilde{k_1^N},\widetilde{k_2^N})=(\partial_t f_2,\widetilde{k_2^N}).\\	
	\end {aligned}\label{b}
	\right.
\end{equation}%13
 Substituting $k_i^N$ and $f_i$ with $\widetilde{k_i^N}$ and $\widetilde{ f_i}$ $(i=1,2)$  in the proof of Lemma \ref{lem:2.4}, we have
\begin{equation}
	\begin{aligned}
		&\max\limits_{0\leq {t}\leq T}\int_M h^{-2\alpha}e^{-2\varphi^0_2} | \widetilde{k^N_1}|^2+| \widetilde{k^N_2}|^2 \ud V({t})\\&
		+\int_{Q_T} h^{-2\alpha}e^{-2\varphi^0_2}|\nabla  \widetilde{k^N_1}|^2 +|\nabla  \widetilde{k^N_2}|^2\ud V\ud s\\
		\leq & C(\int_M h^{-2\alpha}e^{-2\varphi^0_2}|\partial_t  k^N_1|^2 + |\partial_t  k^N_2|^2 \ud V(0) \\&+\|h^{-\alpha+1}e^{-\varphi^0_2}\widetilde f_1\|^2_{L^2(Q_T)}+ \|\widetilde{f}_2 \|^2_{L^2(Q_T)}).
	\end{aligned}\label{c}
\end{equation}
We employ \eqref {10} to find $\partial_t k_1(0)=f_1(0)$ and $\partial_t k_2(0)=f_2(0)$, therefore
\begin{equation*}
	\begin{aligned}
		&	\int_M h^{-2\alpha}e^{-2\varphi^0_2}|\partial_t  k^N_1|^2 + |\partial_t  k^N_2|^2 \ud V(0) 
		&\leq C\left(\|h^{-\alpha}e^{-\varphi^0_2}f_1(0)\|_{L^2(M)}^2 +\|f_2(0)\|_{L^2(M)}^2 \right).
	\end{aligned}\label{d}
\end{equation*}
Passing to limits as $N=N_i\rightarrow \infty$ in (\ref{c}), we deduce that
\begin{equation}
	\begin{aligned}
		&\max\limits_{0\leq {t}\leq T}\int_M h^{-2\alpha}e^{-2\varphi^0_2} | \widetilde{k_1}|^2+| \widetilde{k_2}|^2 \ud V({t})
		+\int_{Q_T} h^{-2\alpha}e^{-2\varphi^0_2}|\nabla  \widetilde{k_1}|^2 +|\nabla  \widetilde{k_2}|^2\ud V\ud s\\
		\leq &C(\|h^{-\alpha}e^{-\varphi^0_2}f_1(0)\|_{L^2(M)}^2 +\|f_2(0)\|_{L^2(M)}^2 \\&+\|h^{-\alpha+1}e^{-\varphi^0_2}\widetilde {f_1}\|^2_{L^2(Q_T)}+ \|\widetilde{f_2} \|^2_{L^2(Q_T)}).
	\end{aligned}\label{cc}
\end{equation}
Fixing $N\geq 1$, we multiply the first equation of (\ref{10}) by $\frac{dC_m^{N,1}}{dt}$ and multiply   the second equation of (\ref{10}) by $\frac{dC_m^{N,2}}{dt}$, summing $m$ from 1 to N to discover

 \begin{equation*}
\begin{aligned}
&\int_M h^{-2\alpha}e^{-2\varphi^0_2} | \partial_tk^N_1|^2+|\partial_tk^N_2 |^2\ud V+\frac{1}{2}\partial_t\int_M ( h^{-2\alpha}e^{-2\varphi^0_2} |\nabla k^N_1|^2+   |\nabla k^N_2|^2)\ud V\\
=&-2\int_M h^{-2\alpha}e^{-2\varphi^0_2}|\nabla \varphi^0_1|^2k^N_2\partial_t k^N_2\ud V+\int_M 2h^{-2\alpha}e^{-2\varphi^0_2}
\nabla \varphi^0_1 \nabla k^N_1\partial_t k^N_2\ud V\\
&-2\int_M h^{-2\alpha}e^{-2\varphi^0_2} \partial_tk^N_1\nabla\varphi^0_1\nabla k^N_2\ud V +\int_M h^{-2\alpha}e^{-2\varphi^0_2} f_1 \partial_t k^N_1+ f_2 \partial_tk^N_2\ud V.
  \end{aligned}\label{33}
  \end{equation*}

By the assumption of $\varphi_1^0$ and Young's inequality, 
\iffalse
Furthermore, for each $\delta>0$, by it holds that
\begin{equation*}
\begin{aligned}
	-2\int_M h^{-2\alpha}e^{-2\varphi^0_2}|\nabla \varphi^0_1|^2k^N_2\partial_t k^N_2\ud V &\leq \delta\int_M |\partial_t k^N_2|^2dV+C\int_M | k^N_2|^2\ud V,\\
\int_M 2h^{-2\alpha}e^{-2\varphi^0_2}
\nabla \varphi^0_1 \nabla k^N_1\partial_t k^N_2\ud V &\leq \delta\int_M |\partial_t k^N_2|^2\ud V+C\int_M h^{-2\alpha}e^{-2\varphi^0_2} |\nabla k^N_1|^2\ud V,\\
-2\int_M h^{-2\alpha}e^{-2\varphi^0_2} \partial_tk^N_1\nabla\varphi^0_1\nabla k^N_2\ud V  &\leq \delta \int_M h^{-2\alpha}e^{-2\varphi^0_2} |\partial_tk^N_1|^2\ud V+C\int_M |\nabla k^N_2|^2\ud V,
  \end{aligned}\label{34}
  \end{equation*}
\begin{equation*}
\begin{aligned}
&\int_M h^{-2\alpha}e^{-2\varphi^0_2} f_1 \partial_t k^N_1\ud V+\int_M f_2 \partial_tk^N_2\ud V\\
\leq & \delta\int_M h^{-2\alpha-2}e^{-2\varphi^0_2} |\partial_t k^N_1|^2+ |\partial_tk^N_2|^2\ud V+C\int_M h^{-2\alpha+2}e^{-2\varphi^0_2}|f_1|^2 + |f_2|^2 \ud V.
  \end{aligned}\label{35}
  \end{equation*}
\fi
 we find
 \begin{equation*}
\begin{aligned}
&\int_M h^{-2\alpha}e^{-2\varphi^0_2} | \partial_tk^N_1|^2+|\partial_tk^N_2 |^2\ud V+\frac{1}{2}\partial_t\int_M (h^{-2\alpha}e^{-2\varphi^0_2}|\nabla k^N_1|^2+   |\nabla k^N_2|^2)\ud V\\
\leq & C\left(\int_M h^{-2\alpha}e^{-2\varphi^0_1} |\nabla k^N_1|^2+|\nabla k^N_2|^2+| k^N_2|^2+ |f_2|^2 \right.\\
&\left.+ h^{-2\alpha+2}e^{-2\varphi^0_2}|f_1|^2 + h^{-2\alpha-2}e^{-2\varphi^0_2} |\partial_tk^N_1|^2 \ud V\right).
  \end{aligned}\label{36}
  \end{equation*}
Integrating the above inequality in $[0,T]$ and by Remark \ref{rem1}, Lemma \ref{lem:2.4} and \eqref {cc}, we find that
\begin{equation}
	\begin{aligned}
		&\int_{Q_T}h^{-2\alpha}e^{-2\varphi^0_2}| \widetilde {k^N_1}|^2+|\widetilde {k^N_2} |^2\ud V\ud s \\&+\sup\limits_{0\leq t\leq T}\int_M h^{-2\alpha}e^{-2\varphi^0_2}|\nabla k^N_1|^2+  |\nabla k^N_2|^2\ud V(t)\\
		\leq &%\int_M (|\nabla v_1|^2+  h^{-4}e^{4\varphi^0_1} |\nabla v_2|^2)dV(0)%+\int_M 4h^{-4}e^{4\varphi^0_1}|\nabla \varphi^0_2|^2v_1^2 dV(0)
		C\left(\int_0^T\int_M h^{-2\alpha+2}e^{-2\varphi^0_2} (|f_1|^2 +|\widetilde {f_1}|^2 )+|f_2|^2+|\widetilde {f_2}|^2 \ud V\ud s\right.\\
		&\left.+\|h^{-\alpha}e^{-\varphi^0_1}f_1\|_{L^2(M)}^2(0)
		+\|f_2\|_{L^2(M)}^2(0)\right ).\\
	\end{aligned}\label{37}
\end{equation}

Taking subsequences $\{k^{N_i}_1\}_i$ and $\{k^{N_i}_2\}_i$ and passing to limits as $N_i\rightarrow \infty$, we deduce
\begin{equation*}
	\begin{aligned}
	k^{N_i}_1\rightarrow k_1 \ &\mbox{ in } L^\infty(0,T;H^1(M;h^{-\alpha})), \  \partial_t k^{N_i}_1\rightarrow\partial_t k_1 \ \mbox{ in } L^2(0,T;L^2(M; h^{-\alpha})),\\
		k^{N_i}_2\rightarrow k_2 \ & \mbox{ in } L^\infty(0,T;H^1(M)), \  \ \quad \partial_t k^{N_i}_2\rightarrow\partial_t k_2 \ \mbox{ in } L^2(0,T;L^2(M)).
	\end{aligned}\label{371}
\end{equation*}
 %Therefore, there exist subsequences of $\{\partial_tk_1^N \}_{N=1}^{\infty}$ and $\{h^{-3}e^{-\varphi^0_1}\partial_tk_2^N \}_{N=1}^{\infty}$ converging weakly to  $\partial_tk_1$ and $h^{-3}e^{-\varphi^0_1}\partial_tk_2 $ in $L^2(Q_T)$. For any $t\in[0,T]$,  there exist subsequences of $\{\nabla k_1^N \}_{N=1}^{\infty}$ and $\{h^{-3}e^{-\varphi^0_1}\nabla k_2^N \}_{N=1}^{\infty}$ converging weakly to  $\nabla k_1$ and $h^{-3}e^{-\varphi^0_1}\nabla k_2 $ in $L^2(M)$.

Hence we have the desired estimate.
\iffalse
\begin{equation*}
\begin{aligned}
		&\int_0^T\int_M| \partial_tk_1|^2+h^{-6}e^{-2\varphi^0_1}|\partial_tk_2 |^2\ud V\ud s +\sup\limits_{0\leqslant t\leqslant T}\int_M(|\nabla k_1|^2+  h^{-6}e^{-2\varphi^0_1} |\nabla k_2|^2)\ud V(t)\\
&\leq 
C(\int_0^T\int_M |f_1|^2 +|\partial_t  f_1|^2 +h^{-4}e^{-2\varphi^0_1}(|f_2|^2+|\partial_t  f_2|^2) \ud V\ud s\\
&+\|f_1\|_{L^2(M)}^2(0)+\|h^{-3}e^{-2\varphi^0_1}f_2\|_{L^2(M)}^2(0)).
  \end{aligned}\label{38}
 \end{equation*}
 \fi
\end{proof}
\begin{rem}\label{rem2-10} Take $\varepsilon$ such that $0<\varepsilon<\frac{1}{2}(\alpha-1)$. If the initial function
	 \[(\varphi_1^0,\varphi_2^0)\in C^{[2\alpha-2\varepsilon],2\alpha-2\varepsilon- [2\alpha-2\varepsilon]}(M)\times \{u\in  C^{[2\alpha-2\varepsilon],2\alpha-2\varepsilon- [2\alpha-2\varepsilon]}(M)| u=0 \mbox{ on } \Gamma\},\] then there exists a constant $C$ depending on $M$, $T$ and $\varphi^0$ such that
		\begin{equation*}
		\begin{aligned}
			&\sup\limits_{0\leqslant t \leqslant T}(\|\rho^{-\alpha+1} k_1\|_{H^1(M)}+\| k_2\|^2_{H^1(M)})\\
			&+\|\rho^{-\alpha+1}\partial_tk_1\|_{L^2(0,T;L^2(M))}+\| \partial_tk_2\|_{L^2(0,T;L^2(M))}\\
			\leq &C(\|\rho^{-\alpha+1}f_1\|_{{L^2(0,T;L^2(M))}}+\| f_2\|_{{L^2(0,T;L^2(M))}}.
		\end{aligned}
	\end{equation*}
 \end{rem}
\begin{lem}\label{lem:2.10}
	If $(k_1, k_2)$ is a solution of (\ref{8'}), then there exists a constant $C$ depending on $M$, $T$ and $\varphi^0$ such that
	\begin{equation*}
		\begin{aligned}
			&\sup\limits_{0\leqslant t\leqslant T}  ( \|\rho^{-\alpha+1}\partial_t{k}_1(t)\|_{H^1{(M)}}+	\|\partial_t{k}_2(t)\|_{H^1{(M)}})+\|\rho^{-\alpha+1}\partial_tk_1\|_{L^2(0,T;H^2(M))}\\
			&+\|\partial_t k_2\|_{L^2(0,T;H^2(M))}+\|\rho^{-\alpha+1}\partial^2_t {k}_1\|_{L^2(0,T;L^2{(M)})}+	\|\partial^2_t {k}_2\|_{L^2(0,T;L^2{(M)})}\\
			\leq &C(\|\rho^{-\alpha+1}f_1\|_{{L^2(0,T;L^2(M))}}+\|\rho^{-\alpha+1}\partial_t f_1\|_{{L^2(0,T;L^2(M))}}+\|\rho^{-\alpha+1}f_1(0)\|_{H^1(M)}^2\\
			&+\| f_2\|_{{L^2(0,T;L^2(M))}}+\| \partial_t f_2\|_{{L^2(0,T;L^2(M))}}
			+\|f_2(0)\|_{H^1(M)}^2).
		\end{aligned}
	\end{equation*}
\end{lem}
\begin{proof} 
	
	By equation (\ref{8'}) and direct calculation, we have
	\begin{equation}
		\begin{aligned}\label{g1}
			%  \Delta(\rho^{\varepsilon}k_1)&=\rho^{\varepsilon}(\partial_tk_1+2h^{-6}e^{-2\varphi^0_1}|\nabla \varphi^0_2|^2k_1-2h^{-6}e^{-2\varphi^0_1}\nabla \varphi^0_2\nabla k_2-f_1)+2\nabla(\rho^{\varepsilon})\nabla k_1+\Delta(\rho^{\varepsilon})k_1,\\
			%\partial_t k_2-	\Delta k_2=&-2h^{-2\alpha}e^{-2\varphi^0_2}|\nabla \varphi^0_1|^2k_2+2h^{-2\alpha}e^{-2\varphi^0_2}\nabla \varphi^0_1\nabla k_1+f_2,\\
			&\partial_t(\rho^{-\alpha+1}k_1)-\Delta(\rho^{-\alpha+1}k_1)
			=-2\nabla(\rho^{-\alpha+1})\nabla k_1-\Delta(\rho^{-\alpha+1})k_1 \\
			&+\rho^{-\alpha+1}(-2(\nabla \varphi^0_2+\frac{\alpha\nabla h}{h})\nabla k_1-2\nabla\varphi^0_1\nabla k_2+f_1),\\
			&\partial_t{k}_2-\Delta {k}_2=-2h^{-2\alpha}e^{-2\varphi^0_2}|\nabla \varphi^0_1|^2{k}_2+2h^{-2\alpha}e^{-2\varphi^0_2}\nabla \varphi^0_1\nabla {k}_1+ f_2.\\
		\end{aligned}
	\end{equation}
	Differentiating equations in $\eqref{g1}$ with respect to $t$, we can get
	\begin{equation}\label{g2}
		\begin{aligned}
			&\partial_t(\rho^{-\alpha+1}\widetilde{k_1})-\Delta(\rho^{-\alpha+1}\widetilde{k_1})=-2\nabla(\rho^{-\alpha+1})\nabla \widetilde{k_1}-\Delta(\rho^{-\alpha+1})\widetilde{k}_1\\
			&+\rho^{-\alpha+1}(-2(\nabla \varphi^0_2+\frac{\alpha\nabla h}{h})\nabla \widetilde{k_1}-2\nabla\varphi^0_1\nabla \widetilde{k_2}+\partial_t f_1),\\
			&\partial_t\widetilde{k_2}-\Delta \widetilde{k_2}=-2h^{-2\alpha}e^{-2\varphi^0_2}|\nabla \varphi^0_1|^2\widetilde{k_2}+2h^{-2\alpha}e^{-2\varphi^0_2}\nabla \varphi^0_1\nabla \widetilde{k_1}+
			\partial_t f_2.\\
		\end{aligned}
	\end{equation}
	 Notice that 
	\begin{equation}
		\begin{aligned}\label{2.23}
			\rho^{-\alpha+1}\widetilde{k_1}(0)=	\rho^{-\alpha+1}\partial_t k_1(0)=\rho^{-\alpha+1}f_1(0),\ 
			\widetilde{k_2}(0)=	\partial_t k_2(0)=f_2(0).\\
		\end{aligned}
	\end{equation}	
The equations \eqref{g2} with the initial function (\ref{2.23}) have weak solutions, since the right-hand sides of the equations belong to $L^2(Q_T)$ and $\rho^{-\alpha+1}\widetilde{k}_1(0), \widetilde{k}_2(0)\in L^2(M)$. By the regularity of parabolic equations, in view of Theorem 5 of Chapter 7.1 of \cite{C}, we have
\begin{equation}\label{g3}
	\begin{aligned}
		&\sup\limits_{0\leq t\leq T}  ( \|\rho^{-\alpha+1}\widetilde{k_1}(t)\|_{H^1{(M)}}+\|\widetilde{k_2}(t)\|_{H^1{(M)}})
		+\|\rho^{-\alpha+1}\widetilde{k_1}(t)\|_{L^2(0,T;H^2{(M)})}\\
		&+	\|\widetilde{k_2}(t)\|_{L^2(0,T;H^2{(M)})}+\|\rho^{-\alpha+1}\partial_t \widetilde{k_1}(t)\|_{L^2(0,T;L^2{(M)})}+	\|\partial_t \widetilde{k_2}(t)\|_{L^2(0,T;L^2{(M)})}\\
		\leq & C\left(\|\rho^{-\alpha+1}\partial_t\widetilde{k_1}\|_{L^2(0,T;L^2(M))}+\|\nabla^2(\rho^{-\alpha+1}\widetilde{k_1})\|_{L^2(0,T;L^2(M))}+\|f_2(0)\|_{H^1(M)}\right.\\
		&\left.+\|\partial_t\widetilde{k}_2\|_{L^2(0,T;L^2(M))}+\|	\nabla^2 \widetilde{k_2}\|_{L^2(0,T;L^2(M))}+\|\rho^{-\alpha+1}f_1(0)\|_{H^1(M)}\right).
		\end{aligned}	
	\end{equation}
	By combining Lemma \ref{lem:W112}, we can complete the proof.
\end{proof}

\begin{lem}\label{lem:W212}
If $(k_1, k_2)$ is a solution of (\ref{8'}), then there exists a constant $C$ depending on $M$, $T$ and $\varphi^0$ such that
	\begin{equation*}
	\begin{aligned}
	&\|\rho^{-\alpha+3}k_1\|_{L^2(0,T;H^4(M))}+\| k_2\|_{L^2(0,T;H^4(M))}\\
			\leq &C(\|\rho^{-\alpha+3}f_1\|_{{L^2(0,T;H^2(M))}}+\|\rho^{-\alpha+1}\partial_t f_1\|_{{L^2(0,T;L^2(M))}}+\|\rho^{-\alpha+1}f_1(0)\|_{H^1(M)}\\
		&+\| f_2\|_{{L^2(0,T;H^2(M))}}+\| \partial_t f_2\|_{{L^2(0,T;L^2(M))}}
		+\|f_2(0)\|_{H^1(M)}).
	\end{aligned}
\end{equation*}
\end{lem}

\begin{proof}

According to Proposition \ref{prop:prop1}, Lemma \ref{lem:2.4} and Lemma \ref{lem:W112}, we find
that the equations
\begin{equation*}
	\begin{aligned}
		%  \Delta(\rho^{\varepsilon}k_1)&=\rho^{\varepsilon}(\partial_tk_1+2h^{-6}e^{-2\varphi^0_1}|\nabla \varphi^0_2|^2k_1-2h^{-6}e^{-2\varphi^0_1}\nabla \varphi^0_2\nabla k_2-f_1)+2\nabla(\rho^{\varepsilon})\nabla k_1+\Delta(\rho^{\varepsilon})k_1,\\
		\Delta k_2=&\partial_tk_2+2h^{-2\alpha}e^{-2\varphi^0_2}|\nabla \varphi^0_1|^2k_2-2h^{-2\alpha}e^{-2\varphi^0_2}\nabla \varphi^0_1\nabla k_1-
		f_2,\\
		\Delta(\rho^{-\alpha+1}k_1)=&\rho^{-\alpha+1}(\partial_tk_1+2(\nabla \varphi^0_2+\frac{\alpha\nabla h}{h})\nabla k_1+2\nabla\varphi^0_1\nabla k_2-f_1)\\
		&+2\nabla(\rho^{-\alpha+1})\nabla k_1+\Delta(\rho^{-\alpha+1})k_1,\\
		\Delta(\rho^{-\alpha+2}k_1)=&\rho^{-\alpha+2}(\partial_tk_1+2(\nabla \varphi^0_2+\frac{\alpha\nabla h}{h})\nabla k_1+2\nabla\varphi^0_1\nabla k_2-f_1)\\
		&+2\nabla(\rho^{-\alpha+2})\nabla k_1+\Delta(\rho^{-\alpha+2})k_1,\\
		\Delta (\rho^{-\alpha+3}k_1)=&\rho^{-\alpha+3}(\partial_tk_1+2(\nabla\varphi^0_2+\frac{\alpha\nabla h}{h})\nabla k_1+2\nabla\varphi^0_1\nabla k_2-f_1)\\
		&+2\nabla(\rho^{-\alpha+3})\nabla k_1+\Delta(\rho^{-\alpha+3}) k_1
	\end{aligned}
\end{equation*}
all possess weak solutions, since the right-hand sides of the equations belong to $L^2(M)$.
By the regularity of elliptic equations, in view of Theorem 1 and Theorem 2 of Chapter 6.3 in \cite{C}, we have
\begin{equation*}
	\begin{aligned}
		\|k_2\|^2_{H^2(M)}&\leq C( \|k_2\|^2_{L^2(M)}+\|\Delta k_2\|^2_{L^2(M)}),\\
		\|\rho^{-\alpha+1}k_1\|^2_{H^2(M)}&\leq C( \|\rho^{-\alpha+1}k_1\|^2_{L^2(M)}+\|\Delta(\rho^{-\alpha+2}k_1)\|^2_{L^2(M)}),	\end{aligned}
\end{equation*}
\begin{equation*}
	\begin{aligned}
			\|k_2\|^2_{H^3(M)}&\leq C( \|k_2\|^2_{L^2(M)}+\|\Delta k_2\|^2_{H^1(M)}),\\
		\|\rho^{-\alpha+2}k_1\|^2_{H^3(M)}&\leq C( \|\rho^{-\alpha+2}k_1\|^2_{L^2(M)}+\|\Delta(\rho^{-\alpha+2}k_1)\|^2_{H^1(M)}).\\
		\|k_2\|^2_{H^4(M)}&\leq C( \|k_2\|^2_{L^2(M)}+\|\Delta k_2\|^2_{H^2(M)}),\\
		\|\rho^{-\alpha+3}k_1\|^2_{H^4(M)}&\leq C( \|\rho^{-\alpha+3}k_1\|^2_{L^2(M)}+\|\Delta (\rho^{-\alpha+3}k_1)\|^2_{H^2(M)}).
	\end{aligned}
\end{equation*}
Integrating from $0$ to $T$ and utilizing the estimate from Lemma \ref{lem:W112} and Lemma \ref{lem:2.10}, we have completed the proof.
\end{proof}
\begin{rem}\label{rem:3-1}Take $\varepsilon$ such that $0<\varepsilon<\frac{1}{2}(\alpha-1)$, if the initial function	  
	   $$(\varphi_1^0,\varphi_2^0)\in C^{[2\alpha-2\varepsilon],2\alpha-2\varepsilon- [2\alpha-2\varepsilon]}(M)\times\{ u\in  C^{[2\alpha-2\varepsilon],2\alpha-2\varepsilon- [2\alpha-2\varepsilon]}(M)| u=0 \mbox{ on } \Gamma\},$$
	   then there exists a constant $C$ depending on $M$, $T$ and $\varphi^0$ such that
		\begin{equation*}
		\begin{aligned}
			&\|\rho^{-\alpha+1}k_1\|_{L^2(0,T;H^2(M))}+\| k_2\|_{L^2(0,T;H^2(M))}\\
			\leq &C(\|\rho^{-\alpha+1}f_1\|_{{L^2(0,T;L^2(M))}}+\| f_2\|_{{L^2(0,T;L^2(M))}}).
		\end{aligned}
	\end{equation*}	   
  \end{rem}

\begin{thm} \label{thm:W422}
For any $t\in [0,T] $, if $(k_1,k_2)$ is a solution of \eqref{8'}, 
 then $ k_1\in W^{4,2}_{2}(Q_T;\rho^{ -\alpha}) $
and $ k_2\in W^{4,2}_{2}(Q_T)$. Futhermore,
\begin{equation*}
	\begin{aligned}
		&\|k_1\|^2_{W^{4,2}_{2}(Q_T;\rho^{-\alpha})}+\|k_2\|^2_{W^{4,2}_{2}(Q_T)}\\
		\leq &C\left(\|f_1\|^2_{W^{2,1}_{2}(Q_T;\rho^{-\alpha})}+\|f_2\|^2_{W^{2,1}_{2}(Q_T)}
		+\|\rho^{-\alpha+1}f_1(0)\|_{H^1(M)}+\|f_2(0)\|_{H^1(M)}\right).
	\end{aligned}
\end{equation*}
 
 Moreover, according to the Sobolev embedding theorem, for $i=1,2,3$, we have $$
\rho^{-\alpha+3}\nabla_{x_i} k_1, \nabla_{x_i} k_2 \in C^{\frac{1}{2},\frac{1}{4}}(Q_T), $$ 
and $$\|\nabla(\rho^{-\alpha+3} k_1)\|_{C^{\frac{1}{2},\frac{1}{4}}(Q_T)}+\|\nabla  k_2\|_{C^{\frac{1}{2},\frac{1}{4}}(Q_T)}\leq C \left(\| \rho^{-\alpha+3} k_1\|_{W^{4,2}_{2}(Q_T)}+\|  k_2\|_{W^{4,2}_{2}(Q_T)}\right).$$
\end{thm}
Since $f=(f_1,f_2) \in  \mathcal{C}
\times\mathcal{D} $, we have $k_1,k_2\in C^{2+\beta,1+\frac{\beta}{2}}({M\backslash \Gamma\times [0,T]})$. For any fixed  $t\in[0,T]$, the regularity theorem of elliptic equations and Lemma \ref{lem:W112} yield
	\begin{equation}
	\begin{aligned}\label{2.25}
		&\|\rho^{-\alpha+1}k_1(t)\|_{H^2(M)}\\
		\leq &C\left(\|\rho^{-\alpha+1}f_1\|_{{L^2(0,T;L^2(M))}}+\|\rho^{-\alpha+1}\partial_t f_1\|_{{L^2(0,T;L^2(M))}}+\|\rho^{-\alpha}f_1(0)\|_{L^2(M)}^2\right.\\
		&\left.+\| f_2\|_{{L^2(0,T;L^2(M))}}+\| \partial_t f_2\|_{{L^2(0,T;L^2(M))}}
		+\|f_2(0)\|_{L^2(M)}^2\right).
	\end{aligned}
\end{equation}	 
By the Sobolev embedding theorem, we have
\begin{equation*}
 \begin{aligned}
	\|\rho^{-\alpha+1}k_1(t)\|_{C^{0,\frac{1}{2}}(M)}	\leq&C\|\rho^{-\alpha+1}k_1(t)\|_{H^2(M)}.\\
 \end{aligned}
\end{equation*}	   
In conclusion, $|k_1(x,t)|\leq C\rho^{\alpha-1},  \forall(x,t)\in M\times[0,t]$, where $C$ is the right hand side of (\ref{2.25}).

Next, we will give the weighted Schauder estimates using the method of classical theory for the linear parabolic equations.

\begin{lem}\label{key}
  For any $x_0\in M\backslash \Gamma$ and $2<\gamma<2\alpha$, if $B_1(x_0)\cap \Gamma\neq \varnothing$, let $v\in C^{2+\beta,1+\frac{\beta}{2}}(( B_1(x_0)\setminus \Gamma)\times [0,T])$ and
$|\nabla v|\leq C^{'}\rho^{-\delta}$ with $\delta<1$, $f\in C^{\beta,\frac{\beta}{2}}(\overline{Q_T})$ with $\max\limits_{B_1(x_0)\times[0,T]}(\rho^{-\gamma+2}|f|)\leq C^{'}$. Then  there is a uniform constant $C_{\gamma}$ depending only on
$C^{'},\delta,\gamma$, such that if $u$ is a
solution of
   \begin{equation*}
    \left\{ \begin{aligned}
        \partial_t u-\Delta u&=2\nabla
u(\frac{\alpha\nabla \rho}{\rho}+\nabla v)+f   &\mbox{ in} \ B_1(x_0)\times[0,T]\\
        u&=0 & \mbox{ on} \ B_1(x_0)\times \{t=0 \}
     \end{aligned}\right.
   \end{equation*}
   and $|u|_{C^0(\partial B_{\frac{2}{3}}(x_0)\times
[0,T] )}\leq C^{'}, u|_{\Gamma\times [0,T]}
\equiv 0$, then
     \begin{equation*}
     \begin{aligned}
        |u(x,t)|\leq C_{\gamma}\rho(x)^{\gamma}
, \enspace 
\mbox{ in } B_1(x_0)\times[0,T].
     \end{aligned}
   \end{equation*}
\end{lem}
\begin{proof}
It's similar to the proof of Lemma 4.4 of \cite{LT} by constructing the barrier and using maximum principle.
For any $\overline{x}_0\in \Gamma \cap B_1(x_0) $, we choose $\rho^\gamma +\rho^{\gamma-\delta}r_{\overline{x}_0}(x)^2$ as upper barrier for $u$, where $\delta \leq \gamma <2\alpha $, $0<\delta<1$ and $r_{\overline{x}_0}(x)$ is the projection onto $\Gamma$ of the distance between $\overline{x}_0$ and $x$ for any fixed $\overline{x}_0$ in $B_1(x_0)$.
By direct computations, we can see that there is a neighborhood $U$ of $\Gamma\cap B_1(x_0)$ such that
$$(\partial _t-\Delta-2(\frac{\alpha\nabla \rho}{\rho}+\nabla v)\nabla)(\rho^\gamma +\rho^{\gamma-\delta}r_{\overline{x}_0}(x)^2)\geq C_0\rho^{\gamma-2} \quad  \mbox{in} \ U\times [0,T].
$$

Assume $\mu$ is a small positive constant, then
\begin{equation}
\begin{aligned}
&(\partial _t-\Delta-2(\frac{\alpha\nabla \rho}{\rho}+\nabla v)\nabla)[(\rho^\gamma +\rho^{\gamma-\delta}r_{\overline{x}_0}(x)^2)\|\rho^{2-\gamma}f\|_{C^0(B_1(x_0)\times[0,T])}\pm\mu u]\\ 
\geq & (C_0-\mu) \|\rho^{2-\gamma}f\|_{C^0(B_1(x_0)\times[o,T])}\rho^{\gamma-2} 
%\\ \geq &(C_0-\mu\|\rho^{2-\gamma}f\|_{C^0(Q_T)})\rho^{\gamma-2}
\geq0 \quad  \mbox{in} \ U\times [0,T].	
\end{aligned}
\end{equation}
Now, fix $\widetilde{\varepsilon}$ such that $l\widetilde{\varepsilon}=\gamma$ for some $l>0$. Initially, we set $\delta=\gamma=\widetilde{\varepsilon}$.
Using the maximum principle and comparing $u$ and $\rho^\gamma +r_{\overline{x}_0}(x)^2$ for any $\overline{x}_0$ in $\Gamma\cap B_1(x_0)$, we obtain
 $|u(x,t)|\leq C\|\rho^{-\gamma+2}f\|_{C^0(B_1(x_0)\times[0,T])}\rho^{\widetilde{\varepsilon}}\leq C_\gamma \rho^{\widetilde{\varepsilon}}$. Then, we repeat this argument with $\gamma=j\widetilde{\varepsilon}, \delta=\widetilde{\varepsilon}$ for $2\leqslant j\leqslant l$ to complete the proof.
\end{proof}

Note that the second equation of (\ref{8'}) is equivalent to
\begin{equation}
   \partial_t k_1-\Delta k_1 +2 (\frac{\alpha\nabla h}{h}+\nabla \varphi_2^0) \nabla k_1 =f_1-2\nabla\varphi_1^0 \nabla k_2.\label {87}
\end{equation}
Denote $F_1:=f_1-2\nabla \varphi_1^0 \nabla k_2$, we apply  Lemma \ref{key} to $k_1$,
$$\|\rho^{-\gamma+2}F_1\|_{C^{0}(Q_T)}\leq C(\| k_2\|_{W^{4,2}_2(Q_T)}+ \|\rho^{-\gamma+2}f_1\|_{C^{0}(Q_T)} \leqslant C^{'},$$
where $C^{'}=C^{'}(M,f_1,f_2)$.
 We obtain
   \begin{equation}\label{k2}
   \begin{aligned}
     |k_1(x, t)|\leqslant C \max\limits_{B_1(x)}(\rho^{-\gamma+2}|f_1|+|\nabla k_2|)\rho(x)^{\gamma} \ \mbox{ for } (x, t) \in B_1(x)
\times[0,T].
   \end{aligned} 
\end{equation}

Next, we aim to improve the regularity of $k_1,k_2$. %To achieve this,  we will combine the Schauder theory for linear parabolic equation with the method of weighted H\"older estimate in the elliptic equation, as presented in Chapter 4 of \cite{GT}. This will allow us to derive weighted Schauder estimates for the linear parabolic equation.
Set $g_1=f_1-2(\frac{\alpha\nabla h}{h}+\nabla\varphi^0_2)\nabla k_1-2\nabla\varphi^0_1\nabla k_2$, then
\begin{equation}
\begin{aligned}
  \partial_t k_1-\Delta k_1=g_1. \label{92}
  \end{aligned}
  \end{equation}
Take $X=(x_0,t_0)\in Q_T$, $R=\frac{1}{6}\rho(x_0)$ and set $Q_R(X)=B_R(x_0)\times (t_0,t_0+R^2]$, $\partial_pQ_{4R}(X)=\left(B_R(x_0)\times \{t=t_0\}\right)\cup \left(\partial B_R(x_0)\times (t_0,t_0+R^2]\right)$. 

 Consider the problem
\begin{equation}
	\left\{
	\begin{aligned}
		\partial_t w-\Delta w&=g_1  \  \mbox{ in }Q_{4R}(X)\\
		w&=0  \  \ \mbox{ on } \partial_pQ_{4R}(X).
	\end{aligned}
	\right.\label{94}
\end{equation}
There exists a solution to this problem since $g_1\in L^2(Q_T)$. Set $u=k_1-w$, then $u$ satisfies
\begin{equation}
 \partial_t u-\Delta u=0 \ \mbox{ in } Q_{4R}(X).\label {93}\\
  \end{equation}

By the properties for parabolic equations, we can get the lemmas.
\begin{lem}\label{lem:u} (Estimate of u).
  For u in  (\ref{93}), we have for any $\beta\in (0,1)$,
\begin{equation*}
   \begin{aligned}
     &R^{2+\beta}(\|\nabla^2 u\|_{C^\beta(Q_R(X))}+\|\partial_tu\|_{C^\beta(Q_R(X))})+R^{2}\|\nabla^2 u\|_{C^{0}(Q_{R}(X))} + R^{2}\|\partial_tu\|_{C^{0}(Q_{R}(X))} \\
     & + R\|\nabla u\|_{C^{0}(Q_{R}(X))}\leq  C\|u\|_{C^{0}(Q_{4R}(X))}.
   \end{aligned}\label {95}
\end{equation*}
\end{lem}
\begin{lem}\label{lem:w} (Estimate of w).
  For w in  (\ref{94}), we have for any $\beta\in (0,1)$,
\begin{equation*}
\begin{aligned}
&R\|\nabla w\|_{C^0({Q_{R}(X)})}+R^2\left(\|\nabla^2w\|_{C^0({Q_{R}(X)})}+\|\partial_t w\|_{C^0({Q_{R}(X)})}\right)+R^{2+\beta}\left(\|\nabla^2w\|_{C^\beta(Q_{R}(X))}\right.\\
&+ \left. \|\partial_tw\|_{C^\beta(Q_{R}(X))} \right)\leq  C \left( \|w\|_{C^0(Q_{4R}(X))}+  R^{2}\|g_1\|_{C^0({Q_{4{R}}(X)})}+R^{2+\beta}\|g_1\|_{C^\beta(Q_{4R}(X))}\right),
  \end{aligned}\label {106}
  \end{equation*}
where $C=C(\beta)$.
\end{lem}
Once we have these two lemmas, we can prove the weighted Schauder estimates.
\begin{thm}\label{thm:k2}
 There is a uniform constant C such that for any $0<\beta<\min\{2\alpha-2,1\}$ and $2+\beta<\gamma<2\alpha$,
 \begin{equation*}
\begin{aligned}
&\sup\limits_{X\in Q_T}\rho_X^{1-\gamma}|\nabla k_1(X)|+\sup\limits_{X\in Q_T}\rho_X^{2-\gamma}\left(|\nabla^2k_1(X)|+|\partial_t k_1(X)|\right)\\
&+\sup\limits_{X,Y\in Q_T}\rho_{X,Y}^{2+\beta-\gamma}\left(\frac{|\nabla^2k_1(X)-\nabla^2k_1(Y)|}{\delta(X,Y)^{\beta}}+\frac{|\partial_tk_1(X)-\partial_tk_1(Y)|}{\delta(X,Y)^{\beta}}\right)\\
\leq& C  \left( \sup\limits_{Q_{T}}  \rho_{X}^{-\gamma}|k_1(X)|+  \sup\limits_{Q_{T}}  \rho_{X}^{2-
\gamma} |g_1(X)| +  \sup\limits_{Q_{T}}\rho_{Y,Z}^{2+\beta-\gamma} \frac{|g_1(Y)-g_1(Z)|}{\delta(Y,Z)^\beta}\right),
  \end{aligned}\label{115}
  \end{equation*}
where $C=C(M,\gamma,\beta)$.
 \end{thm}
\begin{proof}
Set $G_1:=\|g_1\|_{C^0(Q_{4R}(X))}$. Since $w$ satisfies the equations (\ref{94}), for any $t\in[t_0,t_0+16R^2)$
   \begin{equation*}
     \left\{
     \begin{aligned}
        \partial_t (G_1(t-t_0)\pm w)-\Delta(G_1(t-t_0)\pm w)&=G_1\pm g_1 \geq 0 & \mbox{ in }Q_{4R}(X)\\
         G_1(t-t_0)\pm w &\geq 0 \ & \mbox{ on }\partial_pQ_{4R}(X).\\
     \end{aligned}
     \right.
   \end{equation*}  
By the maximum principle, we have $
\min\limits_{ Q_{4R}(X)}G_1(t-t_0)\pm w=\min\limits_{ \partial_p Q_{4R}(X)}G_1(t-t_0)\pm w\geq 0, |w|\leq G_1(t-t_0)\leq 16G_1R^2$. 
That is 
\begin{equation}
    \|w\|_{C^0(Q_{4R}(X))}\leq CR^2\|g_1\|_{C^0(Q_{4R}(X))}.\label{115+}
\end{equation}
Notice that $k_1=u+w$, by Lemma \ref{lem:u} and Lemma \ref{lem:w}, we have
\begin{equation}
\begin{aligned}
&R\|\nabla k_1\|_{C^0({Q_{R}(X)})}+R^2(\|\nabla^2k_1\|_{C^0({Q_{R}(X)})}+\|\partial_t k_1\|_{C^0({Q_{R}(X)})})\\
&+R^{2+\beta}([\nabla^2k_1]_{\beta,Q_{R}(X)}+[\partial_tk_1]_{\beta,Q_{R}(X)} )\\
\leq &C ( \|w\|_{C^0(Q_{4R}(X))}+  \|k_1\|_{C^0(Q_{4R}(X))}+ R^{2}\|g_1\|_{C^0({Q_{4{R}}(X) })}+R^{2+\beta}[g_1]_{\beta,Q_{4R}(X)})\\
\leq & C (  \|k_1\|_{C^0(Q_{4R}(X))}+ R^{2}\|g_1\|_{C^0({Q_{4{R}}(X) })}+R^{2+\beta} [g_1]_{\beta,Q_{4R}(X)}).
  \end{aligned}\label{116}
  \end{equation}
For any $X=(x,t)\in Q_T$, recall that
$R=\frac{1}{6}\rho_X, \rho_{Y,Z}=max\{\rho_{Y},\rho_{Z}\}$. Following (\ref{116}), for any $2+\beta<\gamma<2\alpha$, we have
 \begin{equation}
\begin{aligned}
&\rho_X^{1-\gamma}|\nabla k_1(X)|+\rho_X^{2-\gamma}|\nabla^2k_1(X)|\\
\leq&(6R)^{1-\gamma}\|\nabla k_1\|_{C^0({Q_{R}(X)})}+(6R)^{2-\gamma}\|\nabla^2k_1\|_{C^0({Q_{R}(X)})}\\
\leq & C \left(R^{-\gamma} \|k_1\|_{C^0(Q_{4R}(X))}+R^{2-\gamma}\|g_1\|_{C^0({Q_{4{R}}(X) })}+R^{2+\beta-\gamma}[g_1]_{\beta,Q_{4R}(X)}\right)\\
\leq & C \left(R^{-\gamma} \sup\limits_{Y\in Q_{4R}(X)}\rho_Y^{\gamma}  \sup\limits_{Y\in Q_{4R}(X)}\rho_Y^{-\gamma} |k_1(Y)|\right.\\&+  R^{2-\gamma} \sup\limits_{Y \in Q_{4R}(X)}\rho_Y^{\gamma-2}\sup\limits_{Y\in Q_{4R}(X)}\rho_Y^{2-\gamma} |g_1(Y)|\\
&\left.+R^{2+\beta-\gamma} \sup\limits_{Y,Z\in Q_{4R}(X)} \rho_{Y,Z}^{\gamma-2-\beta}\sup\limits_{Y,Z\in Q_{4R}(X)}  \rho_{Y,Z}^{2+\beta-\gamma}\frac{|g_1(Y)-g_1(Z)|}{\delta(Y,Z)^\beta}  \right)   \\
\leq &C \left( \sup\limits_{X\in Q_{T}}  \rho_{X}^{-\gamma} |k_1(X)|+  \sup\limits_{X\in Q_{T}}\rho_{X}^{2-\gamma} |g_1(X)| +  \sup\limits_{Y,Z\in Q_{T}}  \rho_{Y,Z}^{2+\beta-\gamma} \frac{|g_1(Y)-g_1(Z)|}{\delta(Y,Z)^\beta}\right).
  \end{aligned}\label {117}
  \end{equation}
Without loss of generality , let $\rho_{X}\geq \rho_{Y} $, then $\rho_{X,Y}=6R$. In view of \eqref{116} and \eqref{117}, we find for any $2+\beta<\gamma<2\alpha$,
\begin{equation}
   \begin{aligned}
   &\rho_{X,Y}^{2+\beta-\gamma}\frac{|\nabla^2k_1(X)-\nabla^2k_1(Y)|}{\delta(X,Y)^{\beta}}\\
   \leq & (6R)^{2+\beta-\gamma}[\nabla^2k_1]_{\beta,Q_{R}(X)}+6^{2+\beta-\gamma}R^{2-\gamma}\left(|\nabla^2k_1(X)|+|\nabla^2k_1(Y)|\right)\\
\leq & C \left(R^{-\gamma} \|k_1\|_{C^0(Q_{4R}(X))}+R^{2-\gamma}\|g_1\|_{C^0({Q_{4{R}}(X) })}+R^{2+\beta-\gamma}[g_1]_{\beta,Q_{4R}(X)}\right)\\
\leq& C \left( \sup\limits_{X\in Q_{T}}  \rho_{X}^{-\gamma} |k_1(X)|+  \sup\limits_{X\in Q_{T}}\rho_{X}^{2-\gamma} |g_1(X)| +  \sup\limits_{Y,Z
\in Q_{T}}  \rho_{Y,Z}^{2+\beta-\gamma} \frac{|g_1(Y)-g_1(Z)|}{\delta(Y,Z)^\beta}\right).
   \end{aligned}\label {118}
  \end{equation}
Note that $\partial_t k_1-\Delta k_1=g_1$, we have
\begin{equation*}
\begin{aligned}
&\sup\limits_{X\in Q_T}\rho_{X}^{2-\gamma}|\partial_t k_1(X)|+\sup\limits_{X,Y\in Q_T}\rho_{X,Y}^{2+\beta-\gamma}\frac{|\partial_tk_1(X)-\partial_tk_1(Y)|}{\delta(X,Y)^{\beta}}\\
\leq &\sup\limits_{X\in Q_T}\rho_{X}^{2-\gamma}\left(|\Delta k_1(X)|+|g_1(X)|\right)\\&+\sup\limits_{X,Y\in Q_T}\rho_{X,Y}^{2+\beta-\gamma}\left(\frac{|\Delta k_1(X)-\Delta k_1(Y)|}{\delta(X,Y)^{\beta}}+\frac{|g_1(X)-g_1(Y)|}{\delta(X,Y)^{\beta}}\right)\\
\leq &C \left( \sup\limits_{Q_{T}}  \rho_{X}^{-\gamma}|k_1(X)|+  \sup\limits_{Q_{T}}  \rho_{X}^{2-\gamma} |g_1(X)| +  \sup\limits_{Q_{T}}\rho_{Y,Z}^{2+\beta-\gamma} \frac{|g_1(Y)-g_1(Z)|}{\delta(Y,Z)^\beta}\right).
  \end{aligned}\label{120}
  \end{equation*}
Combining with \eqref{117} and \eqref{118}, the lemma holds.
\end{proof}

\begin{lem}\label{Interpolation} (Interpolation
inequality).
For any $2<\gamma<2\alpha$, $\varepsilon>0$  and some constant $C=C(\varepsilon)$ we have
 \begin{equation}
\begin{aligned}
&\sup\limits_{Q_{T}}  \rho_{X}^{1-\gamma} |\nabla k_1(X)|+\sup\limits_{X,Y\in Q_{T}}  \rho_{X,Y}^{1-\gamma+\beta}\frac{|\nabla k_1(X)-\nabla k_1(Y)|}{\delta(X,Y)^\beta}\\
&\leq C\sup\limits_{Q_{T}}  \rho_{X}^{-\gamma} |k_1(X)|+\varepsilon \sup\limits_{X\in Q_T}\rho_{X}^{2-\gamma}|\nabla^2k_1(X)|,
  \end{aligned}\label{121}
  \end{equation}
where $C=C(\varepsilon,\gamma,\beta)$.
\end{lem}
\begin{proof}
By the interpolation inequality in \cite{LGM}, it's easy to obtain.
\end{proof}

\begin{thm}\label{k1k2C2}
There is an uniform constant $C$ such that for any $0<\beta<\min\{2\alpha-2,\frac{1}{2} \}, 2+\beta<\gamma<2\alpha$,
\begin{equation*}
\begin{aligned}
& \|k_2\|_{C^{2+\beta,1+\frac{\beta}{2}}(Q_T)}+ \sup\limits_{X\in Q_T}\rho_X^{1-\gamma}|\nabla k_1(X)|+\sup\limits_{X\in Q_T}\rho_X^{2-\gamma}(| \nabla^2k_1(X)|+|\partial_t k_1(X)|)\\ &+\sup\limits_{X,Y\in Q_T}\rho_{X,Y}^{2+\beta- \gamma}(\frac{|\nabla^2k_1(X)-\nabla^2k_1(Y)|}
{\delta(X,Y)^{\beta}}+\frac{|\partial_tk_1(X)-\partial_tk_1(Y)|}{\delta(X,Y)^{\beta}})\\
\leq& C \left( \|\nabla k_2\|_{C^{\beta,\frac{\beta}{2}}(Q_T)}+\sup\limits_{Q_{T}}  \rho_X^{-\gamma} |k_1(X)|+ \|f_2\|_{C^{\beta,\frac{\beta}{2}}(Q_T)}+\sup\limits_{Q_{T}}  \rho_X^{2-\gamma} |f_1(X)| \right.\\
&\left.+ \sup\limits_{Q_{T}}  \rho_{Y,Z}^{2+\beta-\gamma} \frac{|f_1(Y)-f_1(Z)|}{\delta(Y,Z)^\beta}+\|k_2\|_{W_2^{2,1}(Q_T)}\right),
  \end{aligned}\label{122}
  \end{equation*}
where $C=C(M,\varphi^0,\gamma,\beta)$.\end{thm}
\begin{proof}

Recall that $g_1=f_1-2(\frac{\alpha\nabla h}{h}+\nabla\varphi^0_2)\nabla k_1-2\nabla\varphi^0_1\nabla k_2$. According to $\varphi_1^0\in\mathcal{A}_0$ and Theorem \ref{thm:k2}, we just need
to estimate $\sup\limits_{Q_{T}}  \rho_{X}^{2-\gamma} |g_1(X)|$ and $\sup\limits_{Q_{T}}\rho_{X,Y}^{2+\beta-\gamma} \frac{|g_1(X)-g_1(Y)|}{\delta(X,Y)^\beta} $.
\begin{equation*}
\begin{aligned}
 \sup\limits_{Q_{T}}\rho_{X}^{2-\gamma} |g_1(X)|\leq C\left(  \sup\limits_{Q_{T}}\rho_{X}^{2-\gamma} |f_1(X)|+\sup\limits_{Q_{T}}\rho(X)^{1-\gamma}|\nabla k_1(X)|+ \| \rho\nabla k_2 \|_{C^0(Q_T)} \right),
  \end{aligned}\label {123}
  \end{equation*}
\begin{equation}
\begin{aligned}
 &\sup\limits_{Q_{T}}  \rho_{X,Y}^{2+\beta-\gamma}\frac{|g_1(X)-g_1(Y)|}{\delta(X,Y)^\beta}
 \leq\sup\limits_{Q_{T}} \rho_{X,Y}^{2+\beta-\gamma}\frac{|f_1(X)-f_1(Y)|}{\delta(X,Y)^\beta}\\
  &+\sup\limits_{Q_{T}} \rho_{X,Y}^{2+\beta-\gamma} \frac{|\xi\eta(X)-\xi \eta(Y)|}{\delta(X,Y)^\beta} +\sup\limits_{Q_{T}}\rho_{X,Y}^{2+\beta-\gamma}\frac{|2\nabla\varphi^0_1\nabla k_2(X)-2\nabla\varphi^0_1\nabla k_2(Y)|}{\delta(X,Y)^\beta}
  \end{aligned}\label {124}
  \end{equation}
where $\xi=-2(\frac{\alpha\nabla h}{h}+\nabla\varphi^0_2),\eta= \nabla k_1, |\xi(X)|\leq C\rho_{X}^{-1}$. In fact, $\xi$ is independent of $t$.
For any $X,Y\in Q_T$, assume $\rho_{X}
\geq \rho_{Y}, \rho_{X,Y}^{1-\gamma}=\rho_{X}^{1-
\gamma}\leq \rho_{Y}^{1-\gamma}, |\nabla \xi(X)|
\leq C\rho_{X}^{-2}$,

  \begin{equation}
\begin{aligned}
&\sup\limits_{X,Y\in Q_T}\rho_{X,Y}^{2+\beta-\gamma}\frac{|\xi(X)\eta(X)-\xi(Y)\eta(Y)|}{\delta(X,Y)^{\beta}}\\
\leq &   \sup\limits_{X,Y\in Q_{T}} \rho_{X,Y}^{2+\beta-\gamma} \frac{|\xi(X)-\xi(Y)||\eta(Y)|}{\delta(X,Y)^\beta}+\sup\limits_{X,Y\in Q_{T}}\rho_{X,Y}^{2+\beta-\gamma} \frac{|\xi(X)||\eta(X)-\eta(Y)|}{\delta(X,Y)^\beta}.
  \end{aligned}\label {125}
  \end{equation}
Set $d=\mu \rho_{X}$,  $\mu \leq \frac{1}{2}$. If $Y\in Q_d(X)$, then $(1-\mu)\rho_{X}\leq \rho_{Y}\leq (1+\mu)\rho_{X}$,
\begin{equation}
 \begin{aligned}
\rho_{X,Y}^{2+\beta-\gamma} \frac{|\xi(X)-\xi(Y)||\eta(Y)|}{\delta(X,Y)^\beta}&\leq \rho_{X}^{2+\beta-\gamma} \|\nabla\xi\|_{C^0_{Q_d(X)}}{\delta(X,Y)^{1-\beta}}|\eta(Y)|\\
 &\leq \rho_{X}^{2+\beta-\gamma} \frac{1}{(1-\mu)^2\rho_{X}^2}{d^{1-\beta}}|\eta(Y)|\\
 &\leq  \frac{\mu^{1-\beta}}{(1-\mu)^2}{\rho_{X}^{1-\gamma}}|\eta(Y)|\\
 %&\leq  \frac{\mu^{1-\beta} (1+\mu)^{\gamma-1} }{(1-\mu)^2}{\rho_{Y}^{1-\gamma}}|\eta(Y)|\\
 &\leq  \frac{\mu^{1-\beta} (1+\mu)^{\gamma-1} }{(1-\mu)^2} \sup\limits_{Y\in Q_d(X)}{\rho_{Y}^{1-\gamma}}|\eta(Y)|.
  \end{aligned}\label{126}
  \end{equation}
If $Y\not\in Q_d(X)$, then for any $\gamma> 2+\beta$, we have
\begin{equation}
\begin{aligned}
 \rho_{X,Y}^{2+\beta-\gamma} \frac{|\xi(X)-\xi(Y)||\eta(Y)|}{\delta(X,Y)^\beta}&\leq \rho_{X}^{2+\beta-\gamma} d^{-\beta}(|\xi(X)|+|\xi(Y)|)|\eta(Y)|   \\
 &\leq \rho_{X}^{2-\gamma}  \mu^{-\beta} (|\xi(X)|+|\xi(Y)|)|\eta(Y)|  \\
 &\leq C\mu^{-\beta} (\rho_{X}^{1-\gamma} |\eta(Y)| + \rho_{Y}^{1-\gamma} |\eta(Y)| ) \\
& \leq C\rho_{Y}^{1-\gamma} |\eta(Y)|.
  \end{aligned}\label {127}
  \end{equation}
We conclude from \eqref{126} and \eqref{127} that
\begin{equation}
\begin{aligned}
 \sup\limits_{X,Y\in Q_{T}} \rho_{X,Y}^{2+\beta-\gamma} \frac{|\xi(X)-\xi(Y)||\eta(Y)|}{\delta(X,Y)^\beta}&\leq C \sup\limits_{Y\in Q_{T}}{\rho_{Y}^{1-\gamma}}|\eta(Y)|.
  \end{aligned}\label {128}
  \end{equation}
Combining with (\ref{128}), (\ref{125}) yields
\begin{equation*}
\begin{aligned}
&\sup\limits_{X,Y\in Q_T}\rho_{X,Y}^{2+\beta-\gamma}\frac{|\xi(X)\eta(X)-\xi(Y)\eta(Y)|}{\delta(X,Y)^{\beta}}\\ 
\leq &  C( \sup\limits_{X\in Q_{T}}  \rho_{X}^{1-\gamma}|\eta(X)|+\sup\limits_{X,Y\in Q_{T}}\rho_{X,Y}^{1+\beta-\gamma}\frac{|\eta(X)-\eta(Y)|}{\delta(X,Y)^\beta}).
  \end{aligned}\label {130}
  \end{equation*} 
Next, for the third term of the right hand of (\ref{124}),
\begin{equation*}
\begin{aligned}
&\rho_{X,Y}^{2+\beta-\gamma}\frac{|2\nabla\varphi^0_1\nabla k_2(X)-2\nabla\varphi^0_1\nabla k_2(Y)|}{\delta(X,Y)^\beta}\\
\leq & C \left(\rho_{X,Y}^{2+\beta-\gamma}|\nabla\varphi^0_1(X)|\frac{|\nabla k_2(X)-\nabla k_2(Y)|}{\delta(X,Y)^\beta}+\rho_{X,Y}^{2+\beta-\gamma}
\frac{|\nabla\varphi^0_1(X)-\nabla\varphi^0_1(Y)|}{\delta(X,Y)^\beta}|\nabla k_2(Y)|\right)\\
\leq & C \left(\|\nabla k_2\|_{C^{\beta,\frac{\beta}{2}}(Q_T)}+\rho_{X,Y}^{2+\beta-\gamma}\frac{|\nabla\varphi^0_1(X)-\nabla\varphi^0_1(Y)|}{\delta(X,Y)^\beta}|\nabla k_2(Y)|\right).
  \end{aligned}\label{131}
  \end{equation*}
If $Y\in Q_d(X)$, then $(1-\mu)\rho_{X}\leq
\rho_{Y}\leq (1+\mu)\rho_{X}$, $0 \leq \delta (X,Y)\leq d$,
\begin{equation*}
\begin{aligned}
&\rho_{X,Y}^{2+\beta-\gamma}\frac{|\nabla\varphi^0_1(X)-\nabla\varphi^0_1(Y)|}{\delta(X,Y)^\beta}|\nabla k_2(Y)|\\
\leq &\rho_{X}^{2+\beta-\gamma} \|\nabla^2\varphi^0_1\|_{C^0({Q_d(X)})}{\delta(X,Y)^{1-\beta}}|\nabla k_2(Y)|\\
 \leq& \rho_{X}^{2+\beta-\gamma} (1+\mu)^{\gamma-2}\rho_{X}^{\gamma-2}{d^{1-\beta}}|\nabla k_2(Y)|\\
 \leq& C \|\nabla k_2\|_{C^0(Q_T)}.
  \end{aligned}\label {132}
  \end{equation*}
If $Y\not\in Q_d(X)$, $\delta (X,Y)\geq d$, then
\begin{equation*}
\begin{aligned}
&\rho_{X,Y}^{2+\beta-\gamma}\frac{|\nabla\varphi^0_1(X)-\nabla\varphi^0_1(Y)|}{\delta(X,Y)^\beta}|\nabla k_2(Y)|\\
\leq &\rho_{X}^{2+\beta-\gamma} d^{-\beta}(|\nabla\varphi^0_1(X)|+|\nabla\varphi^0_1(Y)|)|\nabla k_2(Y)|   \\
 \leq &\rho_{X}^{2-\gamma}  \mu^{-\beta} (|\nabla\varphi^0_1(X)|+|\nabla\varphi^0_1(Y)|)|\nabla k_2(Y)| \\
 \leq& C \|\nabla k_2\|_{C^0(Q_T)}.
  \end{aligned}\label {133}
  \end{equation*}
  So \begin{equation*}
  	\begin{aligned}
  		\rho_{X,Y}^{2+\beta-\gamma}\frac{|2\nabla\varphi^0_1\nabla k_2(X)-2\nabla\varphi^0_1\nabla k_2(Y)|}{\delta(X,Y)^\beta}\leq  C \|\nabla k_2\|_{C^{\beta,\frac{\beta}{2}}(Q_T)}.
  	\end{aligned}
  \end{equation*}
By \eqref{124},   
\begin{equation*}
	\begin{aligned}
		&\sup\limits_{Q_{T}}  \rho_{Y,Z}^{2+\beta-\gamma}\frac{|g_1(Y)-g_1(Z)|}{\delta(Y,Z)^\beta}
		\leq \sup\limits_{Q_{T}} \rho_{Y,Z}^{2+\beta-\gamma}\frac{|f_1(Y)-f_1(Z)|}{\delta(Y,Z)^\beta}+C(\sup\limits_{X\in Q_{T}} \rho_{X}^{1-\gamma}|\nabla k_1(X)|\\
		&+\sup\limits_{Q_{T}} \rho_{X,Y}^{1+\beta-\gamma}
		\frac{|\nabla k_1(X)-\nabla k_1(Y)|}{\delta(X,Y)^\beta}+ \|\nabla k_2\|_{C^{\beta,\frac{\beta}{2}}(Q_T)}).		
	\end{aligned}
\end{equation*}  
According to Theorem \ref{thm:k2} and Lemma \ref{Interpolation} and choosing $\varepsilon$ small in Lemma \ref{Interpolation}, we can obtain
\begin{equation*}
   \begin{aligned}
     &\sup\limits_{X\in Q_T}\rho_X^{1-\gamma}|\nabla k_1(X)|+\sup\limits_{X\in Q_T}\rho_X^{2-\gamma}(|\nabla^2k_1(X)|+|\partial_t k_1(X)|)\\
     &+\sup\limits_{X,Y\in Q_T}\rho_{X,Y}^{2+\beta-\gamma}(\frac{|\nabla^2k_1(X)-\nabla^2k_1(Y)|}{\delta(X,Y)^{\beta}}+\frac{|\partial_tk_1(X)-\partial_tk_1(Y)|}{\delta(X,Y)^{\beta}})\\
     \leq & C \left( \sup\limits_{Q_{T}}  \rho_X^{-\gamma} |k_1(X)|+  \sup\limits_{Q_{T}}  \rho_X^{2-\gamma} |f_1(X)| +  \sup\limits_{Q_{T}}\rho_{Y,Z}^{2+\beta-\gamma} \frac{|f_1(Y)-f_1(Z)|}{\delta(Y,Z)^\beta}+ \|\nabla k_2\|_{C^{\beta,\frac{\beta}{2}}(Q_T)} \right).
   \end{aligned} 
\end{equation*}
Consider the second equation of \eqref{8'}. The
standard parabolic theory yields
\begin{equation*}
   \begin{aligned}
   \|k_2\|_{C^{2+\beta,1+\frac{\beta}{2}}(Q_T)}\leq &C(\|f_2\|_{C^{\beta,\frac{\beta}{2}}(Q_T)}+\sup\limits_{X\in Q_T}\rho_X^{1-\gamma}|\nabla k_1(X)|\\
   &+\sup\limits_{X,Y\in Q_T}\rho_{X,Y}^{1+\beta-\gamma}\frac{|\nabla k_1(X)-\nabla k_1(Y)|}{\delta(X,Y)^{\beta}}+\|k_2\|_{W_2^{2,1}(Q_T)}).
   \end{aligned}
\end{equation*}
Applying Lemma \ref{Interpolation}, we finish the proof.
\end{proof}

Combined with Theorem \ref{thm:W422} and \eqref{k2}, we have
\begin{equation*}
	\begin{aligned}
&\|k_1\|_{W^{4,2}_{2}(Q_T;\rho^{-\alpha})}+\|k_2\|_{W^{4,2}_{2}(Q_T)}+\|k_1\|_{C^{2+\beta,1+\frac{\beta}{2}}(Q_T;\rho^{-\gamma})}+\|k_2\|_{C^{2+\beta,1+\frac{\beta}{2}}(Q_T)}\\
	\leqslant &C(\|f_1\|_{W^{2,1}_2(Q_T;\rho^{-\alpha})}+\|f_2\|_{W^{2,1}_{2}(Q_T)}+\|f_1\|_{C^{\beta,\frac{\beta}{2}}(Q_T;\rho^{-\gamma})}+\|f_2\|_{C^{\beta,\frac{\beta}{2}}(Q_T)}\\
	&+\|\rho^{-\alpha+1}{f}_1(0)\|_{H^1(M)}+\|{f}_2(0)\|_{H^1(M)}).
	\end{aligned}
\end{equation*}

Next, we give the similar weighted estimates for $(\widetilde{k_1},\widetilde{k_2})$ as we did for $(k_1,k_2)$ in the following lemmas. For brevity, we denote the second-order time derivatives $(\partial^2_t{k_1},\partial^2_t{k_2})$ of a given vector function $k=(k_1, k_2)$ by $\hat{k}=(\hat{k_1}, \hat{k_2})$.

\begin{lem}\label{lem:2.21}
	If $(k_1, k_2)$ is the solution of (\ref{8'}), then there exists a constant $C$ depending on $M$, $T$ and $\varphi^0$ such that
	\begin{equation*}
		\begin{aligned}
			&\sup\limits_{0\leq t \leq T}(\|\rho^{-\alpha+2}\partial_t\widetilde{k_1}(t)\|_{L^2(M)}+\|\partial_t\widetilde{k_2}(t)\|_{L^2(M)}+\|\rho^{-\alpha+2} \widetilde{k_1}(t)\|_{H^1(M)}+\| \widetilde{k_2}(t)\|_{H^1(M)})\\
			&+\|\rho^{-\alpha+2}\partial_t \widetilde{k_1}\|_{L^2(0,T;H^1(M))}+\| \partial_t \widetilde{k_2}\|_{L^2(0,T;H^1(M))}\\
				\leq &C(\|\rho^{-\alpha+1}f_1\|_{{L^2(0,T;L^2(M))}}+\|\rho^{-\alpha+1}\partial_t f_1\|_{{L^2(0,T;L^2(M))}}+\|\rho^{-\alpha+3}\partial^2_t f_1\|_{{L^2(0,T;L^2(M))}}\\
			&+\| f_2\|_{{H^2(0,T;L^2(M))}}+\|\rho^{-\alpha+2}f_1(0)\|_{H^2(M)}\\
			&+\|\rho^{-\alpha+2}\partial_tf_1(0)\|_{L^2(M)}+\|f_2(0)\|_{H^2(M)}+\|\partial_tf_2(0)\|_{L^2(M)}).
		\end{aligned}
	\end{equation*} 
\end{lem}

\begin{proof}
We know that the functions of form \eqref{a28} are dense in the space $W^{1,1}_2(Q_T;\rho^{-\alpha})$. In particular by the first equation of \eqref {a}, 
	\begin{equation}\label{g5}
	\begin {aligned}
	&(\partial_t \widetilde{k_1}^N, h^{-2\alpha} e^{-2\varphi^0_2}\Lambda)- (\mbox{div}(h^{-2\alpha}e^{-2\varphi^0_2}\nabla \widetilde{k^N_1}),\Lambda)\\
	&+2(\nabla\varphi^0_1\nabla \widetilde{k^N_2}, h^{-2\alpha} e^{-2\varphi^0_2}\Lambda)=(\partial_t f_1,h^{-2\alpha}e^{-2\varphi^0_2}\Lambda)\\
	\end {aligned}
\end{equation}
for each $\Lambda\in H^1(M,h^{-\alpha})$ and a.e. $ 0\leq t\leq T$. 
Taking $\Lambda=h^4\psi^1_m$ in \eqref{g5}, we deduce
\begin{equation}\label{g6}
	\begin {aligned}
	&(\partial_t \widetilde{k_1^N}, h^{-2\alpha+4} e^{-2\varphi^0_2}\psi^1_m)- (h^4\mbox{div}(h^{-2\alpha}e^{-2\varphi^0_2}\nabla \widetilde{k^N_1}),\psi^1_m)\\
	&+2(\nabla\varphi^0_1\nabla \widetilde{k^N_2}, h^{-2\alpha+4} e^{-2\varphi^0_2}\psi^1_m)=(\partial_t f_1,h^{-2\alpha+4}e^{-2\varphi^0_2}\psi^1_m).\\
	\end {aligned}
\end{equation}
	Fix $N\geqslant 1$ and differentiate equation \eqref{g6}  and the second equation of \eqref{a} with respect to $t$, we find
	\begin{equation}\label{g7}
		\left \{
		\begin {aligned}
		&(\partial_t\hat{k_1^N}, h^{-2\alpha+4} e^{-2\varphi^0_2}\psi^1_m)- (h^4\mbox{div}(h^{-2\alpha}e^{-2\varphi^0_2}\nabla \hat{k^N_1}),\psi^1_m)\\
		&+2(\nabla\varphi^0_1\nabla \hat{k^N_2},h^{-2\alpha+4} e^{-2\varphi^0_2}\psi^1_m)=(\partial^2_t f_1,h^{-2\alpha+4} e^{-2\varphi^0_2}\psi^1_m),\\
		&(\partial_t \hat{k_2^N},\psi^2_m)-(\Delta \hat{k_2^N},\psi^2_m) +2(h^{-2\alpha}e^{-2\varphi ^0_2}|\nabla \varphi^{0}_{1}|^2\hat{k_2^N} ,\psi^2_m)\\
		&-2(h^{-2\alpha}e^{-2\varphi ^0_2}\nabla \varphi^{0}_{1}\nabla \hat{k_1^N},\psi^2_m)=(\partial^2_t f_2,\psi^2_m), m=1,\dots N .\\	
		\end {aligned}
		\right.
	\end{equation}%13
 
	Multiply the first equation of \eqref{g7} by $\frac{d^2C_m^{N,1}}{dt^2}$ and  the second equation by $\frac{d^2C_m^{N,2}}{dt^2}$, then we sum the results for $m=1,\cdots, N$ to obtain the following:  
	\begin{equation}\label{g8}
		\left \{
		\begin {aligned}
		&(\partial_t\hat{k_1^N},h^{-2\alpha+4} e^{-2\varphi^0_2}\hat{k_1^N})- (h^4\mbox{div}(h^{-2\alpha}e^{-2\varphi^0_2}\nabla \hat{k^N_1}),\hat{k_1^N})\\
		&+2(\nabla\varphi^0_1\nabla \hat{k^N_2},h^{-2\alpha+4} e^{-2\varphi^0_2}\hat{k}_1^N)=(\partial^2_t f_1,h^{-2\alpha+4} e^{-2\varphi^0_2}\hat{k_1^N}),\\
		&(\partial_t \hat{k_2^N},\hat{k_2^N})-(\Delta \hat{k_2^N},\hat{k_2^N}) +2(h^{-2\alpha}e^{-2\varphi ^0_2}|\nabla \varphi^{0}_{1}|^2\hat{k_2^N} ,\hat{k_2^N})\\
		&-2(h^{-2\alpha}e^{-2\varphi ^0_2}\nabla \varphi^{0}_{1}\nabla \hat{k_1^N},\hat{k_2^N})=(\partial^2_t f_2,\hat{k_2^N}).\\	
		\end {aligned}
		\right.
	\end{equation}%13
\eqref{g8} yields
\begin{equation*}
	\begin {aligned}
	&\partial_t\frac{1}{2}\int_M h^{-2\alpha+4} e^{-2\varphi^0_2}|\hat{k_1^N}|^2\ud V+\int_M h^{-2\alpha+4}e^{-2\varphi^0_2}|\nabla \hat{k^N_1}|^2\ud V\\
	=&-4\int_M h^{-2\alpha+3} e^{-2\varphi^0_2} \hat{k_1^N}\nabla \hat{k_1^N} \nabla h \ud V-2\int_M h^{-2\alpha+4} e^{-2\varphi^0_2}\hat{k_1^N}\nabla\varphi^0_1\nabla \hat{k^N_2}\ud V\\
	&+\int_M h^{-2\alpha+4} e^{-2\varphi^0_2}\partial^2_t f_1\hat{k_1^N}\ud V
	\end {aligned}
\end{equation*}and
	\begin{equation*}
	\begin {aligned}&\partial_t\frac{1}{2}\int_M|\hat{k_2^N}|^2\ud V+\int_M |\nabla \hat{k^N_2}|^2\ud V =-2\int_M h^{-2\alpha}e^{-2\varphi ^0_2}|\nabla \varphi^{0}_{1}|^2|\hat{k_2^N} |^2\ud V\\
	&+2\int_M h^{-2\alpha}e^{-2\varphi ^0_2}\hat{k_2^N}\nabla \varphi^{0}_{1}\nabla \hat{k_1^N}\ud V+\int_M \partial^2_t f_2\hat{k_2^N}\ud V.\\	
	\end {aligned}
\end{equation*}%13
By the Cauchy inequality, we have 
\begin{equation*}
	\begin {aligned}
	&\partial_t\frac{1}{2}\int_M h^{-2\alpha+4} e^{-2\varphi^0_2}|\hat{k_1^N}|^2+ |\hat{k_2^N}|^2\ud V+\int_M h^{-2\alpha+4}e^{-2\varphi^0_2}|\nabla \hat{k^N_1}|^2+|\nabla \hat{k^N_2}|^2\ud V\\
	\leq & \delta\int_M h^{-2\alpha+4} e^{-2\varphi^0_2} |\nabla \hat{k_1^N}|^2+ |\nabla \hat{k_2^N}|^2\ud V\\
	&+C\int_M h^{-2\alpha+2} e^{-2\varphi^0_2} |\hat{k_1^N} |^2 +|\hat{k_2^N} |^2 +h^{-2\alpha+6} e^{-2\varphi^0_2} |\partial^2_t f _1|^2+|\partial^2_t f_2|^2\ud V.  \\
	\end {aligned}
\end{equation*}%13
Taking $\delta=\frac{1}{2}$ and considering Lemma \ref{lem:2.10}, the differential form of Gronwall's inequality yields the estimate
	\begin{equation*}
		\begin{aligned}
			&\max\limits_{0\leq {t}\leq T}\int_M h^{-2\alpha+4}e^{-2\varphi^0_2} | \hat{k^N_1}|^2+|  \hat{k^N_2}|^2 \ud V({t})
			\\&+\int_{Q_T} h^{-2\alpha+4}e^{-2\varphi^0_2}|\nabla  \hat{k^N_1}|^2 +|\nabla  \hat{k^N_2}|^2\ud V\ud s\\
			\leq &C\left(\|h^{-\alpha+2}e^{-\varphi^0_2}\hat {k^N_1}(0)\|^2_{L^2(M)} + \|\hat{k^N_2}(0)\|^2_{L^2(M)}+\|\rho^{-\alpha+1}f_1\|^2_{{L^2(0,T;L^2(M))}} \right.\\
			&+\|\rho^{-\alpha+1}\widetilde {f_1}\|^2_{{L^2(0,T;L^2(M))}}+\|\rho^{-\alpha+3}\hat {f_1}\|^2_{{L^2(0,T;L^2(M))}}+\| f_2\|^2_{{H^2(0,T;L^2(M))}}\\
			&\left.+\|\rho^{-\alpha+1}f_1(0)\|_{H^1(M)}^2+\|f_2(0)\|_{H^1(M)}^2\right).
		\end{aligned}
	\end{equation*}
Notice that	$\partial_t k_1(0)=f_1(0)$ and $\partial_t k_2(0)=f_2(0)$,
	\begin{equation*}
	\begin{aligned}
			\partial^2_tk_1(0)&=\Delta f_1(0)-2(\nabla \varphi^0_2+\frac{\alpha\nabla h}{h})\nabla f_1(0)-2\nabla\varphi^0_1\nabla f_2(0)+\partial_t f_1(0),\\
		\partial^2_tk_2(0)
		&=\Delta f_2(0)-2h^{-2\alpha}e^{-2\varphi^0_2}|\nabla \varphi^0_1|^2f_2(0)+2h^{-2\alpha}e^{-2\varphi^0_2}\nabla \varphi^0_1\nabla f_1(0)+\partial_tf_2(0).\\
		\end{aligned}
\end{equation*}		
Therefore,
	\begin{equation}\label{g9}
	\begin{aligned}
		&\max\limits_{0\leq {t}\leq T}\int_M h^{-2\alpha+4}e^{-2\varphi^0_2} | \hat{k^N_1}|^2+|  \hat{k^N_2}|^2 \ud V({t})
		\\&+\int_{Q_T} h^{-2\alpha+4}e^{-2\varphi^0_2}|\nabla  \hat{k^N_1}|^2 +|\nabla  \hat{k}^N_2|^2\ud V\ud s\\
		\leq &C\left(\|\rho^{-\alpha+1}f_1\|^2_{{L^2(0,T;L^2(M))}}+\|\rho^{-\alpha+1}\widetilde{f}_1\|^2_{{L^2(0,T;L^2(M))}}\right.\\
		&+\|\rho^{-\alpha+3}\hat{f_1}\|^2_{{L^2(0,T;L^2(M))}}+\| f_2\|^2_{{H^2(0,T;L^2(M))}}+\|\rho^{-\alpha+2}f_1(0)\|_{H^2(M)}^2\\
		&\left.+\|\rho^{-\alpha+2}\widetilde {f_1}(0)\|_{L^2(M)}^2+\|f_2(0)\|_{H^2(M)}^2+\|\widetilde {f_2}(0)\|_{L^2(M)}^2\right).
	\end{aligned}
\end{equation}
	Passing to limits as $N=N_i\rightarrow \infty$ in (\ref{g9}), we obtain the desired estimate.
	
	With $N\geq 1$ fixed, we multiply equation \eqref{g6} by $\frac{d^2C_m^{N,1}}{dt^2}$ and the second equation of \eqref{a} by $\frac{d^2C_m^{N,2}}{dt^2}$, and then sum over $m$ from $1$ to $N$ to discover
	
	\begin{equation*}
		\begin{aligned}
			&\int_M h^{-2\alpha+4}e^{-2\varphi^0_2} |\hat{k^N_1}|^2+|\hat{k^N_2} |^2\ud V+\frac{1}{2}\partial_t\int_M ( h^{-2\alpha+4}e^{-2\varphi^0_2} |\nabla \widetilde{k^N_1}|^2+   |\nabla \widetilde{k^N_2}|^2)\ud V\\
			=&-4\int_M h^{-2\alpha+3}e^{-2\varphi^0_2}\hat{k^N_1}\nabla\widetilde{k}^N_1\nabla h\ud V\\
			&-2\int_M h^{-2\alpha}e^{-2\varphi^0_2}|\nabla \varphi^0_1|^2\widetilde{k^N_2} \hat{k^N_2}\ud V+\int_M 2h^{-2\alpha}e^{-2\varphi^0_2}
			\nabla \varphi^0_1 \nabla \widetilde{k^N_1}\hat{k^N_2}\ud V\\
			&-2\int_M h^{-2\alpha+4}e^{-2\varphi^0_2} \hat{k^N_1}\nabla\varphi^0_1\nabla \widetilde{k^N_2}\ud V +\int_M h^{-2\alpha+4}e^{-2\varphi^0_2} \widetilde {f_1} \hat{k^N_1}+ \widetilde {f_2} \hat{k^N_2}\ud V.
		\end{aligned}
	\end{equation*}
	For each $\delta>0$, by the assumption of $\varphi_1^0$ and the Young's inequality, we have
	\begin{equation*}
		\begin{aligned}
			&\int_M h^{-2\alpha+4}e^{-2\varphi^0_2} |\hat{k^N_1}|^2+|\hat{k^N_2} |^2\ud V+\frac{1}{2}\partial_t\int_M ( h^{-2\alpha+4}e^{-2\varphi^0_2} |\nabla \widetilde{k^N_1}|^2+   |\nabla \widetilde{k^N_2}|^2)\ud V\\ \leq &\delta \int_M h^{-2\alpha+4}e^{-2\varphi^0_2}|\hat{k^N_1}|^2+|\hat{k^N_2}|^2\ud V+C\int_M h^{-2\alpha+6}e^{-2\varphi^0_2}|\partial_tf_1|^2 + |\partial_tf_2|^2 \ud V\\&+C\int_M h^{-2\alpha}e^{-2\varphi^0_2} |\nabla \widetilde{k^N_1}|^2+h^{-2\alpha+2}e^{-2\varphi^0_2} |\partial_t\widetilde{k^N_1}|^2+| \widetilde{k^N_2}|^2+ |\nabla \widetilde{k^N_2}|^2\ud V.		\end{aligned}
	\end{equation*}

To obtain the desired estimate,	we take $\delta=\frac{1}{4}$ and integrate the above inequality over $[0,T]$. By utilizing Lemma \ref{lem:W112} and Lemma \ref{lem:2.10}, we find that
	\begin{equation*}
		\begin{aligned}
			&\sup\limits_{0\leq t\leq T}\int_M (h^{-2\alpha+4}e^{-2\varphi^0_2}|\nabla \widetilde{k^N_1}|^2+   |\nabla \widetilde{k^N_2}|^2)\ud V(t)\\
			&+\int_{Q_T} h^{-2\alpha+4}e^{-2\varphi^0_2} | \partial_t\widetilde{k^N_1}|^2+|\partial_t\widetilde{k^N_2} |^2\ud V\ud s\\
				\leq &C\left(\|\rho^{-\alpha+1}f_1\|^2_{{L^2(0,T;L^2(M))}}+\|\rho^{-\alpha+1}\widetilde {f_1}\|^2_{{L^2(0,T;L^2(M))}}\right.\\
			&\left.+\| f_2\|^2_{{H^1(0,T;L^2(M))}}+\|\rho^{-\alpha+1}f_1(0)\|_{H^1(M)}^2+\|f_2\|_{H^1(M)}^2(0)\right).
		\end{aligned}
	\end{equation*}
	
%	Here we used the estimate
%	\begin{equation*}
%		\begin{aligned}
%			&\|h^{-\alpha}e^{-\varphi^0_1}f_1\|_{L^2(M)}^2(0)+\|f_2\|_{L^2(M)}^2(0)
%			\leq C\left(\|h^{-\alpha}e^{-\varphi^0_1}f_1\|_{L^2(0,T;L^2(M))}^2\right.\\
%			&\left.+\|h^{-\alpha}e^{-\varphi^0_1}\partial_tf_1\|_{L^2(0,T;L^2(M))}^2
%			+\|f_2\|_{L^2(0,T;L^2(M))}^2+\|\partial_t f_2\|_{L^2(0,T;L^2(M))}^2\right).\\
%		\end{aligned}
%	\end{equation*}
%	which follows from Theorem 2 in Chapter 5.9.2 in \cite{LC}.
	
	Taking subsequences $\{\widetilde{k^{N_i}_1}\}_i$ and$\{\widetilde{k^{N_i}_2}\}_i$ and passing to limits as $N_i\rightarrow \infty$, we deduce
	\begin{equation*}
		\begin{aligned}
			\widetilde{k^{N_i}_1}\rightarrow \widetilde{k_1} \ &\mbox{ in } L^\infty(0,T;H^1(M;h^{-\alpha+2})), \  \partial_t \widetilde{k^{N_i}_1}\rightarrow\partial_t \widetilde{k_1} \ \mbox{ in } L^2(0,T;L^2(M; h^{-\alpha+2})),\\
			\widetilde{k}^{N_i}_2\rightarrow \widetilde{k_2} \ & \mbox{ in } L^\infty(0,T;H^1(M)), \  \ \quad \partial_t \widetilde{k^{N_i}_2}\rightarrow\partial_t \widetilde{k_2} \ \mbox{ in } L^2(0,T;L^2(M)).
		\end{aligned}
	\end{equation*}
	%Therefore, there exist subsequences of $\{\partial_tk_1^N \}_{N=1}^{\infty}$ and $\{h^{-3}e^{-\varphi^0_1}\partial_tk_2^N \}_{N=1}^{\infty}$ converging weakly to  $\partial_tk_1$ and $h^{-3}e^{-\varphi^0_1}\partial_tk_2 $ in $L^2(Q_T)$. For any $t\in[0,T]$,  there exist subsequences of $\{\nabla k_1^N \}_{N=1}^{\infty}$ and $\{h^{-3}e^{-\varphi^0_1}\nabla k_2^N \}_{N=1}^{\infty}$ converging weakly to  $\nabla k_1$ and $h^{-3}e^{-\varphi^0_1}\nabla k_2 $ in $L^2(M)$.
	
	Hence we have completed the proof.
	\iffalse
	\begin{equation*}
		\begin{aligned}
			&\int_0^T\int_M| \partial_tk_1|^2+h^{-6}e^{-2\varphi^0_1}|\partial_tk_2 |^2\ud V\ud s +\sup\limits_{0\leqslant t\leqslant T}\int_M(|\nabla k_1|^2+  h^{-6}e^{-2\varphi^0_1} |\nabla k_2|^2)\ud V(t)\\
			&\leq 
			C(\int_0^T\int_M |f_1|^2 +|\partial_t  f_1|^2 +h^{-4}e^{-2\varphi^0_1}(|f_2|^2+|\partial_t  f_2|^2) \ud V\ud s\\
			&+\|f_1\|_{L^2(M)}^2(0)+\|h^{-3}e^{-2\varphi^0_1}f_2\|_{L^2(M)}^2(0)).
		\end{aligned}
	\end{equation*}
	\fi
\end{proof}
\begin{lem}\label{lem:2.100}
	If $(k_1, k_2)$ is a solution of (\ref{8'}), then there exists a constant $C$ depending on $M$, $T$ and $\varphi^0$ such that
	\begin{equation*}
		\begin{aligned}
			&\sup\limits_{0\leq t\leq T}  ( \|\rho^{-\alpha+3}\partial_t\widetilde{k_1}(t)\|_{H^1{(M)}}+	\|\partial_t\widetilde{k_2}(t)\|_{H^1{(M)}})+\|\rho^{-\alpha+3}\partial_t\widetilde{k_1}\|_{L^2(0,T;H^2(M))}\\
			&+\|\partial_t \widetilde{k_2}\|_{L^2(0,T;H^2(M))}+\|\rho^{-\alpha+3}\partial^2_t \widetilde{k_1}\|_{L^2(0,T;L^2{(M)})}+	\|\partial^2_t \widetilde{k_2}\|_{L^2(0,T;L^2{(M)})}\\
				\leq &C\left(\|\rho^{-\alpha+1}f_1\|_{{H^1(0,T;L^2(M))}}++\|\rho^{-\alpha+3}\partial^2_t f_1\|_{{L^2(0,T;L^2(M))}}+\| f_2\|_{{H^2(0,T;L^2(M))}}\right.\\
			&\left.+\|\rho^{-\alpha+3}f_1(0)\|_{H^3(M)}+\|\rho^{-\alpha+3}\partial_tf_1(0)\|_{H^1(M)}+\|f_2(0)\|_{H^3(M)}^2+\|\partial_tf_2(0)\|_{H^1(M)}\right).
		\end{aligned}
	\end{equation*} 
\end{lem}
\begin{proof} By equation (\ref{8'}) and direct calculation, we have
	\begin{equation}
		\begin{aligned}\label{g10}
			%  \Delta(\rho^{\varepsilon}k_1)&=\rho^{\varepsilon}(\partial_tk_1+2h^{-6}e^{-2\varphi^0_1}|\nabla \varphi^0_2|^2k_1-2h^{-6}e^{-2\varphi^0_1}\nabla \varphi^0_2\nabla k_2-f_1)+2\nabla(\rho^{\varepsilon})\nabla k_1+\Delta(\rho^{\varepsilon})k_1,\\
			%\partial_t k_2-	\Delta k_2=&-2h^{-2\alpha}e^{-2\varphi^0_2}|\nabla \varphi^0_1|^2k_2+2h^{-2\alpha}e^{-2\varphi^0_2}\nabla \varphi^0_1\nabla k_1+f_2,\\
			&\partial_t(\rho^{-\alpha+3}k_1)-\Delta(\rho^{-\alpha+3}k_1)
			=-2\nabla(\rho^{-\alpha+3})\nabla k_1-\Delta(\rho^{-\alpha+3})k_1 \\
			&+\rho^{-\alpha+3}(-2(\nabla \varphi^0_2+\frac{\alpha\nabla h}{h})\nabla k_1-2\nabla\varphi^0_1\nabla k_2+f_1),\\
			&\partial_t{k}_2-\Delta {k}_2=-2h^{-2\alpha}e^{-2\varphi^0_2}|\nabla \varphi^0_1|^2{k}_2+2h^{-2\alpha}e^{-2\varphi^0_2}\nabla \varphi^0_1\nabla {k}_1+ f_2.\\
		\end{aligned}
	\end{equation}
	By differentiating equations twice in $\eqref{g10}$ with respect to $t$, we can get
	\begin{equation}\label{g11}
		\begin{aligned}
			&\partial_t(\rho^{-\alpha+3}\hat{k_1})-\Delta(\rho^{-\alpha+3}\hat{k_1})=-2\nabla(\rho^{-\alpha+3})\nabla \hat{k_1}-\Delta(\rho^{-\alpha+3})\hat{k_1}\\
			&+\rho^{-\alpha+3}(-2(\nabla \varphi^0_2+\frac{\alpha\nabla h}{h})\nabla \hat{k_1}-2\nabla\varphi^0_1\nabla \hat{k_2}+\partial^2_t f_1),\\
			&\partial_t\hat{k}_2-\Delta \hat{k_2}=-2h^{-2\alpha}e^{-2\varphi^0_2}|\nabla \varphi^0_1|^2\hat{k_2}+2h^{-2\alpha}e^{-2\varphi^0_2}\nabla \varphi^0_1\nabla \hat{k_1}+
			\partial^2_t f_2.
		\end{aligned}
	\end{equation}
	 Notice that 
	\begin{equation}
		\begin{aligned}\label{2.53}
			&\rho^{-\alpha+3}\hat{k_1}(0)=	\rho^{-\alpha+3}(\Delta f_1(0)-2(\nabla \varphi^0_2+\frac{\alpha\nabla h}{h})\nabla f_1(0)-2\nabla\varphi^0_1\nabla f_2(0)+\partial_t f_1(0)),\\
				&\hat{k_2}(0)=\Delta f_2(0)-2h^{-2\alpha}e^{-2\varphi^0_2}|\nabla \varphi^0_1|^2f_2(0)+2h^{-2\alpha}e^{-2\varphi^0_2}\nabla \varphi^0_1\nabla f_1(0)+\partial_tf_2(0).\\
		\end{aligned}
	\end{equation}

	The equations\eqref{g11} with the initial function (\ref{2.53}) have weak solutions, since the right-hand sides of the equations belong to $L^2(Q_T)$ and $\rho^{-\alpha+3}\hat{k}_1(0), \hat{k}_2(0)\in L^2(M)$. By the regularity of parabolic equations, in view of Theorem 5 of Chapter 7.1 of \cite{C}, we have
	\begin{equation*}
		\begin{aligned}
			&\sup\limits_{0\leq t\leq T}  ( \|\rho^{-\alpha+3}\hat{k_1}(t)\|_{H^1{(M)}}+	\|\hat{k_2}(t)\|_{H^1{(M)}})
			+\|\rho^{-\alpha+3}\hat{k_1}(t)\|_{L^2(0,T;H^2{(M)})}\\
			&+	\|\hat{k_2}(t)\|_{L^2(0,T;H^2{(M)})}+\|\rho^{-\alpha+3}\partial_t \hat{k_1}(t)\|_{L^2(0,T;L^2{(M)})}+	\|\partial_t \hat{k_2}(t)\|_{L^2(0,T;L^2{(M)})}\\
			\leq& C\left(\|\partial_t(\rho^{-\alpha+3}\hat{k_1})-\Delta(\rho^{-\alpha+3}\hat{k_1})\|_{L^2(0,T;L^2(M))}+\|\partial_t\hat{k_2}-	\Delta \hat{k_2}\|_{L^2(0,T;L^2(M))}\right.\\
			&\left.+\|\rho^{-\alpha+3}\hat{k_1}(0)\|_{H^1(M)}+\|\hat{k_2}(0)\|_{H^1(M)}\right)\\
			\leq &C\left(\|\rho^{-\alpha+1}f_1\|_{{H^1(0,T;L^2(M))}}+\|\rho^{-\alpha+3}\partial^2_t f_1\|_{{L^2(0,T;L^2(M))}}+\| f_2\|_{{H^2(0,T;L^2(M))}}\right.\\
			&\left.+\|\rho^{-\alpha+3}f_1(0)\|_{H^3(M)}+\|\rho^{-\alpha+3}\partial_tf_1(0)\|_{H^1(M)}+\|f_2(0)\|_{H^3(M)}+\|\partial_tf_2(0)\|_{H^1(M)}\right).
		\end{aligned}
	\end{equation*}
	%Then, by combining the estimates from the previous theorems and Theorem 3 of Chapter 5.9 of \cite{C} we have 
	%\begin{equation}
	%	\begin{aligned}
		%		&\sup\limits_{0\leqslant t\leqslant T}  ( \|\rho^{-\alpha+1}\widetilde{k}_1(t)\|_{H^1{(M)}}+\|\widetilde{k}_2(t)\|_{H^1{(M)}})
		%		+\|\rho^{-\alpha+1}\widetilde{k}_1(t)\|_{L^2(0,T;H^2{(M)})}\\
		%		&+	\|\widetilde{k}_2(t)\|_{L^2(0,T;H^2{(M)})}+\|\rho^{-\alpha+1}\partial_t \widetilde{k}_1(t)\|_{L^2(0,T;L^2{(M)})}+	\|\partial_t \widetilde{k}_2(t)\|_{L^2(0,T;L^2{(M)})}\\
		%		\leqslant& C(\|\rho^{-\alpha+1}{f}_1\|_{L^2(0,T; H^2(M))}^2+\|\rho^{-\alpha+1}\partial_t{f}_1\|_{L^2(0,T; L^2(M))}^2\\
		%		&+\|{f}_2\|_{L^2(0,T;H^2(M))}
		%		+\|\partial_t{f}_2\|_{L^2(0,T;L^2(M))}^2).
		%	\end{aligned}
	%\end{equation}
\end{proof}

\begin{lem}
	If $(k_1, k_2)$ is a solution of (\ref{8'}), then there exists a constant $C$ depending on $M$, $T$ and $\varphi^0$ such that
	\begin{equation*}
		\begin{aligned}
			&\|\rho^{-\alpha+3}\widetilde{k_1}\|_{L^2(0,T;H^4(M))}+\|\widetilde{k_2}\|_{L^2(0,T;H^4(M))}\\
			\leq &C\left(\|\rho^{-\alpha+1}f_1\|_{{L^2(0,T;L^2(M))}}+\|\rho^{-\alpha+3}\partial_t f_1\|_{{L^2(0,T;H^2(M))}}+\|\rho^{-\alpha+3}\partial^2_t f_1\|_{{L^2(0,T;L^2(M))}}\right.\\
			&+\| f_2\|_{{L^2(0,T;L^2(M))}}+\| \partial_t f_2\|_{{L^2(0,T;H^2(M))}}+\| \partial^2_t f_2\|_{{L^2(0,T;L^2(M))}}\\
			&\left.+\|\rho^{-\alpha+3}f_1(0)\|_{H^3(M)}+\|f_2(0)\|_{H^3(M)}^2\right).
		\end{aligned}
	\end{equation*}
\end{lem}

\begin{proof}
	
The equations
	\begin{equation*}
		\begin{aligned}
			%  \Delta(\rho^{\varepsilon}k_1)&=\rho^{\varepsilon}(\partial_tk_1+2h^{-6}e^{-2\varphi^0_1}|\nabla \varphi^0_2|^2k_1-2h^{-6}e^{-2\varphi^0_1}\nabla \varphi^0_2\nabla k_2-f_1)+2\nabla(\rho^{\varepsilon})\nabla k_1+\Delta(\rho^{\varepsilon})k_1,\\
			\Delta \widetilde{k_2}=&\partial_t\widetilde{k_2}+2h^{-2\alpha}e^{-2\varphi^0_2}|\nabla \varphi^0_1|^2\widetilde{k_2}-2h^{-2\alpha}e^{-2\varphi^0_2}\nabla \varphi^0_1\nabla \widetilde{k_1}-
			\partial_tf_2,\\
			\Delta (\rho^{-\alpha+1}\widetilde{k_1})=&\rho^{-\alpha+1}(\partial_t\widetilde{k_1}+2(\nabla\varphi^0_2+\frac{\alpha\nabla h}{h})\nabla \widetilde{k_1}+2\nabla\varphi^0_1\nabla \widetilde{k_2}-\partial_tf_1)\\
			&+2\nabla(\rho^{-\alpha+1})\nabla \widetilde{k_1}+\Delta(\rho^{-\alpha+1}) \widetilde{k_1},
		\end{aligned}
	\end{equation*}\begin{equation*}
		\begin{aligned}
				\Delta(\rho^{-\alpha+2}\widetilde{k_1})=&\rho^{-\alpha+2}(\partial_t\widetilde{k_1}+2(\nabla \varphi^0_2+\frac{\alpha\nabla h}{h})\nabla \widetilde{k_1}+2\nabla\varphi^0_1\nabla \widetilde{k_2}-\partial_tf_1)\\
			&+2\nabla(\rho^{-\alpha+2})\nabla \widetilde{k_1}+\Delta(\rho^{-\alpha+2})\widetilde{k_1},\\
			\Delta(\rho^{-\alpha+3}\widetilde{k_1})=&\rho^{-\alpha+3}(\partial_t\widetilde{k_1}+2(\nabla \varphi^0_2+\frac{\alpha\nabla h}{h})\nabla \widetilde{k_1}+2\nabla\varphi^0_1\nabla \widetilde{k_2}-\partial_tf_1)\\
			&+2\nabla(\rho^{-\alpha+3})\nabla \widetilde{k_1}+\Delta(\rho^{-\alpha+3})\widetilde{k_1}\\
		\end{aligned}
	\end{equation*}
	all possess weak solutions, since the right-hand sides of the equations belong to $L^2(M)$.
	By the regularity of elliptic equations, in view of Theorem 1 and Theorem 2 of Chapter 6.3 of \cite{C}, we have
%	\begin{equation*}
%		\begin{aligned}
%			\|\widetilde{k}_2\|^2_{H^2(M)}&\leq C( \|\widetilde{k}_2\|^2_{L^2(M)}+\|\Delta \widetilde{k}_2\|^2_{L^2(M)}),\\
%			\|\rho^{-\alpha+2}\widetilde{k}_1\|^2_{H^2(M)}&\leq C( \|\rho^{-\alpha+2}\widetilde{k}_1\|^2_{L^2(M)}+\|\Delta(\rho^{-\alpha+2}\widetilde{k}_1)\|^2_{L^2(M)}),\\
%			\|\widetilde{k}_2\|^2_{H^3(M)}&\leq C( \|\widetilde{k}_2\|^2_{L^2(M)}+\|\Delta \widetilde{k}_2\|^2_{H^1(M)}),\\
%			\|\rho^{-\alpha+3}\widetilde{k}_1\|^2_{H^3(M)}&\leq C( \|\rho^{-\alpha+3}\widetilde{k}_1\|^2_{L^2(M)}+\|\Delta(\rho^{-\alpha+3}\widetilde{k}_1)\|^2_{H^1(M)}).\\
%		\end{aligned}
%	\end{equation*}
%Therefore utilizing the estimate from Lemma \ref{lem:W112}, Lemma \ref{lem:2.10}, Lemma \ref{lem:2.21}, Lemma \ref{lem:2.100}, we deduce
%	\begin{equation*}
%		\begin{aligned}
%			&\|\rho^{-\alpha+3}\widetilde{k}_1\|_{L^\infty(0,T;H^3(M))}+\|\widetilde{k}_2\|_{L^\infty(0,T;H^3(M))}\\
%				\leq &C(\|\rho^{-\alpha+1}f_1\|_{{L^2(0,T;L^2(M))}}+\|\rho^{-\alpha+1}\partial_t f_1\|_{{L^2(0,T;L^2(M))}}+\|\rho^{-\alpha+3}\partial^2_t f_1\|_{{L^2(0,T;L^2(M))}}\\
%			&+\| f_2\|_{{L^2(0,T;L^2(M))}}+\| \partial_t f_2\|_{{L^2(0,T;L^2(M))}}+\| \partial^2_t f_2\|_{{L^2(0,T;L^2(M))}}\\
%			&+\|\rho^{-\alpha+3}f_1(0)\|_{H^3(M)}+\|\rho^{-\alpha+3}\partial_tf_1\|_{L^\infty(0,T;H^1(M))}\\
%			&+\|f_2(0)\|_{H^3(M)}^2+\|\partial_tf_2\|_{L^\infty(0,T;H^1(M))}).			
%		\end{aligned}
%	\end{equation*}
%	Similarly,
	\begin{equation*}
		\begin{aligned}
			\|\widetilde{k}_2\|^2_{H^2(M)}&\leq C( \|\widetilde{k_2}\|^2_{L^2(M)}+\|\Delta \widetilde{k_2}\|^2_{L^2(M)}),\\	\|\rho^{-\alpha+1}\widetilde{k_1}\|^2_{H^2(M)}&\leq C( \|\rho^{-\alpha+1}\widetilde{k_1}\|^2_{L^2(M)}+\|\Delta(\rho^{-\alpha+1}\widetilde{k_1})\|^2_{L^2(M)}),
			\end{aligned}
	\end{equation*}
		\begin{equation*}
		\begin{aligned}
			\|\widetilde{k_2}\|^2_{H^3(M)}&\leq C( \|\widetilde{k_2}\|^2_{L^2(M)}+\|\Delta \widetilde{k_2}\|^2_{H^1(M)}),\\
			\|\rho^{-\alpha+2}\widetilde{k_1}\|^2_{H^3(M)}&\leq C( \|\rho^{-\alpha+2}\widetilde{k_1}\|^2_{L^2(M)}+\|\Delta(\rho^{-\alpha+2}\widetilde{k_1})\|^2_{H^1(M)}),\\
			\|\widetilde{k_2}\|^2_{H^4(M)}&\leq C( \|\widetilde{k_2}\|^2_{L^2(M)}+\|\Delta \widetilde{k_2}\|^2_{H^2(M)}),\\
			\|\rho^{-\alpha+3}\widetilde{k_1}\|^2_{H^4(M)}&\leq C( \|\rho^{-\alpha+3}\widetilde{k_1}\|^2_{L^2(M)}+\|\Delta (\rho^{-\alpha+3}\widetilde{k_1})\|^2_{H^2(M)}).
		\end{aligned}
	\end{equation*}
	Integrating $t$ from $0$ to $T$ and utilizing the estimate from Lemma \ref{lem:W112}, Lemma \ref{lem:2.10}, Lemma \ref{lem:2.21}, Lemma \ref{lem:2.100}, we deduce
	\begin{equation*}
	\begin{aligned}
		&\|\rho^{-\alpha+3}\widetilde{k_1}\|_{L^2(0,T;H^4(M))}+\|\widetilde{k_2}\|_{L^2(0,T;H^4(M))}\\
		\leq &C\left(\|\rho^{-\alpha+1}f_1\|_{{L^2(0,T;L^2(M))}}+\|\rho^{-\alpha+3}\partial_t f_1\|_{{L^2(0,T;H^2(M))}}+\|\rho^{-\alpha+3}\partial^2_t f_1\|_{{L^2(0,T;L^2(M))}}\right.\\
		&+\| f_2\|_{{L^2(0,T;L^2(M))}}+\| \partial_t f_2\|_{{L^2(0,T;H^2(M))}}+\| \partial^2_t f_2\|_{{L^2(0,T;L^2(M))}}\\
		&\left.+\|\rho^{-\alpha+3}f_1(0)\|_{H^3(M)}+\|f_2(0)\|_{H^3(M)}^2\right).			
	\end{aligned}
\end{equation*}	
Here we used the estimate
	\begin{equation*}
		\begin{aligned}
	&\|\rho^{-\alpha+3}\partial_tf_1\|_{L^\infty(0,T;H^1(M))}+\|\partial_tf_2\|_{L^\infty(0,T;H^1(M))}\\
	\leq&C\left(\|\rho^{-\alpha+3}\partial_tf_1\|_{L^2(0,T;H^2(M))}+\|\rho^{-\alpha+3}\partial^2_tf_1\|_{L^2(0,T;L^2(M))}\right.\\
	&\left.+\|\partial_tf_2\|_{L^2(0,T;H^2(M))}+\|\partial^2_tf_2\|_{L^2(0,T;L^2(M))}\right).\\ 				
			\end{aligned}
	\end{equation*}
		which follows from Theorem 4 in Chapter 5.9.2 of \cite{LC}.
\end{proof}

In conclusion, we get the estimate
	\begin{equation*}
	\begin{aligned}
	&\|\rho^{-\alpha+3}\widetilde{k_1}\|_{W^{4,2}_2(Q_T)}+\|\widetilde{k_2}\|_{W^{4,2}_2(Q_T)}\\
	\leq &C\left(\|\rho^{-\alpha+3}\partial_t{f}_1\|_{W^{2,1}_2(Q_T)}+\|\partial_t{f}_2\|_{W^{2,1}_2(Q_T)}+\|\rho^{-\alpha+1}f_1\|_{{L^2(0,T;L^2(M))}}\right.\\
		&\left.+\|f_2\|_{{L^2(0,T;L^2(M))}}+\|\rho^{-\alpha+3}f_1(0)\|_{H^3(M)}+\|f_2(0)\|_{H^3(M)}^2\right).
	\end{aligned}
\end{equation*}
In particular, $\partial_t k_2\in W^{4,2}_2(Q_T)$. Moreover, according to the Sobolev embedding theorem, $
\nabla_{x_i} \partial_t k_2 \in C^{\frac{1}{2},\frac{1}{4}}(Q_T), i=1,2,3$ and $\|\nabla \partial_t k_2\|_{C^{\frac{1}{2},\frac{1}{4}}(Q_T)}\leqslant C \| \partial_t k_2\|_{W^{4,2}_{2}(Q_T)}$.
Since $f=(f_1,f_2) \in  \mathcal{C}
\times\mathcal{D} $, we have $\partial_tk_1,\partial_tk_2\in C^{2+\beta,1+\frac{\beta}{2}}({M\backslash \Gamma\times [0,T]})$. %Therefore $k_1, k_2\in C^{\infty}(M\backslash \Gamma\times [0,T])$.

It is worth mentioning that $\partial_tk_1(0)=f_1(0)\neq0$. However,  it satisfies the condition $\sup\limits_{x\in M}|\rho^{-\gamma}f_1 |(x,0)\leq C$. Therefore, the result of Lemma \ref{key} still holds for $\partial_t k_1$.
 We obtain
\begin{equation}
	\begin{aligned}
		|\partial_t k_1(x,t)|\leqslant C \rho(x)^{\gamma} \ \mbox{ for } (x,t) \in B_1(x)
		\times[0,T].
	\end{aligned}
\end{equation}

Similar to the proof of Theorem \ref{k1k2C2}, we have
\begin{thm}
	There is an uniform constant C such that for any $0<\beta<\min\{2\alpha-2,\frac{1}{2}\}, 2+\beta<\gamma<2\alpha$,
	\begin{equation*}
		\begin{aligned}
			& \|\widetilde{k_2}\|_{C^{2+\beta,1+\frac{\beta}{2}}(Q_T)}+ \sup\limits_{X\in Q_T}\rho_X^{1-\gamma}|\nabla \widetilde{k_1}(X)|+\sup\limits_{X\in Q_T}\rho_X^{2-\gamma}(| \nabla^2\widetilde{k_1}(X)|+|\partial_t \widetilde{k_1}(X)|)\\ &+\sup\limits_{X,Y\in Q_T}\rho_{X,Y}^{2+\beta- \gamma}(\frac{|\nabla^2\widetilde{k_1}(X)-\nabla^2\widetilde{k_1}(Y)|}
			{\delta(X,Y)^{\beta}}+\frac{|\partial_t\widetilde{k_1}(X)-\partial_t\widetilde{k_1}(Y)|}{\delta(X,Y)^{\beta}})\\
			\leq& C \left( \|\nabla \widetilde{k_2}\|_{C^{\beta,\frac{\beta}{2}}(Q_T)}+\sup\limits_{Q_{T}}  \rho_X^{-\gamma} |\widetilde{k_1}(X)|+ \|\partial_tf_2\|_{C^{\beta,\frac{\beta}{2}}(Q_T)}+\sup\limits_{Q_{T}}  \rho_X^{2-\gamma} |\partial_tf_1(X)| \right.\\
			&\left.+ \sup\limits_{Q_{T}}  \rho_{Y,Z}^{2+\beta-\gamma} \frac{|\partial_t f_1(Y)-\partial_tf_1(Z)|}{\delta(Y,Z)^\beta}+\|\widetilde{k_2}\|_{W_2^{2,1}(Q_T)}\right),
		\end{aligned}
	\end{equation*}
where $C=C(M,\varphi^0,\beta,\gamma)$.\end{thm}

So far, we have proven the Theorem \ref{thm:theorem2}.
 
 \begin{rem}
 	By Theorem 4 in Chapter 5.9.2 of \cite{LC}, we have
 \begin{equation*}
 	\begin{aligned}
 		& \|\rho^{-\alpha+3}\partial^2_t{k}_1(t)\|_{L^\infty(0,T;H^1(M))}+	\|\partial^2_t{k}_2(t)\|_{L^\infty(0,T;H^1(M))}\\
 		&+	\|\rho^{-\alpha+3}\partial_t{k}_1\|_{L^\infty(0,T;H^3(M))}+\|\partial_t{k}_2\|_{L^\infty(0,T;H^3(M))}\\
 		\leqslant&C(\|\rho^{-\alpha+3}\widetilde{k}_1\|_{W^{4,2}_2(Q_T)}+\|\widetilde{k}_2\|_{W^{4,2}_2(Q_T)}).\\
 	\end{aligned}
 \end{equation*}
 This ensures that $D\mathcal{P}(\varphi^0,k)$ in \eqref{8'} is an isomorphism of $\mathcal{A}\times\mathcal{B}\rightarrow\mathcal{C}\times\mathcal{D}$.
 \end{rem}
 
\section{Existence of short time solutions}\label{section3}
In this section, we will utilize the inverse function theorem on Banach spaces to prove the existence of short time solutions. 
%+ \sum\limits_{|a|=5}\sup\limits_{x,y\in Q_T}\rho(x,y)^{5+\alpha-\gamma}\frac{|D^au(x)-D^au(y)|}{|x-y|^{\alpha}}
\begin{thm}\label{thm:exixtence}
	For $\varphi^0=(\varphi^0_1,\varphi^0_2)\in \mathcal{A}_0\times \mathcal{B}_0$,  there exists a $\delta>0$ and a function $\varphi=(\varphi_1,\varphi_2)\in \{\varphi^0_1+\mathcal{A}\}\times \{\varphi^0_2+\mathcal{B}\}$ solving (\ref{eq:equation}) on the domain $M\times [0,\delta)$.
\end{thm}

\begin{proof}
   According to Theorem \ref{thm:theorem2},
$D\mathcal{P}(\varphi^0,k)$ in \eqref{8'} is an isomorphism of $\mathcal{A}\times\mathcal{B}\rightarrow\mathcal{C}\times\mathcal{D}$. Therefore, by the inverse function theorem in Banach spaces, there exist a neighborhood $V$ of 0 in $\mathcal{A}\times\mathcal{B}$ and a neighborhood $W$ of $\mathcal{P}(\varphi^0)$ in $\mathcal{C}\times\mathcal{D}$ such that $\varphi^b\rightarrow \mathcal{P}(\varphi^b+\varphi^0):V\longrightarrow W$ is a homeomorphism.
Consider the function
\begin{equation}
   p(x,t)=\left\{
   \begin{aligned}
     &0, \enspace 0\leqslant t\leqslant \delta\\
     & \mathcal{P}(\varphi^0),\enspace 2\delta\leqslant t\leqslant T
   \end{aligned}
   \right.
\end{equation}
such that $p(x,t)$ is smooth for $t\in[\delta, 2\delta]$. It's easy to verify that $p(x,t)\in W$ when $\delta>0$ is sufficiently small. Then there exists an
unique $\varphi^b\in V\subset \mathcal{A}\times\mathcal{B}$ such that $\mathcal{P}(\varphi^b+\varphi^0)=0$ in $M\times[0,\delta)$.
\end{proof}

In the process of the proof, we need to pay attention to the selection of initial values $\varphi^0=(\varphi^0_1,\varphi^0_2)$ to ensure that $\mathcal{P}(\varphi^0)$ belongs to $	\mathcal{C}\times	\mathcal{D}$. 
\section{Existence and regularity of long time solutions}\label{section4}
To get the existence of long time solutions, the most important thing is obtaining a uniform estimate for the solutions.
\begin{lem}\label{lem:boundness}
Assume $\varphi=(\varphi_1,\varphi_2)$ is a solution in $[0,T)$ of \eqref{eq:equation}, then $\varphi_2$ is uniformly bounded in $M\times [0,T)$.
\end{lem}
\begin{proof}
We firstly denote $\Phi=(\Phi_1, \Phi_2)=(\varphi_1,h^\alpha e^{\varphi_2} )$, $\Phi^0=(\Phi^0_1,\Phi^0_2)=( \varphi^0_1,h^\alpha e^{\varphi^0_2})$.
Given the hypothesis that the curvature of $N$ is nonpositive, the Hessian $H$ of the distance function $d_N(\cdot,\cdot):N\times N\rightarrow \mathbb{R}$ is positive semi-definite. The tension of $\Phi$ is
$$\tau^\alpha(\Phi)=\Delta _g\Phi_\alpha+\Gamma ^\alpha_{ks}(\Phi)g(\nabla\Phi_k,\nabla\Phi_s),\alpha=1,2,$$
where $\Gamma ^\alpha_{ks}(\Phi)$ are the Christoffel symbols of $N$ evaluated along $\Phi$ and $\Delta_g$ is the Laplace- Beltrami operator on $M$.

Next, we utilize the method outlined in Lemma 1 from \cite{G2} to obtain an inequality.
Denoting $d_p(\cdot)=d_N(p,\cdot)$, through direct calculus, we obtain the result that
\begin{equation*}\begin{aligned}
		\Delta _g d(\Phi,\Phi^0)&=\nabla d_{\Phi}\tau(\Phi^0)+\nabla d_{\Phi^0}\tau(\Phi)+H(d\Phi+d\Phi^0,d\Phi+d\Phi^0),\\%\\&\geq\nabla d_{\Phi}\tau(\Phi^0)+\nabla d_{\Phi^0}\tau(\Phi),\\	
\partial_td(\Phi,\Phi^0)&=\nabla d_{\Phi^0}\cdot\partial_t{\Phi}.
	\end{aligned}
\end{equation*}Hence,
\begin{equation}\label{3.1}
	\begin{aligned}
		(\partial_t-\Delta_g)d(\Phi,\Phi^0)&=\nabla d_{\Phi_0}(\partial_t \Phi-\tau( \Phi))-\nabla d_{\Phi}\tau( \Phi^0)-H(d\Phi+d\Phi^0,d\Phi+d\Phi^0)
		\\&\leq \nabla d_{\Phi_0}(\partial_t \Phi-\tau( \Phi))-\nabla d_{\Phi}\tau( \Phi^0).
	\end{aligned}
\end{equation}
Furthermore we see from equations \eqref{eq:equation} that
$$\partial_t\Phi _2-\tau(\Phi_2)=h^\alpha e^{\varphi_2}(\partial_t\varphi_2-\Delta\varphi_2-h^{-2\alpha}e^{2\varphi_2}|\nabla \varphi_1|^2-\Delta (\log h))=0,$$
$$\partial_t\Phi _1-\tau(\Phi_1)=\partial_t\varphi_1-\Delta\varphi_1-2\nabla \varphi_1\varphi_2+2\alpha\frac{\nabla \varphi_1\nabla h}{h}=0.$$
Therefore \eqref{3.1} implies
\begin{align*}
	(\partial_t-\Delta_g)d(\Phi,\Phi^0)
	&\leq -\nabla d_{\Phi}\tau( \Phi^0)=:g(x)\leq G
\end{align*}
where $G=\max\limits_{M}|g(x)|$.

Since $M$ is assumed to be bounded, we can define $d$ as the diameter of $M$. Taking a point $\widehat{x}\in M$, define
$
 q(x,t)=\frac{G}{6}(d^2-|x-\widehat{x}|^2)-d(\Phi,
\Phi^0).
$
Then
$
 (\partial_t-\Delta)q(x,t)\geq G-g\geq 0.
$
By the weak maximum principle, we have
\begin{align*}
 \min\limits_{Q_T}q(x,t)=\min\limits_{\{t=0\}}q(x,t)\geqslant0.
\end{align*}
Thus, $q(x,t)\geq 0$ in $Q_T$. That is $d(\Phi,
\Phi^0)\leq \frac{G}{6}(d^2-|x-\widehat{x}|^2)\leq\frac{Gd^2}{6}$ in $Q_T$.

By the definition of the distance on the hyperplane $\mathbb{H}^2$ \cite{P},
\begin{align*}
 d(\Phi,\Phi^0)&=2\tanh^{-1}\sqrt{\frac{(\Phi_1-\Phi_1^0)^2+(\Phi_2-\Phi_2^0)^2}{(\Phi_1-\Phi_1^0)^2+(\Phi_2+\Phi_2^0)^2}}\\
&=2\tanh^{-1}\sqrt{\frac{h^{2\alpha}(e^{-\varphi_2}-e^{-\varphi^0_2})^2+(\varphi_1-\varphi_1^0)^2}{h^{2\alpha}(e^{-\varphi_2}+e^{-\varphi^0_2})^2+(\varphi_1-\varphi_1^0)^2}}\\
&\geq  %2\tanh^{-1}\sqrt{\frac{(e^{-\varphi_2}-e^{-\varphi^0_2})^2}{(e^{-\varphi_2}+e^{-\varphi^0_2})^2}}=
2\tanh^{-1}\frac{|e^{-\varphi_2}-e^{-\varphi^0_2}|}{e^{-\varphi_2}+e^{-\varphi^0_2}}\\
&=\ln \frac{1+\frac{|e^{-\varphi_2}-e^{-\varphi^0_2}|}{e^{-\varphi_2}+e^{-\varphi^0_2}}}{1-\frac{|e^{-\varphi_2}-e^{-\varphi^0_2}|}{e^{-\varphi_2}+e^{-\varphi^0_2}}}=|\varphi_2-\varphi_2^0|.
\end{align*}
Therefore,  $|\varphi_2(x,t)|\leq |\varphi^0_2(x)|+ \frac{Gd^2}{6}$, $\forall (x,t)\in M\times[0,T)$.
\end{proof}

Next lemma will also be used in the later section to show the convergence of the solution.
\begin{lem}\label{lem:L^2decay}
There exist positive constants $C$ depending on $M,\varphi^0$ and $C_0$ depending on $M$ such that for every
$t\in (0,T)$ and $x\in M\backslash \Gamma$, there hold
\begin{align*}
\int_M h^{-2\alpha}e^{-2\varphi_2}|\partial_t \varphi_1|^2+|\partial_t \varphi_2|^2 \ud x(t)\leq Ce^{-\frac{C_0}{2}t}
\end{align*}
and
\begin{align*}
\rho^{\frac{3}{2}-\alpha} |\partial_t \varphi_1|+\rho^{\frac{3}{2}}|\partial_t \varphi_2| (x,t) \leq Ce^{-\frac{C_0}{4}t} .
\end{align*}
\end{lem}
\begin{proof}
Take $\theta(x,t):=|\partial_t \Phi|_{{\mathbb{H}^2}}^2(x,t)=h^{-2\alpha}e^{-2\varphi_2}|\partial_t \varphi_1|^2+|\partial_t \varphi_2|^2$.
By the Bochner identity for harmonic maps, we can calculate that
\begin{equation}\label{eq:partialtbochner}
(\partial_t-\Delta)(|\partial_t \Phi|_{{\mathbb{H}^2}}^2)=-|\nabla \partial_t\Phi|^2+R^{\mathbb{H}^2}(\nabla\Phi,\partial_t\Phi,\nabla\Phi,\partial_t \Phi)\leq 0
\end{equation}
where $R^{\mathbb{H}^2}$ is the Riemannian curvature tensors of $\mathbb{H}^2$.
Then
\begin{equation}\label{kj}
\begin{aligned}
\frac{d}{dt}\int_M \theta^2 \ud x =2\int_M \theta \partial_t\theta \ud x\leq 2\int_M \theta \Delta \theta \ud x= -2\int_M |\nabla \theta|^2 \ud x.
 \end{aligned}
\end{equation}
 According to the Poincar\'{e} inequality
 \begin{align*}
-2\int_M |\nabla \theta|^2 \ud x\leq -C_0 \int_M \theta^2\ud x.
\end{align*}
Thus we have
 \begin{align*}
\frac{d}{dt}\int_M \theta^2 \ud x \leq -C_0 \int_M \theta^2\ud x.
\end{align*}
 It follows that
 \begin{align*}
\int_M \theta^2 \ud x (t)\leq \int_M \theta^2 \ud x (0)e^{-C_0t}
\end{align*}
 \[
\theta(x,0)=h^{-2\alpha}e^{-2\varphi^0_2}|\partial_t\varphi_1|^2(0)+|\partial_t \varphi_2|^2(0).
\]
 Regarding the equation \eqref{eq:equation}, we have
  \begin{align*}
\partial_t \varphi_1(0)=\Delta\varphi^0_1-2(\nabla\varphi^0_2+\frac{\alpha\nabla h}{h})\nabla\varphi_1^0,\\
\partial_t \varphi_2(0)=\Delta\varphi_2^0+h^{-2\alpha}e^{-2\varphi^0_2}|\nabla \varphi^0_1|^2.
\end{align*}
 By the initial condition, we have $\int_M \theta^2 \ud x (0)\leq C$
 and
 \begin{align*}
\int_M \theta^2 \ud x (t)\leq Ce^{-C_0t}.
\end{align*}
Using H\"{o}lder inequality, we obtain
\begin{align*}
 \int_M h^{-2\alpha}e^{-2\varphi_2}|\partial_t \varphi_1|^2+|\partial_t \varphi_2|^2 \ud x(t)=\int_M \theta \ud x\leq C(\int_M \theta^2 \ud x(t))^{\frac{1}{2}}\leq Ce^{-\frac{C_0}{2}t}.
\end{align*}
For any $(x,s)\in M\setminus \Gamma\times[0,T)$ and $\overline{\rho}=\rho(x)$, \eqref{eq:partialtbochner} implies
\begin{align*}
 h^{-2\alpha}e^{-2\varphi_2}|\partial_t \varphi_1|^2+|\partial_t \varphi_2|^2 (x,s) &\leq C\overline{\rho}^{-5}\int^s_{s-(\frac{\overline{\rho}}{2})^2}\int_Mh^{-2\alpha}e^{-2\varphi_2}|\partial_t\varphi_1|^2+|\partial_t\varphi_2|^2 \ud y \ud t\\
 &\leq C  \overline{\rho}^{-5}\int^s_{s-(\frac{\overline{\rho}}{2})^2} e^{-\frac{C_0}{2}t}\ud t\\
 &\leq C  \overline{\rho}^{-3} e^{-\frac{C_0}{2}s}.
\end{align*} 
 Therefore, $\rho^{\frac{3}{2}-\alpha} |\partial_t \varphi_1|+\rho^{\frac{3}{2}}|\partial_t \varphi_2| (x,s) \leq C e^{-\frac{C_0}{4}s}.$
\end{proof}

\begin{rem}\label{rem}
	In the last section, it has been proven that $(\varphi_1,\varphi_2)\in \{\varphi^0_1+\mathcal{A}\}\times \{\varphi^0_2+\mathcal{B}\} $ and it ensures that several integrals of \eqref{kj} are finite in the case of $\alpha>1$.
\end{rem}

The next part is a  generalization of Theorem 1.1 in Li-Tian \cite{LT}. We fix $X=(x,t)\in Q_T$ and assume $Q_R(X)=(t-R^2,t]\times B_R(x)$ in the following.
\begin{lem}\label{lem:poincare}
There exists $C>0$ such that for every $u\in
W^{2,1}_{2}(Q_R(X))$, there holds
\begin{align*}
\int_{Q_R(X)}|u-u_{Q_R(X)}|^2 \ud y \ud s \leq C\int_{Q_R(X)}R^2|\nabla u|^2 + R^4|\partial_t u|^2\ud y \ud s
\end{align*}
where $u_{Q_R(X)}= \frac{1}{|Q_R(X)|}\int_{Q_R(X)}u(y,s) \ud y\ud s$.
\end{lem}
\begin{proof}

Let
\[
v(z,r)=u(x+Rz,t+R^2r),v\in W^{2,1}_{2}(Q_1(0)).
\]
Note
\[
(v(\cdot,r))_{B_1{(0)}}=\frac{1}{|B_1{(0)|}}\int_{B_1{(0)}} v(\overline{z},r) \ud \overline{z}(r).
\]
Using triangular inequality and Poincar\'{e} inequality in $H^{1}(B_1)$, we have
\begin{align*}
&\int_{Q_1(0)}|v-v_{Q_1(0)}|^2 \ud z\ud r \\
\leq &2\int_{Q_1(0)}|v-(v(\cdot,r))_{B_1{(0)}}|^2 \ud z\ud r+2\int_{Q_1(0)}|(v(\cdot,r))_{B_1{(0)}}-v_{Q_1(0)}|^2 \ud z \ud r\\
\leq &C(\int_{Q_1(0)}|\nabla v|^2 \ud z \ud r+\int_{Q_1(0)}|(v(\cdot,r))_{B_1{(0)}}-v_{Q_1(0)}|^2 \ud z \ud r).
\end{align*}
The direct computations and the H\"{o}lder
inequality yield that
\begin{align*}
&\int_{Q_1(0)}|(v(\cdot,r))_{B_1{(0)}}-v_{Q_1(0)}|^2 \ud z \ud r\\
=&\int_{-1}^0\int_{B_1(0)}\left|\frac{1}{|B_1{(0)|}}\int_{B_1{(0)}} v(\overline{z},r) \ud \overline{z}-\frac{1}{|Q_1(0)|}\int_{Q_1(0)}v(\overline{z},\tau) \ud \overline{z}\ud \tau\right|^2 \ud z \ud r\\
\leq & C\int_{-1}^0\left|\int_{-1}^0\int_{B_1{(0)}}v(\overline{z},r) - v(\overline{z},\tau) \ud\overline{z}\ud \tau\right|^2 \ud r\\
\leq & C\left|\int_{B_1{(0)}} \int_{-1}^0\partial_{\overline{r}} v(\overline{z},\overline{r}) \ud \overline{r} \ud \overline{z}\right|^2\\
\leq & C \int_{-1}^0\int_{B_1{(0)}} |\partial_{r}v(z,r)|^2 \ud z \ud r .
\end{align*}
Thus
\begin{align*}
\int_{Q_1(0)}|v-v_{Q_1(0)}|^2 \ud z\ud r \leq C\int_{Q_1(0)}|\nabla v|^2 +|\partial_{r} v|^2\ud z \ud r
\end{align*}
and
\begin{align*}
\int_{Q_R(X)}|u-u_{Q_R(X)}|^2 \ud y \ud s \leq C\int_{Q_R(X)}R^2|\nabla u|^2 + R^4|\partial_t u|^2\ud y \ud s.
\end{align*}
\end{proof}

From now on, we fix a point $x_0$ in $\Gamma\subset M$.
\begin{lem}\label{lem:bounded}

There is a uniform constant C such that for any $x
\in B_{\frac{1}{2}}(x_0),t\in [0,T)$,
\begin{equation}\label{3.6}
\begin{aligned}
\int_{B_\frac{1}{4}(x)} |y-x|^{-1}(h^{-2\alpha}e^{-2\varphi_2}|\nabla\varphi_1|^2+|\nabla\varphi_2|^2)\ud y(t)\leq C.
\end{aligned}
\end{equation}
 \end{lem}
\begin{proof}
The desired result can be proven by following the proofs of Lemma 3.1 and Lemma 3.2 in \cite{LT}, as well as Lemma \ref{lem:L^2decay} in this paper.\end{proof}
Similar to the proof of Lemma 3.3 of \cite{LT}, we can get $|\nabla \varphi_1|(x,t)\leq C\rho(x)^{\alpha-1}$, $(x,t)\in M\times[0,T)$.

For the sake of brevity in writing, we denote
\begin{align*}
&f_{\sigma}(x,t)=\sigma^{-1}\int_{B_{\sigma}(x)}h^{-2\alpha}e^{-2\varphi_2}|\nabla \varphi_1|^2+|\nabla\varphi_2|^2 \ud y(t),\\
&g_{\sigma}(x,t)=\sigma\int_{B_{\sigma}(x)}h^{-2\alpha}e^{-2\varphi_2}|\partial_t \varphi_1|^2+|\partial_t\varphi_2|^2\ud y(t),\\
&E_{\sigma}(x,t)=f_{\sigma}(x,t)+g_{\sigma}(x,t).
\end{align*}

\begin{prop}\label{prop:smallnessenergy}
For any $\varepsilon>0$, $x\in M$, there is a $\sigma_x$ between $\frac{1}{8}e^{-\frac{C}{\varepsilon}}$ and $\frac{1}{4}$ such that for any $t\in [0, T)$,
\begin{align*}
E_{\sigma_x}(x,t)\leq \varepsilon
\end{align*}
where C is a uniform constant of \eqref{3.6}.
\end{prop}
\begin{proof}

By the proof of Proposition 3.1 of \cite{LT} and
Lemma \ref{lem:bounded} in this paper, we know
\begin{equation}\label{eq:f}
\int_0^{\frac{1}{4}}\frac{f_{\sigma}(x,t)}{\sigma}\ud \sigma\leq C.
\end{equation}
 The definition of $g_{\sigma}(x,t)$ and Lemma \ref{lem:L^2decay} yield
\begin{equation}\label{eq:g}
\int_0^{\frac{1}{4}}\frac{g_{\sigma}(x,t)}{\sigma}\ud \sigma \leq \int_0^{\frac{1}{4}}\int_{M}h^{-2\alpha}e^{-2\varphi_2}|\partial_t\varphi_1|^2+|\partial_t\varphi_2|^2 \ud y \ud \sigma\leq C.\\
\end{equation}
Adding \eqref{eq:f} and \eqref{eq:g}, we have
\begin{equation*}
\int_0^{\frac{1}{4}}\frac{E_{\sigma}(x,t)}{\sigma}\ud \sigma\leq C.
\end{equation*}
\end{proof}

Next, we will use the method of Lemma 3.3 of \cite{HX} to prove
the following lemma.
\begin{lem}\label{lem:Caccioppoli-inequality}
For  any $ \beta_1,\beta_2$ with $|\beta_2|\leq 2\|\varphi_2\|_{L^{\infty}(Q_T)}$, $0<\sigma<\frac{1}{4}$ and $x \in B_{\frac{1}{2}}(x_0)$, $t>0$, we have
\begin{align*}
E_{\frac{\sigma}{2}}(x,t)\leq &C\left(\sigma^{-3}\int_{B_{\sigma}(x)}h^{-2\alpha}e^{-2\varphi_2}|\varphi_1-\beta_1|^2+| \varphi_2-\beta_2|^2 \ud y(t)\right.\\
&\left.+\sigma\int_{B_{\sigma}(x)}h^{-2\alpha}e^{-2\varphi_2}|\partial_t  \varphi_1|^2+|\partial_t\varphi_2|^2\ud y(t)\right).
\end{align*}
\end{lem}
\begin{proof}

Let $\eta\in C_c^{\infty}(B_{\sigma}(x))$ be a cutoff function with $0\leq \eta \leq 1$, $\eta\equiv 1$ in $B_{\frac{\sigma}{2}}(x)$, and $|\nabla\eta|\leq10\sigma^{-1}$. Taking $h^{-2\alpha} e^{-2\varphi_2} \eta^2( \varphi_1-\beta_1)$ as a test function of the first equation of \eqref{eq:equation}, we have for any $t>0$,                              
\begin{align*}
\int_{B_{\sigma}(x)}\partial_t\varphi_1 h^{-2\alpha}e^{-2\varphi_2} \eta^2( \varphi_1-\beta_1)\ud y(t)=-
\int_{B_{\sigma}(x)} h^{-2\alpha} e^{-2\varphi_2}\nabla\varphi_1 \nabla[\eta^2( \varphi_1-\beta_1)]\ud y(t)\\
=-\int_{B_{\sigma}(x)} h^{-2\alpha} e^{-2\varphi_2}\eta^2 |\nabla\varphi_1|^2\ud y(t)-2\int_{B_{\sigma}(x)} h^{-2\alpha} e^{-2\varphi_2}\eta\nabla \varphi_1\nabla\eta( \varphi_1-\beta_1)\ud y(t).
\end{align*}
By Cauchy inequality, it follows that
\begin{equation}\label{eq:gradient2}
\begin{aligned}
&\int_{B_{\sigma}(x)} h^{-2\alpha} e^{-2\varphi_2}\eta^2 |\nabla\varphi_1|^2\ud y(t)\\
&\leq \frac{C}{\sigma^2}\int_{B_{\sigma}(x)}h^{-2\alpha} e^{-2\varphi_2}| \varphi_1-\beta_1|^2\ud y(t)+C\sigma^2\int_{B_{\sigma}(x)}|\partial_t\varphi_1|^2 h^{-2\alpha} e^{-2\varphi_2}\eta^2\ud y(t).
\end{aligned}
\end{equation}
Using $ \eta^2( \varphi_2-\beta_2)$ as a test function of the second equation of \eqref{eq:equation} , we have
\begin{align*}
\int_{B_{\sigma}(x)}\partial_t\varphi_2\eta^2( \varphi_2-\beta_2)+\nabla\varphi_2\nabla[\eta^2(\varphi_2-\beta_2)]\ud y(t)=\int_{B_{\sigma}(x)}h^{-2\alpha} e^{-2\varphi_2}|\nabla \varphi_1|^2\eta^2( \varphi_2-\beta_2)\ud y(t).
 \end{align*}
Thus Cauchy inequality and \eqref{eq:gradient2} imply
\begin{equation}\label{eq:gradient1}
\begin{aligned}
\int_{B_{\sigma}(x)}  \eta^2 |\nabla\varphi_2|^2\ud y(t)\leq &\frac{C}{\sigma^2}\int_{B_{\sigma}(x)}h^{-2\alpha} e^{-2\varphi_2}|\varphi_1-\beta_1|^2+|\varphi_2-\beta_2|^2\ud y(t) \\
&+C\sigma^2\int_{B_{\sigma}(x)}(h^{-2\alpha} e^{-2\varphi_2}|\partial_t\varphi_1|^2+|\partial_t\varphi_2|^2) \eta^2\ud y(t).
 \end{aligned}
\end{equation}
We finally add \eqref{eq:gradient2} and \eqref{eq:gradient1} to discover
\begin{equation*}
\begin{aligned}
\int_{B_{\frac{\sigma}{2}}(x)} h^{-2\alpha} e^{-2\varphi_2} |\nabla\varphi_1|^2+|\nabla\varphi_2|^2 \ud y(t)\leq &\frac{C}{\sigma^2}\int_{B_{\sigma}(x)}h^{-2\alpha} e^{-2\varphi_2}| \varphi_1-\beta_1|^2+|\varphi_2-\beta_2|^2\ud y(t) \\
&+C\sigma^2\int_{B_{\sigma}(x)}(h^{-2\alpha} e^{-2\varphi_2}|\partial_t\varphi_1|^2+|\partial_t\varphi_2|^2 \eta^2)\ud y(t).
 \end{aligned}
\end{equation*}                                                                                                                                     The Lemma follows from the above equality.
\end{proof}
\begin{prop}\label{prop:diedai1}

There are $\varepsilon_0$ and $\lambda_0$, independent of $x$ in $B_{\frac{1}{2}}(x_0)$ and $t>0$, such that whenever
\begin{align*}
E_{\sigma}(x,t)\leq \varepsilon_0, 0<\sigma<\frac{1}{4}, x \in B_{\frac{1}{2}}(x_0), t>0,
\end{align*}
we have
\begin{align*}
E_{\lambda_0\sigma}(x,t)\leq \frac{1}{2}E_{\sigma}(x,t).
\end{align*}
\end{prop}
\begin{proof}

Suppose this proposition is not true. Then there are sequences of $\{x_i\}_{i\geq1}$ in
$B_{\frac{1}{2}}(x_0)$, $\{\varepsilon_i\}_{i\geq1}$,  $\{t_i\}_{i\geq1}$ and $\{\sigma_i\}_{i\geq1}$  such that $0<\sigma<\frac{1}{4}$, $t_i>0$, $\lim\limits_{i\to \infty}\varepsilon_i=0$
and
\begin{equation}
\begin{aligned}
\varepsilon_i^2&\geq E_{\sigma_i}(x_i,t_i), 
\label{eq:assume1}\\
 \end{aligned}
\end{equation}
\begin{equation}
\begin{aligned}\label{eq:assume2}
E_{\lambda_0\sigma_i}(x_i,t_i)&>\frac{1}{2}
E_{\sigma_i}(x_i,t_i), 
%=\frac{1}{2}\varepsilon_i^2\\
 \end{aligned}
\end{equation}
where $\lambda_0$ is a small number to be determined.

For simplicity of notations, we assume that $B_1(x_0)$ is the Euclidean ball in $\mathbb{R}^3$ with $x_0$ being the origin and $\Gamma\subset \mathbb{R}\subset\mathbb{R}^3$.

Case $\MakeUppercase{\romannumeral 1}$: there is a subsequence of $\{x_i\}_{i\geq1}$ such that for any $x_i$ in this subsequence, $B_{\frac{\sigma_i}{l_0}}(x_i)\cap \Gamma=\varnothing$, where $l_0$ will be determined later prior to the determination of $\lambda_0$. For simplicity, assume $\{x_i\}_{i\geq1}$ is just this sequence. Define functions in $Q_1(0)$ as follows: for all $i\geq1,$
\begin{equation}
\begin{aligned}
\psi_1^i(x,t)&=\frac{1}{\rho(x_i)^\alpha\varepsilon_i}(\varphi_1(x_i+\sigma_ix,t_i+\sigma_i^2t)-\overline{\varphi}_1^i),\\
\psi_2^i(x,t)&=\frac{1}{\varepsilon_i}(\varphi_2(x_i+\sigma_ix,t_i+\sigma_i^2t)-\overline{\varphi}_2^i),\\
 \end{aligned}
\end{equation}
where $\overline{\varphi}_1^i,
\overline{\varphi}_2^i$ are the averages of ${\varphi}_1,{\varphi}_2$ over $Q_{\frac{\sigma_i}{2l_0}}(x_i,t_i).$
Note that $y\in B_{\frac{\sigma_i}{2l_0}}(x_i),\frac{1}{2}\rho(x_i) <\rho(y)<2\rho(x_i)$. For any $t>0$, $E_{\sigma_i}(x_i,t_i)\leqslant \varepsilon_i^2$ implies that
\begin{equation*}
\begin{aligned}
\int_{Q_{\frac{1}{2l_0}}(0)}|\nabla\psi_j^i|^2+|\partial_t \psi_j^i|^2\ud V\ud t\leq C, \quad i\geq 0, \ j=1,2.
%\int_{B_{\frac{1}{2l_0}}(0)}|\partial_t\psi_j^i|^2(x,t)\ud x\leq C.
 \end{aligned}
\end{equation*}
By the fact that the averages of $\psi_j^i$ in $Q_{\frac{1}{2l_0}}(0)$ are zeros and the Poincar\'{e} inequality, we have
\begin{equation*}
 \begin{aligned}
\|\psi_j^i\|_{L^2(Q_{\frac{1}{2l_0}}(0))}\leq C, \quad i\geq1, \ j=1,2.\\
%\int_{B_{\frac{1}{2l_0}}(0)}|\partial_t\psi_j^i|^2(x,t)\ud x\leq C.
 \end{aligned}
\end{equation*}
By taking subsequence, we can get that $\{\psi_j^i\}_{i\geq 1}$ converge weakly to $\psi_j$ in $L^2(-\frac{1}{4l^2_0},0; H^1(B_{\frac{1}{2l_0}}(0))$ and $\{\partial_s \psi_j^i\}_{i\geq 1}$ converge weakly to $\partial_s \psi_j$ in $L^2(-\frac{1}{4l^2_0},0; L^2(B_{\frac{1}{2l_0}}(0))$. The compactness theorem\cite{LC} shows that $\{\psi_j^i\}_{i\geq1}$ converge to $\psi_j$ in $L^2(Q_{\frac{1}{2l_0}}(0))$. 
For all $ -\frac{1}{4l^2_0}\leqslant \overline{s},s\leqslant 0,  i\geq1, j=1, 2$,
\begin{equation}\label{xin}
	\begin{aligned}
		&\|\psi_j^i(s)-\psi_j(s)\|^2_{L^2(B_{\frac{1}{2l_0}}(0))}-\|\psi_j^i(\overline{s})-\psi_j(\overline{s})\|^2_{L^2(B_{\frac{1}{2l_0}}(0))}\\
		=&2\int_{\overline{s}}^s\int_{B_{\frac{1}{2l_0}}(0)} (\partial_{\tau}\psi_j^i(\tau)-\partial_{\tau}\psi_j(\tau))(\psi_j^i(\tau)-\psi_j(\tau))\ud V\ud \tau.
		%\int_{B_{\frac{1}{2l_0}}(0)}|\partial_t\psi_j^i|^2(x,t)\ud x\leq C.
	\end{aligned}
\end{equation}
Integrating \eqref{xin} with respect to $\overline{s}$, we get
\begin{equation}
	\begin{aligned}
		&\frac{1}{4l_0^2}\|\psi_j^i(s)-\psi_j(s)\|^2_{L^2(B_{\frac{1}{2l_0}}(0))}\\
		\leqslant&\|\psi_j^i-\psi_j\|^2_{L^2(Q_{\frac{1}{2l_0}}(0))}+C \int_{-\frac{1}{4l_0^2}}^0\int_{B_{\frac{1}{2l_0}}(0)} (\partial_\tau\psi_j^i-\partial_\tau\psi_j) (\psi_j^i-\psi_j)\ud V\ud \tau .\\
		%\int_{B_{\frac{1}{2l_0}}(0)}|\partial_t\psi_j^i|^2(x,t)\ud x\leq C.
	\end{aligned}
\end{equation}
We know $\partial_\tau \psi_j^i-\partial_\tau \psi_j\rightharpoonup 0$ and $\psi_j^i-\psi_j\rightarrow 0$ in $L^2(-\frac{1}{4l_0^2},0;L^2(B_{\frac{1}{2l_0}}(0)))$, then
\begin{equation*}
	\begin{aligned}
	\sup\limits_{s\in [-\frac{1}{4l^2_0},0]}\|\psi_j^i(s)-\psi_j(s)\|^2_{L^2(B_{\frac{1}{2l_0}}(0))}\rightarrow 0 \mbox{ as } i\rightarrow \infty.
	\end{aligned}
\end{equation*}

According to \eqref{eq:equation}, we can show that
\begin{equation}
\left\{
\begin{aligned}
\partial_t \psi_1-\Delta \psi_1&=\nabla\theta\nabla\psi_1 \  &\mbox{ in } Q_{\frac{1}{2l_0}}(0)\\
\partial_t \psi_2-\Delta\psi_2&=0 \enspace  &\mbox{ in } Q_{\frac{1}{2l_0}}(0),\\
 \end{aligned}
  \right.
  \end{equation}
where $\theta$ is a smooth and uniformly bounded
function. Then the standard parabolic theorem
yields
\begin{equation}
|\psi_j(x,t)-\psi_j(0,0)|\leq C\delta((x,t),(0,0))\
\mbox{ for } (x,t) \ \in Q_{\frac{1}{2l_0}}(0), j=1,2.
  \end{equation}
Applying Lemma \ref{lem:Caccioppoli-inequality}
with  $ \varepsilon_i\rho(x_i)^\alpha\psi_1(0,0)+\overline{\varphi}_1^i$ as $
\beta_1$ and $\varepsilon_i\psi_2(0,0)+\overline{\varphi}_2^i$ as $\beta_2$, for any $\tau\in (t_i-\frac{\sigma_i^2}{4l_0^2},t_i]$, we have
\begin{equation*}
\begin{aligned}
&E_{\lambda_0\sigma_i}(x_i,\tau)\\ \leq &  C(2\lambda_0\sigma_i)^{-3}\int_{B_{2\lambda_0\sigma_i}(x_i)}h^{-2\alpha}e^{-2\varphi_2}| \varphi_1-\varepsilon_i\rho(x_i)^\alpha\psi_1(0,0)-\overline{\varphi}_1^i|^2\\
& + |\varphi_2-\varepsilon_i\psi_2(0,0)-\overline{\varphi}_2^i|^2\ud y(\tau)+2\lambda_0\sigma_i\int_{B_{2\lambda_0\sigma_i}(x_i)}h^{-2\alpha}e^{-2\varphi_2}|\partial_t \varphi_1|^2+|\partial_t \varphi_2|^2\ud y (\tau)
\\ 
=&C \varepsilon_i^2(2\lambda_0)^{-3}\int_{B_{2\lambda_0}(0)}|\psi^i_1(x,t)-\psi_1(0,0)|^2+|\psi^i_2(x,t)-\psi_2(0,0)|^2 \ud x(\frac{\tau-t_i}{\sigma_i^2})+2\lambda_0\varepsilon_i^2\\
\leq &C\varepsilon_i^2(2\lambda_0)^{-3}\int_{B_{2\lambda_0}(0)}|\psi^i_1(x,t)-\psi_1(x,t)|^2+|\psi^i_2(x,t)-\psi_2(x,t)|^2 \\
   &+|\psi_1(x,t)-\psi_1(0,0)|^2+|\psi_2(x,t)-\psi_2(0,0)|^2\ud x(\frac{\tau-t_i}{\sigma_i^2})+2\lambda_0\varepsilon_i^2\\
   \leq & C_1\varepsilon_i^2(2\lambda_0)^{-3}\int_{B_{2\lambda_0}(0)}|\psi^i_1(x,t)-\psi_1(x,t)|^2+|\psi^i_2(x,t)-\psi_2(x,t)|^2 \ud x(\frac{\tau-t_i}{\sigma_i^2})\\
   &+C_2(\lambda_0^2+2\lambda_0)\varepsilon_i^2.
 \end{aligned}
  \end{equation*}
We choose $l_0, \lambda_0$ such that $l_0>\frac{1}
{2}, C_2(\lambda_0^2+2\lambda_0)<\frac{1}{4}$ and take $i$ large enough such that
$$C_1(2\lambda_0)^{-3}\int_{B_{2\lambda_0}(0)}|\psi^i_1(x,t)-\psi_1(x,t)|^2+| \psi^i_2(x,t)-\psi_2(x,t)|^2 \ud x(\frac{\tau-t_i}{\sigma_i^2})<\frac{1}{4}.$$ Therefore, we have
 \begin{equation*}
\begin{aligned}
E_{\lambda_0\sigma_i}(x_i,\tau)\leq \frac{1}{2}\varepsilon_i^2, \forall \tau\in (t_i-\frac{\sigma_i^2}{4l_0^2},t_i].
 \end{aligned}
  \end{equation*}
 It contradicts to inequality \eqref{eq:assume2}.

Case $\MakeUppercase{\romannumeral 2}$:
$B_{\frac{\sigma_i}{l_0}}(x_i)\cap\Gamma\neq\varnothing$ for all $i$. Define functions in $B_1(0)$ as follows: for any $i\geq 1$
\begin{equation}
\begin{aligned}
	\psi_1^i(x,t)&=\frac{1}{\sigma_i^\alpha\varepsilon_i}\varphi_1(x_i+\sigma_ix,t_i+\sigma_i^2t),\\
\psi_2^i(x,t)&=\frac{1}{\varepsilon_i}(\varphi_2(x_i+\sigma_ix,t_i+\sigma_i^2t)-\overline{\varphi}_2^i),\\
 \end{aligned}
\end{equation}
where $\overline{\varphi}_2^i$ is the averages of ${\varphi}_2$ over $Q_{\sigma_i}(x_i, t_i)$.
Under the transformation: $x \to x_i+\sigma_ix$, the preimages of $\Gamma$ are $\Gamma_i$, parallel to $\Gamma$ and in the distance $\frac{\pi(x_i)}
{\sigma_i}$, where $\pi$ is the orthogonal projection from $M$ to $\Gamma$. Since $B_{\frac{\sigma_i}{l_0}}(x_i)\cap \Gamma\neq\varnothing$, we have $|\frac{\pi(x_i)}{\sigma_i} |\leq \frac{1}{l_0}$. Thus we can assume that $\Gamma_i$ converge to $\Gamma_{\infty}$ within the distance $\frac{1}{l_0}$ from $N$. Let $\rho_i, \rho_{\infty}$ be the distance from $\Gamma_i, \Gamma_{\infty}$.
 By direct computation, we can get for any $t>0$,
\begin{equation*}
\begin{aligned}
\int_{B_{\frac{1}{2l_0}}(0)}\rho_i^{-2\alpha}(|\partial_t\psi_1^i|^2+|\nabla\psi_1^i|^2)+(|\nabla\psi_2^i|^2+|\partial_t\psi_2^i|^2)\ud x(t)\leq C,
\ i\geq 0.
%\int_{B_{\frac{1}{2l_0}}(0)}|\partial_t\psi_j^i|^2(x,t)\ud x\leq C.
 \end{aligned}
\end{equation*}
By Corollary 4.1 of \cite {LT},
\begin{equation*}
\begin{aligned}
\int_{B_{1}(0)}\rho_i^{-2\alpha-2}|\psi_1^i|^2(x,t)\ud x\leq C, \ i\geq 0.
 \end{aligned}
\end{equation*}
As in the previous case, $\{\psi_2^i\}_{i\geq 1}$ converge  to $\psi_2$ in $L^\infty(-\frac{1}{4l_0^2},0; L^2(B_{\frac{1}{2l_0}}(0)))$ and $\{\psi_1^i\}_{i\geq 1}$ converge weakly to $\psi_1$ in $H^1(B_{\frac{1}{2l_0}}(0); \rho_i^{-\alpha})$, where 
\[
\|\psi_1\|_{H^1(B_{\frac{1}{2l_0}}(0); \rho_i^{-\alpha})}=\int_{B_{1}(0)}(\rho_i^{-2\alpha}|\nabla\psi_1^i|^2+
 \rho_i^{-2\alpha-2}|\psi_1^i|^2)(x, t)\ud x, \ i\geq 1.
 \] 
 $\{\partial_t\psi_1^i)\}_{i\geq 1}$
converge  to $\partial_t\psi_1$ in $L^2(0,T;L^2(B_{\frac{1}{2l_0}}(0);\rho_i^{-\alpha} ))$ and 
 $\{\psi_1^i(\cdot,t)\}_{i\geq 1}$
converge  to $\psi_1(\cdot,t)$ in $L^2(B_{\frac{1}{2l_0}}(0);\rho_i^{-\alpha} )$, $t\in (-\frac{1}{4l_0^2},0)$, where
\begin{equation*}
\begin{aligned}
\|\psi_1\|_{L^2(B_{\frac{1}{2l_0}}(0);\rho_i^{-\alpha} )}=\int_{B_{1}(0)}\rho_i^{-2\alpha}|\psi_1^i|^2(x,t)\ud x,
\ i\geq 0.
 \end{aligned}
\end{equation*}
 According to \eqref{eq:equation}, we have
\begin{equation}
\left\{
\begin{aligned}
\partial_t \psi_1-\Delta \psi_1&=-2\nabla\psi_1(\nabla\theta+\alpha\frac{\nabla\rho_{\infty}}{\rho_{\infty}}) \ & \mbox{ in } Q_{\frac{1}{2l_0}}(0)\\
\partial_t \psi_2-\Delta \psi_2&=0 \ &\mbox{ in } Q_{\frac{1}
{2l_0}}(0),
 \end{aligned}
  \right.
  \end{equation}
where $\theta$ is a smooth outside $
\Gamma_{\infty}$ in $B_1(0)$ and $|\nabla \theta|
\leq C \rho_{\infty}^{-\delta}$ for any $
\delta<1$. Then the standard parabolic theory
yields
\begin{equation}
|\psi_2(x,t)-\psi_2(0,0)|\leq C\delta((x,t), (0,0)), (x,t) \in Q_{\frac{1}{2l_0}}(0).
  \end{equation}

Define the functions $\widetilde{\psi}_1=\rho_{\infty}^{-\alpha}\psi_1$ and then $\widetilde{\psi}_1$ satisfies
\begin{equation}\label{76}
\partial_t \widetilde{\psi}_1-\Delta\widetilde{\psi}_1=-2\nabla \widetilde{\psi}_1\nabla\theta-\frac{\alpha\widetilde{\psi}_1}{\rho_{\infty}}(\frac{2\alpha}{\rho_{\infty}}-2\nabla\rho_{\infty}\nabla\theta)
  \end{equation}
where $\theta$ is a function about $x$ and is smooth outside $\Gamma_{\infty}$ in $B_1(0)$ and $|\nabla \theta|\leq C \rho_{\infty}^{-\delta}$
for any $\delta<1$. Also $\widetilde{\psi}_2$ is uniformly bounded outside the set
\begin{equation*}
U=\{(x,t)|\rho_{\infty}^{-1} > \sup\limits_{B_{1}
(0)}\nabla\rho_{\infty}\nabla\theta\}.
  \end{equation*}

 On the other hand, for any $(x,t)\in U$,
 \begin{equation*}
-\frac{\alpha}{\rho_{\infty}}(\frac{2\alpha}{\rho_{\infty}}-2\nabla\rho_{\infty}\nabla\theta)(x,t)< 0.
 \end{equation*}
 Applying the De Giorgi estimate, in view of Chapter 6 in \cite{C}, we can show
\begin{equation*}
\sup\limits_{Q_{\frac{1}{2}}(0)}|\widetilde{\psi}_1|\leq C(1+\int_{Q_{1}(0)}|\widetilde{\psi}_1|^2\ud x \ud t)^{\frac{1}{2}}\leq C.
 \end{equation*}
 Thus, $\widetilde{\psi}_1=\rho_{\infty}^{-\alpha}\psi_1$ is uniformly bounded and $\psi_1$ vanishes along $\Gamma_{\infty}\bigcap B_{1}(0)\times\{-\frac{1}{4}<t\leq 0\}$. $\partial_t\widetilde{\psi}_1=\rho_{\infty}^{-\alpha}\partial_t\psi_1$ is also a solution of \eqref{76}, so $\partial_t\widetilde{\psi}_1$ also has the estimate $|\partial_t\widetilde{\psi}_1|\leqslant C$. And $|\partial_t{\psi}_1|\leqslant C\rho_{\infty}^\alpha$ in $B_{1}(0)\times\{-\frac{1}{4}<t\leq 0\}$.

 Applying Lemma \ref{lem:L^2decay} and Lemma 4.4 in \cite{LT} to $\psi_1$ with $
\varepsilon=\max\{\frac{\alpha}{4},\frac{1}{4}\}$, we obtain $|\psi_1|\leq C\rho_{\infty}^{\min\{\frac{3\alpha}{2},\alpha+\frac{3}{2}\}}\leqslant C\rho_{\infty}^{\alpha+\frac{1}{2}}$.
Applying Lemma \ref{lem:Caccioppoli-inequality} with 0 as $\beta_1$ and $\varepsilon_i\psi_2(0,0)+\overline{\varphi}_2^i$ as $\beta_2$, for any $\tau\in (t_i-\frac{\sigma_i^2}{4l_0^2},t_i]$, we have
\begin{equation*}
\begin{aligned}
&E_{\lambda_0\sigma_i}(x_i,\tau)\\ \leq & C\left((2\lambda_0\sigma_i)^{-3}\int_{B_{2\lambda_0\sigma_i}(x_i)}h^{-2\alpha}e^{-2\varphi_2}| \varphi_1|^2 +|\varphi_2-\varepsilon_i\psi_2(0,0)-
\overline{\varphi}_2^i|^2\ud y(\tau)\right.\\
&\left.+2\lambda_0\sigma_i\int_{B_{2\lambda_0\sigma_i}(x_i)}h^{-2\alpha}e^{-2\varphi_2}|\partial_t \varphi_1|^2+|\partial_t
\varphi_2|^2\ud y(\tau)\right)\\
\leq&C\varepsilon_i^2(2\lambda_0)^{-3}\int_{B_{2\lambda_0}(0)}\left(|\psi^i_2(x,t)-\psi_2(0,0)|^2+\rho_{i}^{-2\alpha}| \psi^i_1(x,t)|^2 \right)\ud x(\frac{\tau-t_i}{\sigma_i^2})+2\lambda_0\varepsilon_i^2\\
\leq&C\varepsilon_i^2(2\lambda_0)^{-3}\int_{B_{2\lambda_0}(0)}\left(|\psi^i_2(x,t)-\psi_2(x,t)|^2+| \rho_{i}^{-\alpha}\psi^i_1(x,t)-\rho_{\infty}^{-\alpha}\psi_1(x,t)|^2 \right.\\
&\left.+|\psi_2(x,t)-\psi_2(0,0)|^2+|\rho_{\infty}
^{-\alpha}\psi_1(x,t)|^2\right) \ud x(\frac{\tau-t_i}{\sigma_i^2})+2\lambda_0\varepsilon_i^2\\
\leq &C_1\varepsilon_i^2(2\lambda_0)^{-3}\int_{B_{2\lambda_0}(0)}\left(|\psi^i_2(x,t)-\psi_2(x,t)|^2+| \rho_{i}
^{-\alpha}\psi^i_1(x,t)-\rho_{\infty}^{-\alpha}\psi_1(x,t)|
^2  \right)\ud x(\frac{\tau-t_i}{\sigma_i^2})\\
&+C_2(\lambda_0^2+2\lambda_0)\varepsilon_i^2+C_3\sup\limits_{B_{2\lambda_0}(0)} \rho_{\infty}
^{\frac{1}{2}}\varepsilon_i^2.
 \end{aligned}
  \end{equation*}
Since  $B_{\frac{\sigma_i}{l_0}}(x_i)\cap\Gamma\neq\varnothing$, $B_{\frac{1}{l_0}}(0)\cap \Gamma_{\infty}\neq\varnothing $ for all $i$ and $\sup\limits_{B_{2\lambda_0}(0)}\rho_{\infty}\leq \frac{2}{l_0}$. We can first take $l_0$ such that $C_3\sup\limits_{B_{2\lambda_0}(0)}\rho_{\infty}^{\frac{1}{2}}\leq C_3 (\frac{2}{l_0})^{\frac{1}{2}}<\frac{1}{8}$, that is $l_0>2(8C_3)^2$. Then choose $\lambda_0$ such that $2\lambda_0<\frac{1}{l_0}$ and $C_2(\lambda_0^2+2\lambda_0)<\frac{1}{8}$. Next, we take $i$ large enough such that $$C_1(2\lambda_0)^{-3}\int_{B_{2\lambda_0}(0)}\left(|\psi^i_2(x,t)-\psi_2(x,t)|^2+|\rho_{i}^{-\alpha}\psi^i_1(x,t)-\rho_{\infty}^{-\alpha}\psi_1(x,t)|^2 \right) \ud x(\frac{\tau-t_i}{\sigma_i^2})<\frac{1}{8}.$$
Therefore, we have
 \begin{equation*}
\begin{aligned}
E_{\lambda_0\sigma_i}(x_i,\tau)\leq \frac{1}{2}\varepsilon_i^2,  \forall \tau\in (t_i-\frac{\sigma_i^2}{4l_0^2},t_i].
 \end{aligned}
  \end{equation*}
 
 It contradicts to the inequality \eqref{eq:assume2}.
\end{proof}

We are now ready to introduce the main theorem of this section.
\begin{thm}\label{thm:holder}
There exists two uniform positive constants C and $\delta$ such that for any $x\in B_{\frac{1}{4}}(x_0)$, $t>0$ and $0<\sigma<\frac{1}{8}$,
\begin{align*}
E_{\sigma}(x,t)\leq C \sigma^{2\delta}.
\end{align*}

In particular, $\varphi_1$ is $\delta$-H\"{o}lder continuous.
\end{thm}
\begin{proof}

According to  Theorem 4.1 of \cite{LT}, Proposition \ref{prop:smallnessenergy} and Proposition \ref{prop:diedai1} in this paper, there are two uniform positive constants C and $\delta$ such that for any $x\in B_{\frac{1}{4}}(x_0)$, $t>0$, and $0<\sigma<\frac{1}{8}$,
\begin{align*}
E_{\sigma}(x,t)\leq C \sigma^{2\delta}.
\end{align*}

Define
\begin{align*}
\widetilde{E}_{\sigma}(x,t):=&\sigma^{-3}\int_{Q_{\sigma}(x,t)}h^{-2\alpha}e^{-2\varphi_2}|\nabla\varphi_1|^2+|\nabla\varphi_2|^2 \ud y\ud s\\
&+\sigma^{-1}\int_{Q_{\sigma}(x,t)}h^{-2\alpha}e^{-2\varphi_2}|\partial_t\varphi_1|^2+|\partial_t\varphi_2|^2\ud y\ud s.
\end{align*}
Then
\begin{align*}
\widetilde{E}_{\sigma}(x,t)=\sigma^{-2} \int_{t-\sigma^2}^{t} {E}_{\sigma}(x,t) \ud t\leq  C\sigma^{2\delta}.
\end{align*}

Combining Campanato's Lemma with Lemma \ref{lem:poincare}, it follows that $\varphi_2$ is $\delta$-H\"{o}lder continuous.
\end{proof}

\begin{lem}\label{lem:gradient1}
There is an $\varepsilon_0>0$, independent of $x\in B_{\frac{1}{2}}(x_0)$, $t>0$, such that
\begin{align*}
|\nabla\varphi_2|(x,t)\leq C\rho(x)^{-1+\varepsilon_0},  \ x \in B_{\frac{1}{2}}(x_0)\backslash \Gamma, \ t>0.
\end{align*}
\end{lem}
\begin{proof}
 For any fixed $t>0$, given any point $x$ in $B_{\frac{1}{2}}(x_0) \backslash \Gamma$, let $2\sigma=\rho(x)$. Then $\rho\geq\sigma$ in $B_{\sigma}(x)$. Let $G_z(y)$ be the Green function in $B_{\sigma}(x)$ with Dirichlet boundary condition and singularity at $z$.
According to the proof of Lemma 5.1 in \cite{LT}, for $\beta<\frac{3}{2}$, we can easily check that
\begin{equation}	
\begin{aligned}\label{4.20}
	&\int_{ B_{\sigma}(x)} |\nabla{\partial G_x}(y)|^\beta\ud V_y\leq C_\beta,\\
	&\int_{ \partial B_{\sigma}(x)} |\nabla \frac{G_x(y)}{\partial n_y}|\ud S_y\leq C,
\end{aligned}\end{equation}
where $C_\beta$ depends on $\beta$ and $n_y$ denotes the outward normal vector at $y$ on $\partial B_{\sigma}(x)$.

Using the second equation of \eqref{eq:equation}, we have 
 \begin{align*}
	\varphi_2(z,t)-\varphi_2(x,t)=&\int_{ \partial B_{\sigma}(x)} (\varphi_2(y,t)-\varphi_2(x,t)) \frac{\partial G_z(y)}{\partial n_y} \ud S_y\\
	&+\int_{ B_{\sigma}(x)} (\partial_t\varphi_2-h^{-2\alpha}e^{-2\varphi_2}|\nabla \varphi_1|^2)(y,t) G_z(y)\ud V_y.
\end{align*}
It follows
 \begin{align*}
|\nabla\varphi_2|(x,t)\leq &\int_{\partial B_{\sigma}(x)} |\varphi_2(y,t)-\varphi_2(x,t)||\nabla(\frac{\partial G_x}{\partial n_y})|(y)\ud S_y\\
&+C\int_{ B_{\sigma}(x)} (|\partial_t\varphi_2|+h^{-2\alpha}e^{-2\varphi_2}|\nabla \varphi_1|^2)(y,t)|\nabla G_x(y)|\ud V_y.
\end{align*}
By H\"older inequality, Theorem \ref{thm:holder} and (\ref{4.20}), we have
\begin{align*}
\int_{ B_{\sigma}(x)} h^{-2\alpha}e^{-2\varphi_2}|\nabla\varphi_1|^2|\nabla G_x(y)|\ud y\leq C \sigma^{-\frac{2}{\beta}+(1+2\delta)\frac{\beta-1}{\beta}}
\end{align*}
and
\begin{align*}
&\int_{ B_{\sigma}(x)} |\partial_t\varphi_2||\nabla G_x(y)|\ud y
%\\ \leq&(\int_{ B_{\sigma}(x)} |\partial_t\varphi_2|^{\frac{\beta}{\beta-1}}\ud y)^{\frac{\beta-1}{\beta}}(\int_{ B_{\sigma}(x)} |\nabla G_x(y)|^{\beta}   \ud y)^{\frac{1}{\beta}}\\
%\leq& C (\int_{ B_{\sigma}(x)} |\partial_t\varphi_2|^{2+\frac{2-\beta}{\beta-1}}\ud y)^{\frac{\beta-1}{\beta}}\\
\leq C\sigma^{\frac{3\beta-6}{2\beta}},
\end{align*}
where $\delta$ is as same as in Theorem \ref{thm:holder}. We can choose $\beta, \varepsilon_0$ such that they satisfy the conditions $\frac{6}{5}<\beta<\frac{3}{2}$, $\frac{3\beta-6}{2\beta}\geq-1+\varepsilon_0$, and $-\frac{2}{\beta}+(1+2\delta)\frac{\beta-1}{\beta}\geq-1+\varepsilon_0$. Given these inequalities, we can obtain
\begin{align*}
|\nabla\varphi_2|(x,t)\leq C\sigma^{-1+\varepsilon_0}\leq C \rho(x)^{-1+\varepsilon_0}.
\end{align*}
The lemma is completed.
%satisfies$\frac{6}{5}<\beta<\frac{3}{2}$,$2\beta-3+2\delta(\beta-1)=\beta\varepsilon_0>0$and $\frac{3\beta-6}{2\beta}=-1+\varepsilon_0$ for some $\varepsilon_0>0$, $\beta, \varepsilon_0$ depends only on $\delta$, $-\frac{2}{\beta}+(n-2+2\delta)\frac{\beta-1}{\beta}\geq-1+\varepsilon_0$.
\end{proof}

\begin{rem}\label{gkey}
	For any $x_0\in M\backslash \Gamma$ and $0<\varepsilon<\alpha$, if $B_1(x_0)\cap \varGamma\neq \varnothing$, let $v\in C^{2+\beta,1+\frac{\beta}{2}}(( B_1(0)\setminus \Gamma)\times [0,T])$ and
	$|\nabla v|\leqslant C^{'}\rho^{-\delta}$ with $\delta<1$, $f\in C^{\beta,\frac{\beta}{2}}(\overline{Q_T})$ with $\max\limits_{B_1(x)\times[0,T]}(\rho^{-2\alpha+2}|f|)\leqslant C^{'}$. Then  there is a uniform constant $C_{\epsilon}$ depending only on
	$C^{'},\beta,\epsilon$, such that if $u$ is a
	solution of the equation
	\begin{equation*}
		\left\{ \begin{aligned}
			\partial_t u-\Delta u&=2\nabla
			u(\frac{\alpha\nabla \rho}{\rho}+\nabla v)+f   &\mbox{ in} \ B_1(x_0)\times[0,T]\\
			|u|&\leqslant C \rho^{2\alpha+2} & \mbox{ on} \ B_1(x_0)\times \{t=0 \}
		\end{aligned}\right.
	\end{equation*}
	and $|u|_{C^0(\partial B_{\frac{2}{3}}(x_0)\times
		[0,T] )}\leqslant C^{'}, u|_{\Gamma\times [0,T]}
	\equiv 0$, then
	\begin{equation*}
		\begin{aligned}
			|u(x,t)|\leqslant C \max\limits_{B_1(x)\times[o,T]}(\rho^{-2\alpha+2}|f|)\rho(x)^{2\alpha-2\varepsilon}\leq C_{\varepsilon}\rho(x)^{2\alpha-2\varepsilon}
			 \enspace 
			\mbox{ in } B_1(x_0)\times[0,T].
		\end{aligned}
	\end{equation*}

	It's similar to the proof of Lemma \ref{key} by constructing the barrier and using maximum principle.
\end{rem}

\begin{cor}\label{cor:regularity}
 For any $0<\varepsilon<\alpha$, there is a uniform positive constant $C_{\varepsilon}$ such that
\begin{align*}
\|\varphi_j\|_{C^{k+\lambda,\frac{k+\lambda}{2}}(B_{\frac{1}{2}}(x_0)\times (0,T])}\leq C_{\varepsilon}, j=1,2,
\end{align*}
where $k=[2\alpha-2\varepsilon]$, and $\lambda=2\alpha-k-2\varepsilon$. In particular, for any $k\leq 2\alpha$, any $x\in B_{\frac{1}{2}}(x_0), t>0$, $|\nabla^k\varphi_2(x,t)|\leq C_{\varepsilon}
\rho(x)^{2\alpha-k-2\varepsilon} $.
\end{cor}
\begin{proof}

By Lemma \ref{lem:gradient1} and Remark \ref{gkey}, for any $0<\varepsilon<\alpha$, there is a uniform constant $C_{\varepsilon}>0$ such that
\begin{align}\label{eq:255}
|\varphi_1(x,t)|\leq C_{\varepsilon}\rho(x)^{2\alpha-2\varepsilon},  \  x\in B_{\frac{1}{2}}(x_0), t>0.
\end{align}
Therefore, $\varphi_1\in C^{k+\lambda,\frac{k+\lambda}{2}}(B_{\frac{1}{2}}(x_0)\times (0,T])$. In the first equation of \eqref{eq:equation}, $\varphi_2$ is H\"{o}lder continuous, and the term $h^{-2\alpha}e^{-2\varphi_2}|\nabla\varphi_1|^2$ is $([2\alpha-2-2\varepsilon],2\alpha-2-2\varepsilon-
[2\alpha-2-2\varepsilon] )$-H\"{o}lder continuous. Consequently, $\varphi_2\in C^{k+\lambda,\frac{k+\lambda}{2}}(B_{\frac{1}{2}}(x_0)\times (0,T])$.
\end{proof}

\begin{thm}\label{longtime}
	Let $\varphi^0=(\varphi^0_1,\varphi^0_2)\in \mathcal{A}_0\times \mathcal{B}_0$. There exists a vector-valued function $\varphi=(\varphi_1,\varphi_2) $ solving the equation (\ref{eq:equation}) in $M\times [0,+\infty)$.
	For any $0<\varepsilon<\min\{\frac{1}{2}, \alpha-1\}$, there is a uniform
	constant $C_{\varepsilon}>0$ such that
	\begin{align*}
		\|\varphi_j\|_{C^{k+\lambda,\frac{k+\lambda}{2}}(M\times [0,T])}\leq C_{\varepsilon}, j=1,2,
	\end{align*}
	where $k=[2\alpha-2\varepsilon]$, and $\lambda=2\alpha-k-2\varepsilon$. %In particular, for any $k\leq 5$, $|\nabla^k\varphi_2(x,t)|\leq C_{\varepsilon}\rho(x)^{6-k-2\varepsilon}$, for all $x \in M, t>0$.
	
\end{thm}

\begin{proof}
	Assume the longest time interval for the solution to  problem \eqref{eq:equation} is $[0, T)$. If $T<\infty$, we conclude that $\varphi^t(x)=\varphi(x,t)$ converge to $\varphi^T(x)$ as $t\rightarrow T$, allowing us to extend the solution $\varphi$ to the closed set $M\times[0,T]$.
	
	Taking $\varphi^T$ as the initial function, we can use the method of existence of short time solutions to extend the solution to $M\times [T, T+\delta]$ for some $\delta>0$. Note that $\varphi^T$ may not possess  the same high regularity as $\varphi^0$.However, according to the Lemma \ref{lem:2.4}, Remark \ref{rem1-1}, Remark \ref{rem2-10}, and  Remark \ref{rem:3-1}, we can prove the existence of a solution in $\{\varphi_1|\rho^{-\alpha+1}\varphi_1\in W^{2,1}_{2}{(Q_T)}, \varphi_1=0 \mbox{ on } \Gamma\}\times W^{2,1}_{2}{(Q_T)}$. 
	
	Since the original solution$(\varphi_1,\varphi_2)\in \{\varphi_1|\varphi_1\in C^{k+\lambda,\frac{k+\lambda}{2}}(Q_T), \varphi_1=0 \mbox{ on } \Gamma\}\times C^{k+\lambda,\frac{k+\lambda}{2}}(Q_T)$ (where $k=[2\alpha-2\varepsilon]$, and $\lambda=2\alpha-k-2\varepsilon$), and the new solution $(\varphi_1,\varphi_2)\in \{\varphi_1|\rho^{-\alpha+1}\varphi_1\in W^{2,1}_{2}{(M\times [T,T+\delta))}, \varphi_1=0 \mbox{ on } \Gamma\}\times W^{2,1}_{2}{(M\times [T,T+\delta))}$ and they agree on $M\times T$, they define a solution $(\varphi_1,\varphi_2)\in \{\varphi_1|\rho^{-\alpha+1}\varphi_1\in W^{2,1}_{2}{(M\times [0,T+\delta))}, \varphi_1=0 \mbox{ on } \Gamma\}\times W^{2,1}_{2}{(M\times [0,T+\delta))}$ with initial value $\varphi^0$. Therefore, we can prove that this solution has the same regularity as the solution in Corollary \ref{cor:regularity}. This is a contradiction. 
	
	Thus $T=\infty$, and the solutions satisfy the conclusion of Corollary \ref{cor:regularity}.
\end{proof}

\section{Convergence}\label{section5}
In this section, we will study the asymptotic behaviors of the solutions to equation (\ref{eq:equation}). We are interesting in the behavior as $t\rightarrow \infty$. Before presenting the final result, we will demonstrate a lemma.

\begin{lem}\label{lem:L2decay}
There exist positive constants $C$ and $C_0$ such that for every $0<t<t^{'}, x\in M\backslash \Gamma$, the following inequalities hold:
\begin{equation*}
\begin{aligned}
\int_M \rho^{-2\alpha}|\varphi_1(x,t)-\varphi_1(x,t^{'})|^2 +|\varphi_2(x,t)-\varphi_2(x,t^{'})|^2 \ud x\leq C e^{-\frac{C_0}{2}t} 
  \end{aligned}
  \end{equation*}
and
\begin{align*}
\rho^{\frac{3}{2}-\alpha} |\varphi_1(x,t)-\varphi_1(x,t^{'})|+\rho^{\frac{3}{2}}|\varphi_2(x,t)-\varphi_2(x,t^{'})|  \leq C e^{-\frac{C_0}{4}t}.
\end{align*}
\end{lem}
\begin{proof}
For any $t< t^{'}$, with the help of the Minkowski inequality and Lemma \ref{lem:L^2decay}, we have
\begin{equation*}
\begin{aligned}
(\int_M |\varphi_2(x,t)-\varphi_2(x,t^{'})|^2 \ud x)^{\frac{1}{2}}
&=(\int_M |\int_t^{t^{'}}\partial_t\varphi_2(x,s) \ud s|^2\ud x)^{\frac{1}{2}}
\\&\leq \int_t^{t^{'}}(\int_M|\partial_t\varphi_2(x,s)|^2 \ud x )^{\frac{1}{2}}\ud s
%\\&\leq \int_t^{t^{'}}e^{-\frac{C_0}{4}s}  \ud s
\\&\leq C e^{-\frac{C_0}{4}t}.
  \end{aligned}
  \end{equation*}
It follows that
\begin{equation}\label{eq:convengence1}
\begin{aligned}
\int_M |\varphi_2(x,t)-\varphi_2(x,t^{'})|^2 \ud x\leq C e^{-\frac{C_0}{2}t}.
  \end{aligned}
  \end{equation}
For $\varphi_1$, we similarly have
\begin{equation}\label{eq:convengence2}
\begin{aligned}
\int_M \rho^{-2\alpha}|\varphi_1(x,t)-\varphi_1(x,t^{'})|^2 \ud x\leq C e^{-\frac{C_0}{2}t}.
  \end{aligned}
  \end{equation}
Using the Minkowski inequality and Lemma \ref{lem:L^2decay} again, we get
\begin{align*}
&\rho^{\frac{3}{2}-\alpha} |\varphi_1(x,t)-\varphi_1(x,t^{'})|+\rho^{\frac{3}{2}}|\varphi_2(x,t)-\varphi_2(x,t^{'})| \\
  =&|\rho^{\frac{3}{2}-\alpha}\int_t^{t^{'}}\partial_t\varphi_1(x,s) \ud s|+|\rho^{\frac{3}{2}}\int_t^{t^{'}}\partial_t\varphi_2(x,s) \ud s|\\
%\leq& C\int_t^{t^{'}}e^{-\frac{C_0}{4}s} \ud s 
\leq& C e^{-\frac{C_0}{4}t}.
\end{align*}
  \end{proof}

With all of the preparatory work now completed, we are ready to establish the limit points for the heat flow upon approach to the prescribed singularity. 

Recall that $$\| w \|_{C_{*}^{2}(M)}(t):= \sum_{k=0}^{2} \max_M
		(|\nabla^k w_1| \rho^{k+\frac{3}{2}-\alpha}+|\nabla^kw_2| \rho^{k+\frac{3}{2}})(t).$$
\begin{thm}\label{thm:C2convengence}
Assume $\varphi=(\varphi_1,\varphi_2)$ is a
solution of \eqref{eq:equation}, let $\varphi(s)=\varphi(\cdot, s)$ for brevity, then 
$$\rho^{\frac{3}{2}-\alpha}\partial_t\varphi_1( s)\rightarrow 0, \enspace \rho^{\frac{3}{2}}\partial_t\varphi_2(s)\rightarrow 0, \enspace as \enspace s\rightarrow +\infty,$$   
$$\varphi(s)\rightarrow\overline {\varphi}\enspace in \enspace C_{*}^{2}(M),\enspace as \enspace s\rightarrow +\infty,$$
 and
\begin{align*}
\|\varphi(s)-\overline{\varphi}\|_{C_{*}^{2}(M)} \leq C e^{-\frac{C_0}{4}s},
\end{align*}
where $\overline{\varphi}$ satisfies
\begin{equation*}
\left\{\begin{aligned}
	\Delta\overline{\varphi}_1-2(\nabla\overline{\varphi}_2+\frac{\alpha\nabla h}{h})\nabla\overline{\varphi}_1&=0  &\mbox{in } Q_T\\
\Delta\overline{\varphi}_2+h^{-2\alpha}e^{-2\overline{\varphi}_2}|\nabla\overline{\varphi}_1|^2&=0  &\mbox{in } Q_T.\\
  \end{aligned}\right.
  \end{equation*}
Also for any $0<\varepsilon<2\alpha$, $(\overline{\varphi_1},\overline{\varphi_2})\in C^{k,\lambda}(M)$, where $k=[2\alpha-2\varepsilon], \lambda=2\alpha-2\varepsilon-k$.
\end{thm}
\begin{proof}

Fixing $\widetilde{t}>0$, define $\widetilde{\varphi}_i(x,t)=\varphi_i(x,t+\widetilde{t} )$, i=1,2.
We can easily find that $(\widetilde{\varphi}_1, \widetilde{\varphi}_2)$ satisfy
\begin{equation}\label{t+t}
\left\{
\begin{aligned}
	 \partial_t\widetilde{\varphi}_1-\Delta\widetilde{\varphi}_1+2(\nabla\widetilde{\varphi}_2+\frac{\alpha\nabla h}{h})\nabla\widetilde{\varphi}_1 &=0 \ \mbox{ in } M\\
\partial_t\widetilde{\varphi}_2-\Delta\widetilde{\varphi}_2-h^{-2\alpha}e^{-2\widetilde{\varphi}_2}|\nabla\widetilde{\varphi}_1|^2&=0 \ \mbox{ in } M.\\
  \end{aligned}
  \right. 
  \end{equation}
Set $\omega_1=\varphi_1-\widetilde{\varphi}_1, \omega_2=\varphi_2-\widetilde{\varphi}_2$. A straightforward computation gives
\begin{align*}
&h^{-2\alpha}e^{-2\varphi_2}|\nabla \varphi_1|^2-h^{-2\alpha}e^{-2\widetilde{\varphi}_2}|\nabla\widetilde{\varphi}_1|^2
\\=&h^{-2\alpha}e^{-2\varphi_2}\nabla (\varphi_1+\widetilde{\varphi}_1)\nabla \omega_1+h^{-2\alpha}e^{-2\widetilde{\varphi}_2}|\nabla\widetilde{\varphi}_1|^2(e^{-2\omega_2}-1),\\
&\nabla\varphi_1\nabla\varphi_2-\nabla\widetilde{\varphi}_1\nabla\widetilde{\varphi}_2
=\nabla\varphi_1\nabla \omega_2+\nabla\widetilde{\varphi}_2\nabla \omega_1.
\end{align*}
Then, by calculating the result of \eqref{eq:equation} minus \eqref{t+t}, we obtain
\begin{equation}
\left\{
\begin{aligned}
\partial_t\omega_1-&\Delta \omega_1+2(\nabla \widetilde{\varphi}_2+\frac{\alpha\nabla h}{h})\nabla\omega_1+2\nabla{\varphi}_1\nabla\omega_2=0  &\mbox{ in } Q_T\\
\partial_t\omega_2-&\Delta \omega_2 -h^{-2\alpha}e^{-2\varphi_2}\nabla (\varphi_1+\widetilde{\varphi}_1)\nabla \omega_1\\-&h^{-2\alpha}e^{-2\widetilde{\varphi}_2}|\nabla\widetilde{\varphi}_1|^2(e^{-2\omega_2}-1)   =0   &\mbox{ in }Q_T.                                                                                                                                                                                             
  \end{aligned}
  \right. \label{eq:equation3}
  \end{equation}

We note that the term involving $e^{-2w_2}-1$ in the second equation can be viewed as a linear term in $w_2$.  The system \eqref{eq:equation3} can be rewritten as
\begin{equation}
\left\{
\begin{aligned}
&\partial_t\omega_1-\Delta \omega_1=\xi_1 \mbox{ in } Q_T\\
 &\partial_t\omega_2-\Delta \omega_2=\xi_2 \mbox{ in } Q_T,
  \end{aligned}
  \right. \label{eq:equation4}
  \end{equation}
where
\begin{equation*}
\left\{
\begin{aligned}
	&\xi_1=b_{11}\nabla \omega_1+b_{12}\nabla \omega_2\\
&\xi_2=b_{21}\nabla \omega_1+b_{22}(e^{-2\omega_2}-1) \\
  \end{aligned}
  \right.
  \end{equation*}
and 
\begin{align*}
	&b_{11}=-2(\nabla \widetilde{\varphi}_2+\frac{\alpha\nabla h}{h}),\quad
	b_{12}=-2\nabla{\varphi}_1,\\
&b_{21}=h^{-2\alpha}e^{-2\varphi_2}\nabla (\varphi_1+\widetilde{\varphi}_1), b_{22}=h^{-2\alpha}e^{-2\widetilde{\varphi}_2}|\nabla\widetilde{\varphi}_1|^2.
\end{align*}
By Corollary \ref{cor:regularity}, for any $0<\varepsilon<2\alpha$, $|b_{11}(x,t)|\leq
C\rho(x)^{-1}$, $|b_{12}(x,t)|\leq
C\rho(x)^{2\alpha-1-2\varepsilon}$, $|b_{21}(x,t)|\leq
C\rho(x)^{-1-2\varepsilon}$, $|b_{22}(x,t)|\leq
C\rho(x)^{2\alpha-2-4\varepsilon}$. 
%The first equation of \eqref{eq:equation4} is equivalent to
%\begin{align*}
%\partial_tw_1-h^{2\alpha} e^{2\widetilde{\varphi}_2}\mbox{div}(h^{-2\alpha}e^{-2\widetilde{\varphi}_2}\nabla w_1)-b_{12}\nabla w_2=0,\mbox{ in } Q_T.
%\end{align*}
We conclude from Lemma \ref{lem:L2decay} that
\begin{equation*}\label{eq:convengence}
\begin{aligned}
\int_M \rho^{-2\alpha}|\omega_1(x,t)|^2 +|\omega_2(x,t)|^2\ud x\leq C e^{-\frac{C_0}{2}t}.
  \end{aligned}
  \end{equation*}
And Lemma \ref{lem:L^2decay} yields
\begin{equation}\label{eq:convengence3}
\begin{aligned}
\int_M \rho^{-2\alpha}|\partial_t \omega_1(x,t)|^2 +|\partial_t \omega_2(x,t)|^2\ud x\leq C e^{-\frac{C_0}{2}t}.
  \end{aligned}
  \end{equation}
Taking  $h^{-2\alpha}e^{-2\widetilde{\varphi}_2}\omega_1$ and $\omega_2$ as test functions for equations of \eqref{eq:equation4} respectively and using Cauchy inequality, we can obtain
\begin{equation*}
	\begin{aligned}
		\int_M h^{-2\alpha}e^{-2\widetilde{\varphi}_2}|\nabla \omega_1|^2\ud x
		&=\int_M-h^{-2\alpha}e^{-2\widetilde{\varphi}_2}\partial_t \omega_1 \omega_1+b_{12}\nabla\omega_2h^{-2\alpha}e^{-2\widetilde{\varphi}_2} \omega_1\\
		&\leq C_{\varepsilon_1}\int_M h^{-2\alpha}e^{-2\widetilde{\varphi}_2}(|\partial_t \omega_1|^2+| \omega_1|^2) \ud x+\varepsilon_1\int_M  |\nabla \omega_2|^2 \ud x,\\
		\int_M |\nabla \omega_2|^2\ud x
		&=\int_M -\partial_t \omega_2\omega_2+b_{21} \nabla \omega_1 \omega_2+b_{22}\omega_2(e^{-2\omega_2}-1)\ud x\\
		&\leq C_{\varepsilon_2}\int_M |\partial_t \omega_2|^2+| \omega_2|^2 \ud x+\varepsilon_2 \int_M h^{-2\alpha}e^{-2w_2} |\nabla w_1|^2\ud x.\\
	\end{aligned}
\end{equation*}
Select $\varepsilon_1$ and $\varepsilon_2$ small enough to obtain
\begin{equation}\label{eq:gradientestimates}
	\begin{aligned}
		&\int_M h^{-2\alpha}e^{-2\widetilde{\varphi}_2}|\nabla w_1|^2+|\nabla w_2|^2\ud x\\
		\leq &C \int_M (|\partial_tw_2|^2+| w_2|^2 )+h^{-2\alpha}e^{-2\widetilde{\varphi}_2}(|\partial_t w_1|^2+| w_1|^2)\ud x\\
		\leq &C e^{-\frac{C_0}{2}t}.
	\end{aligned}
\end{equation}

Fix any $ X=(x,s)\in M \backslash \Gamma \times (1,+\infty)$ and set $\overline{\rho}=\rho(x)$. Define
$Q_{\frac{\overline{\rho}}{2}}(X)=B_{\frac{\overline{\rho}}{2}}(x) \times (s-\frac{\overline{\rho}^2}{4},s)$, $Q_{\rho}= B_{\rho} \times (-1,0)$, where $B_{\rho}$ is the ball centered at $x=0$ with radius $\rho$. For any $(y,t)\in Q_{\frac{\overline{\rho}}{2}}(X)$, consider the transformation
\begin{equation*}
\begin{aligned}
y=x+\frac{\overline{\rho}}{2}z,\enspace t=s+\frac{\overline{\rho}^2}{4}r,
  \end{aligned}
  \end{equation*}
then $(z,r)\in Q_{1}$. Now we define
\begin{equation*}
\begin{aligned}
v_1(z,r)= \omega_1(x+\frac{\overline{\rho}}{2}z,s+\frac{\overline{\rho}^2}{4}r), \\
v_2(z,r)= \omega_2(x+\frac{\overline{\rho}}{2}z,s+
\frac{\overline{\rho}^2}{4}r),
  \end{aligned}
  \end{equation*}
then
 \begin{equation*}
  \left\{
\begin{aligned}
&\partial_rv_1-\Delta v_1=\widetilde{\xi}_1 \ \mbox{ in } Q_1\\
 &\partial_rv_2-\Delta v_2=\widetilde{\xi}_2 \ \mbox{ in } Q_1,\\
  \end{aligned}
  \right.
  \end{equation*}
where $\widetilde{\xi}_i=\frac{\overline{\rho}^2}{4}\xi_i, 
i=1,2$.
For any $r\in(-1,0)$ and $j=1,2$, the regularity theory of elliptic equations leads to
\begin{align*}
\|v_j\|^2_{H^2(B_{\frac{1}{2}})}(r)
&\leq C(\|v_j\|^2_{L^2(B_{1})}+\|\widetilde{\xi}_j\|^2_{L^2(B_{1})}+\|\partial_rv_j\|^2_{L^2(B_{1})})(r)\\
&=C\int_{B_{\frac{\overline{\rho}}{2}}(x)}(\omega_j^2+\frac{\overline{\rho}^4}{16} \xi_j^2+|\partial_t\omega_j|^2\frac{\overline{\rho}^4}{16})\overline{\rho}^{-3}\ud y(t), 
\end{align*}
where
\begin{equation}
\begin{aligned}\label{5.8}
	\overline{\rho}^4 \xi_1^2\leq C(\rho^{2}|\nabla \omega_1|^2+\rho^{4\alpha+2-4\varepsilon}| \nabla \omega_2|^2),\\
\overline{\rho}^4 \xi_2^2\leq C(\rho^{2-4\varepsilon}|\nabla \omega_1|^2+\rho^{4\alpha-8\varepsilon}| \omega_2|^2).
\end{aligned}\end{equation}
According to Lemma \ref{lem:L^2decay},
 \eqref{eq:convengence3}, and
\eqref{eq:gradientestimates}, we have
\begin{equation*}\label{eq:vtL2}
\begin{aligned}
\|\partial_rv_1\|_{L^2(B_{\frac{1}{2}})}(r)\leq C\overline{\rho}^{\alpha+\frac{1}{2}} e^{-\frac{C_0}{4}t},\\
\|\partial_rv_2\|_{L^2(B_{\frac{1}{2}})}(r)\leq C\overline{\rho}^{\frac{1}{2}} e^{-\frac{C_0}{4}t},
\end{aligned}
  \end{equation*}
and
  \begin{equation}\label{eq:vH2}
\begin{aligned}
\|v_1\|_{H^2(B_{\frac{1}{2}})}(r)\leq C \overline{\rho}^{\alpha-\frac{3}{2}}e^{-\frac{C_0}{4}t}, \forall r\in(-1,0), t\in (s-\frac{\overline{\rho}^2}
{4},s),\\
\|v_2\|_{H^2(B_{\frac{1}{2}})}(r)\leq C \overline{\rho}^{-\frac{3}{2}}e^{-\frac{C_0}{4}t},\forall r\in(-1,0), t\in (s-\frac{\overline{\rho}^2}
{4},s).
\end{aligned}
  \end{equation}
After integrating over the appropriate time interval, we conclude that
\begin{equation}\label{eq:v2W212}
	\begin{aligned}
		\|v_1\|_{W^{2,1}_2(Q_{\frac{1}{2}})} \leq C \overline{\rho}^{\alpha-\frac{3}{2}} e^{-\frac{C_0}{4}s},
	\end{aligned}
\end{equation}
\begin{equation}\label{eq:v1W212}
\begin{aligned}
\|v_2\|_{W^{2,1}_2(Q_{\frac{1}{2}})}
&\leq C \overline{\rho}^{-\frac{3}{2}}e^{-\frac{C_0}{4}s}.
\end{aligned}
  \end{equation}
By the Sobolev inequality and \eqref{eq:vH2}, for any $r\in(-1,0)$, $t\in (s-\frac{\overline{\rho}^2}{4},s)$,
\begin{align*}
&\| v_1\|_{L^6(B_{\frac{1}{2}})}(r)+\|\nabla v_1\|_{L^6(B_{\frac{1}{2}})}(r)\leq C\|v_1\|_{H^2(B_{\frac{1}{2}})}(r)\leq C \overline{\rho}^{\alpha-\frac{3}{2}}e^{-\frac{C_0}{4}
t}.
\end{align*}
Similarly,
\begin{align*}
\| v_2\|_{L^6(B_{\frac{1}{2}})}(r)+\|\nabla v_2\|_{L^6(B_{\frac{1}{2}})}(r)\leq C \overline{\rho}^{-\frac{3}{2}}e^{-\frac{C_0}{4}
t}.
\end{align*}
Integrating  $r$ from $-\frac{1}{2}$ to 0, we have
\begin{align*}
\| v_1\|_{L^6(Q_{\frac{1}{2}})}+\|\nabla v_1\|_{L^6(Q_{\frac{1}{2}})}
&\leq  C \overline{\rho}^{\alpha-\frac{3}{2}}  (\int_{s-\frac{\overline{\rho}^2}{8}}^s  e^{-\frac{3C_0}{2}t}  \overline{\rho}^{-2}  \ud t )^{\frac{1}{6}}
%&\leq  C \overline{\rho}^{-\frac{3}{2}}  [\frac{2}{3C_0}e^{-\frac{3C_0}{2}s}(e^{\frac{3C_0}{2}(\frac{\overline{\rho}}{4})^2}-1)]^{\frac{1}{6}}\\
\leq  C \overline{\rho}^{\alpha-\frac{3}{2}} e^{-\frac{C_0}{4}s},
\end{align*}
\begin{align*}
\| v_2\|_{L^6(Q_{\frac{1}{2}})}+\|\nabla v_2\|_{L^6(Q_{\frac{1}{2}})}\leq  C \overline{\rho}^{-\frac{3}{2}}  e^{-\frac{C_0}{4}s}.
\end{align*}
The interior estimate in $W_p^{2,1}$ space of parabolic equations produce
\begin{align*}
\|v_{j}\|_{W^{2,1}_6(Q_{\frac{1}{4}})}\leq C \left(\|v_{j}\|_{W^{2,1}_2(Q_{\frac{1}{2}})}+\|\widetilde{\xi}_{j}\|_{L^6(Q_{\frac{1}{2}})}\right),
\enspace j=1,2.\end{align*}
Therefore, by (\ref{5.8}), (\ref{eq:v2W212}), and (\ref{eq:v1W212})
we obtain  \begin{align*}\|v_{1}\|_{W^{2,1}_6(Q_{\frac{1}{4}})}\leq C\overline{\rho}^{\alpha-\frac{3}{2}}  e^{-\frac{C_0}{4}s},\enspace \|v_{2}\|_{W^{2,1}_6(Q_{\frac{1}{4}}(0))}\leq C
\overline{\rho}^{-\frac{3}{2}}  e^{-\frac{C_0}{4} s}.
\end{align*}
By the embedding theorem in $W_p^{2,1}$ space (see Chapter 3 in \cite{C}), we obtain 
\begin{align*}
\|\nabla v_{1}\|_{C^{m,\frac{m}{2}}(Q_{\frac{1}{4}})}\leq C\|v_{1}\|_{W^{2,1}_6(Q_{\frac{1}{4}})}\leq C \overline{\rho}^{\alpha-
\frac{3}{2}}  e^{-\frac{C_0}{4}s},\\
\|\nabla v_{2}\|_{C^{m,\frac{m}{2}}(Q_{\frac{1}{4}})}\leq C\|v_{2}\|_{W^{2,1}_6(Q_{\frac{1}{4}})}\leq C \overline{\rho}^{-\frac{3}{2}}  e^{-\frac{C_0}{4}s},
\end{align*}
where $m=\frac{1}{6}$.
Therefore,
\begin{equation}\label{eq:fholder}
\begin{aligned}
\|\widetilde{\xi}_{1}\|_{C^{m,\frac{m}{2}}(Q_{\frac{1}{4}})}&\leq C \overline{\rho}^{\alpha-\frac{3}{2}}  e^{-\frac{C_0}{4}s},
\\
\|\widetilde{\xi}_{2}\|_{C^{m,\frac{m}{2}}(Q_{\frac{1}{4}})}&\leq C \overline{\rho}^{-2\varepsilon+\alpha-\frac{3}{2}}  e^{-\frac{C_0}{4}s}.
\end{aligned}
  \end{equation}
We now apply Schauder theory of parabolic equations, to find
\begin{align*}
\|v_{j}\|_{C^{2+m,1+\frac{m}{2}}(Q_{\frac{1}{8}})}\leq C \left(\|v_{j}\|_{W^{2,1}_2(Q_{\frac{1}{4}})}+\|\widetilde{\xi}_{j}\|_{C^{m,\frac{m}{2}}(Q_{\frac{1}{4}})}\right), \enspace j=1,2.
 \end{align*}
By  \eqref{eq:v2W212}, \eqref{eq:v1W212} and \eqref{eq:fholder}, we have
\begin{align*}
\|v_{1}\|_{C^{2+m,1+\frac{m}{2}}(Q_{\frac{1}{8}})}\leq C \overline{\rho}^{\alpha-\frac{3}{2}}  e^{-\frac{C_0}{4}s},\\
\| v_{2}\|_{C^{2+m,1+\frac{m}{2}}(Q_{\frac{1}{8}})}\leq C \overline{\rho}^{-\frac{3}{2}}  e^{-\frac{C_0}{4}s}.
\end{align*}
For $k=0,1,2$, observe that
\begin{align*}
\nabla^k_z v_j(z,r)=\nabla^k_y w_j(y,t)(\frac{\overline{\rho}}{2})^k, \enspace \partial_rv_j(z,r)= \partial_t w_j(y,t)
(\frac{\overline{\rho}}{2})^2, \enspace j=1,2.
\end{align*}
%$|\partial_t w_1(y,t) \rho^{2+\frac{3}{2}}(y)+|\partial_t w_1(y,t) \rho^{2-\frac{3}{2}}(y)|$
In conclusion, for any $s>1$, $\| w \|_{C_{\ast }^{2+\alpha}(M)}(s)\leq Ce^{-\frac{C_0}{4}s}$, that is $\|\varphi(s)-\varphi(s+\widetilde{t})\|_{C_{*}^{2}(M)}(s)\leq C  e^{-\frac{C_0}{4}s}$. Hence, $\varphi(s)$ converges in $C_{*}^{2+\alpha}(M)$ as
$s \rightarrow+\infty                                                                                                                                                                                                                                                                                                                                                                                                                                                                                                                                                                                                                                                                                 $, say to some $\overline{\varphi}\in C_{*}^{2}(M)$. Letting $\widetilde{t}\rightarrow +\infty$, $\|\varphi(s)-\overline{\varphi}\|_{C_{*}^{2}(M)}(s)\leq C e^{-\frac{C_0}{4}s}.$

Multiplying the first equation of \eqref{eq:equation} by $\rho^{-\alpha+\frac{7}{2}}$ and multiplying the second equation of \eqref{eq:equation} by $\rho^{\frac{7}{2}}$, we get
\begin{equation*}
	\left\{
\begin{aligned}
\rho^{\frac{7}{2}-\alpha}\partial_t\varphi_1-\rho^{\frac{7}{2}-\alpha}\Delta\varphi_1+2\rho^{\frac{7}{2}-\alpha}(\nabla \varphi_1+\frac{\alpha\nabla h}{h})\nabla \varphi_1&=0 &\mbox{ in } Q_T\\
\rho^{\frac{7}{2}}\partial_t\varphi_2-\rho^{\frac{7}{2}}\Delta\varphi_2-\rho^{\frac{7}{2}}h^{-2\alpha}e^{-2\varphi_2}|\nabla \varphi_1|^2&=0 &
\mbox{ in } Q_T,\\
  \end{aligned}\right.
  \end{equation*}
As $t\rightarrow +\infty$, by Lemma \ref{lem:L^2decay}, it holds that $\rho^{\frac{7}{2}-\alpha}
\partial_t\varphi_1(\cdot, s)\rightarrow 0$ and $\rho^{\frac{7}{2}}\partial_t\varphi_2(\cdot, s)\rightarrow 0$. Thus $\overline{\varphi}$ satisfies
\begin{equation}\label{L}
\left\{\begin{aligned}
\rho^{\frac{7}{2}}\Delta\overline{\varphi}_1+\rho^{\frac{7}{2}}h^{-2\alpha}e^{-2\overline{\varphi}_1}|\nabla \overline{\varphi}_2|^2&=0 &\mbox{ in }
Q_T,\\
\rho^{\frac{7}{2}-\alpha}\Delta\overline{\varphi}_2-2\rho^{\frac{7}{2}-\alpha}(\nabla \overline{\varphi}_1+\frac{\alpha\nabla h}{h})\nabla \overline{\varphi}_2&=0 &\mbox{ in } Q_T.
  \end{aligned}\right.
\end{equation}
By multiplying the first equation of \eqref{L} by $\rho^{\alpha-\frac{7}{2}}$ and the second equation of \eqref{L} by $\rho^{-\frac{7}{2}}$, we get
\begin{equation*}\left\{
\begin{aligned}
	\Delta\overline{\varphi}_1-2(\nabla
	\overline{\varphi}_2+\frac{\alpha\nabla h}{h})\nabla\overline{\varphi}_1&=0 &\mbox{ in } Q_T\\
\Delta\overline{\varphi}_2+h^{-2\alpha}e^{-2\overline{\varphi}_2}|\nabla\overline{\varphi}_1|^2&=0 &\mbox{ in } Q_T.\\
  \end{aligned}\right.
  \end{equation*}
By Theorem 1.1 of \cite{LT}, for any
$0<\varepsilon<2\alpha$, $(\overline{\varphi_1},\overline{\varphi_2})\in C^{k,\lambda}(M)$, where $k=[2\alpha-2\varepsilon],\lambda=2\alpha-2\varepsilon-k$.
\end{proof}

\setlength{\parindent}{0pt}
\textbf{Acknowledgements}
We thank Professor Jingang Xiong for suggesting the problem and his constant support.

\textbf{Statements and Declarations}
No funding was received to assist with the preparation of this manuscript. The authors declare that they have no conflicts of interests.

%\bibliographystyle{plain}
%\bibliography{HF-bibliography}

\begin{thebibliography}{99}
\bibitem{A} Adams, D., Simon, L.: Rates of asymptotic convergence near isolated singularities of geometric extrema. Indiana Univ. Math. J. \textbf{37}(2), 225-254 (1988)

\bibitem{AA}Allard, W., Almgren, F.: On the radial behavior of minimal surfaces and the uniqueness of their tangent cones. Ann. Math. \textbf{113}(2), 215-265 (1981)
	
\bibitem{P}Andrew, P.: Elementary Differential Geometry. Springer-Verlag London Limited 2010, Corrected printing, (2012)	
	
\bibitem{BFi}Bonforte, M., Figalli, A.: The Cauchy-Dirichlet Problem for the fast diffusion equation on bounded domains. Nonlinear Anal. {\bf 239}, Paper No. 113394, 55 pp. (2024)

	\bibitem{C} Chen, Y.: Parabolic Partial Differential Equation of Second Order. Peking University Press, Beijing, (2003)
		\bibitem{CYS}Chen, Y., Struwe, M.: Existence and Partial Regularity Results for the Heat Flow for Harmonic Maps. Math.Z. \textbf{201}(1), 83-103 (1989)
		
	\bibitem{CS} Chiarenza, F., Serapioni, R.: Degenerate parabolic equations and Harnack inequality. Ann. Mat. Pura Appl. \textbf{137}(4), 139-162 (1984)
	
\bibitem{RC} Coleman, R.: Calculus on Normed Vector Spaces. Springer, New York, (2012)

	\bibitem{ES} Eells, J., Sampson, J.: Harmonic Mappings of Riemannian Manifolds. Amer. J. Math. \textbf{86}, 109-160 (1964)
	
\bibitem{LC} Evans, L: Partial Differential Equations second edition. Wadsworth and Brooks/cole Mathematics, \textbf{19}(1), (2010)
\bibitem{GF}Fioravanti, G.:The Dirichlet problem on lower dimensional boundaries: Schauder estimates via perforated domains.arXiv:2412.11294 [math.AP](2024)
	\bibitem{GT} Gilbarg, D., Trudinger, N.: Elliptic partial differential equations of second order. 2nd. Springer, Berlin, (1983)
	
	\bibitem{H} Hamilton, R.: Harmonic maps of manifolds with boundary. Lecture Notes in Mathematics, \textbf{471}, Springer-Verlag, Berlin-New York, (1975)	
	
	\bibitem{HX} Han, Q., Khuri, M., Weinstein, G. and Xiong, J.: Asymptotic Analysis of Harmonic Maps with prescribed singularities.	arXiv:2212.14826 [math.DG](2022)
	
	\bibitem{MJ} Jendoubi, M.: A simple unified approach to some convergence theorems of L. Simon. J. Funct. Anal. \textbf{153}(1), 187--202 (1998)
	
\bibitem{LM} Lemm, M., Markovi\'c,V. :  Heat flows on hyperbolic spaces. J. Differential Geom. {\bf 108}(3), 495--529 (2018)

\bibitem{LT1} Li, P., Tam, L.:The heat equation and harmonic maps of complete manifolds. Invent. Math. {\bf 105}, 1--46 (1991)

\bibitem{LT3} Li, P., Tam, L.:Uniqueness and regularity of proper harmonic maps.  Ann. of Math. (2) {\bf 137}(1), 167--201 (1993)

\bibitem{LT4}Li, P., Tam, L.:Uniqueness and regularity of proper harmonic maps II. Indiana Univ. Math. J. {\bf 42}(2), 591--635 (1993)

 \bibitem{LW2} Li, P., Wang, J.:Harmonic rough isometries into Hadamard space. Asian J. Math. {\bf 2}(3), 419--422 (1998)
 
 \bibitem{LT} Li, Y., Tian, G.: Regularity of Harmonic Maps with Prescribed Singularities. Commun. Math. Phys. \textbf{149}, 1--30 (1992)
 
\bibitem{LT2} Liao, G., Tam, L.:On the heat equation for harmonic maps from non-compact manifolds. Paciﬁc J. Math. {\bf 153}(1), 129--145 (1992)

\bibitem{LGM} Lieberman, G.: Second Order Parabolic Differential Equations. World Scientific, (1996)
	%the interpolation inequality 	
\bibitem{LW} Lin, F., Wang, C.: The Analysis of Harmonic Maps and their Heat Flows. World Scientific Publishing Co. Pte. Ltd., Hackensack, NJ (2008)
\bibitem{M} Markovi\'c, V.: Harmonic maps between 3-dimensional hyperbolic spaces. Invent. Math. {\bf 199}(3), 921--951 (2015)

\bibitem{S}Schoen, R.:The role of harmonic mappings in rigidity and deformation problems, in  Complex geometry (Osaka, 1990).  Lecture Notes in Pure and Appl. Math. {\bf 143}, 179--200 (1990)
	

	%\bibitem{BF} Fabrice Bethuel,\emph{On the singular set of stationary harmonic maps.}
	%Manuscripta Math.78,417--443,1993.
				
	\bibitem{LS1}Simon, L.: Asymptotics for a class of non-linear evolution equations, with applications to geometric problems. Ann. Math. \textbf{118}, 525--571 (1983)
	
	\bibitem{LS2}Simon, L.: Theorems on Regularity and Singularity of Energy Minimizing Maps. Lectures in Mathematics, ETH Z\"uric, Birkh\"auser, Basel, (1996)
%\bibitem{SU}Schoen, R., Uhlenbeck, K.: A regularity theory for harmonic maps. J.Differ.Geom., \textbf{17}(2), 307-335(1982)
\bibitem{Mlfd} Tu. L.: An introduction to Manifolds, Second Edition. Springer New York, NY, (2011)
	
%\bibitem{SU2}Schoen, R., Uhlenbeck, K.: Boundary regularity and the Dirichlet problem for harmonic maps. J. Differ. Geom., \textbf{18}(2), 253--268(1983)
	\bibitem{W} Wang, J.:The heat ﬂow and harmonic maps between complete manifolds. J. Geom. Analysis. {\bf 8}(3), 485--514 (1998)

\bibitem{G3} Weinstein, G.: On rotating black holes in equilibrium in general relativity. Ph.D. thesis, at Courant Institute, New York University (1989)		
	
\bibitem{G} Weinstein, G.: On the Dirichlet problem for harmonic maps with prescribed singularities. Duke Math. J. \textbf{77}(1), 135--165 (1995)	
	
\bibitem{G1}Weinstein, G.: Harmonic maps with prescribed singularities into Hadamard manifolds. Math. Res. Lett. \textbf{3}(6), 835--844 (1996)
	
	\bibitem{G2}Weinstein, G.: Harmonic maps with prescribed singularities and applications in general relativity, in the role of metrics in the theory of partial differential equations. Adv. Stud. Pure Math. {\bf 85}, Math. Soc. Japan, Tokyo, 479--489 (2020)

\end{thebibliography}

\bigskip
\end{document}